\documentclass[10pt]{article}
\usepackage{amsmath,amsthm,amsfonts,amssymb}
\usepackage[usenames]{color}
\usepackage[hypertex]{hyperref}
\usepackage{graphicx}

\title{Groups acting on hyperbolic $\Lambda$-metric spaces}
\author{\textsf{Andrei-Paul Grecianu} \and
\textsf{Alexei Kvaschuk} \and
\textsf{Alexei Myasnikov} \and
\textsf{Denis Serbin}}

\date{}

\def\N{{\mathbb{N}}}
\def\Z{{\mathbb{Z}}}
\def\R{{\mathbb{R}}}
\def\Q{{\mathbb{Q}}}
\newcommand{\ssm}{\smallsetminus}

\newtheorem{example}{Example}
\newtheorem{corollary}{Corollary}
\newtheorem{prop}{Proposition}
\newtheorem{theorem}{Theorem}
\newtheorem{lemma}{Lemma}
\newtheorem{defn}{Definition}
\newtheorem{remark}{Remark}

\begin{document}
\maketitle

\begin{abstract}
In this paper we study group actions on hyperbolic $\Lambda$-metric spaces, where $\Lambda$ is an 
ordered abelian group. $\Lambda$-metric spaces were first introduced by Morgan and Shalen in their 
study of hyperbolic structures and then Chiswell, following Gromov's ideas, introduced the notion of 
hyperbolicty for such spaces. Only the case of $0$-hyperbolic $\Lambda$-metric spaces (that is, 
$\Lambda$-trees) was systematically studied, while the theory of general hyperbolic $\Lambda$-metric 
spaces was not developed at all. Hence, one of the goals of the present paper was to fill this gap and 
translate basic notions and results from the theory of group actions on hyperbolic (in the usual sense)
spaces to the case of $\Lambda$-metric spaces for an arbitrary $\Lambda$. The other goal was to 
show some principal difficulties which arise in this generalization and the ways to deal with them.
\end{abstract}

\section{Introduction}

In this paper we introduce and  study group actions on hyperbolic $\Lambda$-metric spaces. This is 
a natural development of the theory of groups acting on $\Lambda$-trees. We extend some ideas of 
Morgan, Shalen, Bass, Chiswell, and Gromov to hyperbolic metric spaces, where the metric takes values 
in an arbitrary ordered abelian group $\Lambda$. 

\medskip

{\bf Motivation.} This research stems from several areas. Firstly, it is a very natural generalization 
of the theory of groups acting on $\Lambda$-trees. It turned out that in the study of group actions 
on $\Lambda$-trees is convenient sometimes to take a wider look and consider actions on 
hyperbolic $\Lambda$-metric spaces. This makes results much more general, but also more elegant 
and sometimes shorter. Secondly, this gives a new approach to general hyperbolicity and a new 
framework to study groups acting on hyperbolic $\Lambda$-spaces. Thus, Gromov hyperbolic groups 
can be viewed as $\Z$-hyperbolic, Fuchsian groups as well as Kleinean groups, as $\R$-hyperbolic 
etc. New interesting classes of $\Lambda$-hyperbolic groups appear as a result of various ``limit'' 
constructions. Recall, that limit groups (which are limits of free groups in Gromov-Hausdorff metric) 
are $\Z^n$-free, that is, they act freely on $\Z^n$-trees \cite{Kharlampovich_Myasnikov:1998(1), Myasnikov_Remeslennikov_Serbin:2005}, which is one of the crucial properties of these groups. 
Similarly, limits of torsion-free Gromov hyperbolic groups are $\Z^n$-hyperbolic 
\cite{Myasnikov_Serbin:2013}. Moreover, various non-standard versions of hyperbolic groups 
(ultrapowers of hyperbolic groups and their subgroups) also act nicely on hyperbolic $\tilde{\Z}$-spaces, 
where $\tilde{\Z}$ is the group of non-standard integers ($\tilde{\Z}$ is an ultrapower of $\Z$). 
Thirdly, we believe that this framework gives a unified approach to several open problems related to model 
theory of hyperbolic groups, questions on algebraic structure of subgroups of hyperbolic groups 
and relatively hyperbolic groups, constructions of effective versions of asymptotic cones of 
hyperbolic-like groups and some others.  

\medskip     

{\bf Results.} We lay down foundations of the theory in Sections \ref{se:metric-spaces} and 
\ref{subs:group_actions}.

Let $\Lambda$ be an ordered abelian group. In Section \ref{se:metric-spaces} we discuss hyperbolic 
$\Lambda$-metric spaces. In fact, this notion is not new, in \cite{Morgan_Shalen:1984} Morgan and 
Shalen defined $\Lambda$-metric space for an arbitrary $\Lambda$, while in \cite{Chiswell:2001} 
Chiswell, following Gromov's ideas, gave a definition of a hyperbolic $\Lambda$-metric space. We 
show that most of the classical definitions of hyperbolicity remain valid and equivalent in the general 
case, which gives the base for the whole study. We introduce the notion of a boundary of a hyperbolic 
$\Lambda$-metric space and establish some of its basic properties which we use throughout the paper.  
In Section \ref{subs:isometries} we study isometries of hyperbolic $\Lambda$-metric spaces. The results are 
more technical and proofs are more involved than both in the case of isometries of $\Lambda$-trees 
and the classical $\R$-hyperbolic spaces, since in the general case one has to accommodate the both 
of these. The following result (Theorem \ref{th:Lambda_isom} in Section \ref{subs:isometries}) is a 
crucial result here which gives classification of isometries in the general setting. Let $(X,d)$ be 
a geodesic $\delta$-hyperbolic $\Lambda$-metric space. Then every minimal isometry of $X$ is either 
elliptic, or parabolic, or hyperbolic in the case when $\Lambda = 2 \Lambda$, and is either elliptic, 
or parabolic, or hyperbolic, or an inversion when $\Lambda \neq 2 \Lambda$. We have (see Section 
\ref{subs:isometries}) a more detailed description of isometries and their properties in two principle 
cases, when $\Lambda$ is equal to either $\R^n$, or $\Z^n$ (both with the right lexicographic order). 
We conclude Section \ref{se:metric-spaces}  with examples of hyperbolic $\Lambda$-metric spaces.  

In Section \ref{subs:group_actions} we, following ideas of Lyndon and Gromov we introduce group based 
hyperbolic length functions with values in $\Lambda$. From this view-point Lyndon's length functions 
are $0$-hyperbolic, and our general hyperbolic length functions occur when the Lyndon's 
$0$-hyperbolicity axiom is replaced by a general one that corresponds to the hyperbolicity condition 
on the Gromov's products (which can be easily expressed in terms of the length functions). Chiswell 
in \cite{Chiswell:2001} showed that groups with Lyndon length functions $l : G \to \R$ (and an extra 
axiom) are precisely those ones that act freely on $\R$-trees, and later Morgan and Shalen generalized 
his construction in \cite{Morgan_Shalen:1984} to arbitrary $\Lambda$ (we refer the reader to the 
book \cite{Chiswell:2001} for details). In Section \ref{subsec:actions-length} we show how an action 
of a group $G$ by isometries on a (hyperbolic) $\Lambda$-metric space naturally induces a (hyperbolic) 
length function on $G$ with values in $\Lambda$. And in Section \ref{subsec:length-actions} we prove 
the converse, thus establishing equivalence of these two approaches. In the end of Section 
\ref{subs:group_actions} we give examples of groups acting on hyperbolic $\Lambda$-metric spaces.  
This gives, as in the classical Bass-Serre theory of groups acting on trees,  an equivalent approach 
to study group actions on hyperbolic $\Lambda$-metric spaces.

In Section \ref{subs:delta} we consider {\em kernels} of hyperbolic length functions. Let  $G$ be a 
group with a length function $l : G \to \Lambda$. For a fixed convex subgroup $\Lambda_0 \leqslant 
\Lambda$ one can define the $\Lambda_0$-kernel of $G$ by $G_{\Lambda_0} = \{g \in \mid l(g) \in 
\Lambda_0\}$, which is a subgroup of $G$. If the hyperbolicity constant $\delta$ is greater then any 
element in $\Lambda_0$ (that is, $\delta \not\in \Lambda_0$) then the restriction of the function $l$ 
to $G_{\Lambda_0}$ becomes $\delta$-hyperbolic, in other words, $l$ does not say much about the 
$\Lambda_0$-kernel. This shows that if $\Lambda$ is not archimedean then all elements in $G$ of 
length ``infinitely smaller'' than $\delta$ become invisible for the function $l$, so the 
$\delta$-hyperbolicity axiom does not impose any restrictions on them. To deal with this on the group 
level we use the idea of a group which is hyperbolic relative to a subgroup (see below).

It turns out that for a non-Archimedean $\Lambda$ group actions on hyperbolic $\Lambda$-metric 
spaces can be quite cumbersome, they might have rather strange properties that do not occur in the 
classical situations. In Section \ref{sec:hyp_length_func} with introduce several natural types of 
group actions and the corresponding length functions: {\em regular, complete, free,} and {\em proper}. 
The axioms on length functions associated with these action types shed some light on the algebraic 
structure of the underlying groups. In particular, in Section \ref{subs:delta} we consider actions 
of a finitely generated group $G$ on a geodesic $\delta$-hyperbolic $\R$-metric space $(X, d)$ and 
show (Theorem \ref{th:rel_hyp_1}) that if the action is ``nice'' (regular and proper) then $G$ is 
weakly hyperbolic (in the sense of Farb, and Osin \cite{Farb:1998, Osin:2006}) relative to the kernel 
of the associated length function. This is an analog of the classical result on hyperbolicity of groups 
acting ``nicely'' on hyperbolic metric spaces. We refer the reader to Section \ref{sec:kernel} for 
some interesting applications of this result.

In Section \ref{sec:geod_spaces}  we investigate how one can ``complete''  a given non-geodesic 
hyperbolic $\Z$-metric space $X$ to a geodesic one, that is, how one can construct a geodesic 
hyperbolic $\Z$-metric space $\overline{X}$ which $X$ (quasi-)isometrically embeds into. According 
to Bonk and Schramm, any $\delta$-hyperbolic $\Z$-metric space embeds isometrically into a complete 
geodesic $\delta$-hyperbolic $\R$-metric space (see \cite{Bonk_Schramm:2000}), but unfortunately 
this completion does not have to be a $\Z$-metric space. For a given hyperbolic $\Z$-metric space 
$X$ we introduce two $\Z$-completions of $X$ which we call $\Gamma_1(X)$ and $\Gamma_2(X)$. Our 
constructions will have, compared to Bonk and Schramm's, the disadvantage that the hyperbolicity 
constant will increase. However, they will have the advantage that isometries, embeddings and 
quasi-isometries of $X$ extend easily and that boundaries are easy to work with.

\section{Hyperbolic $\Lambda$-metric spaces}
\label{se:metric-spaces}

\subsection{Ordered abelian groups}
\label{subs:ordered}

In this section   we only mention some  definitions  and facts that are crucial for understanding of 
the main concepts of the paper. For details on ordered abelian groups we refer to books 
\cite{Fuchs:1963, Kopytov_Medvedev:1996, Glass:1999}.

An {\em ordered} abelian group is an abelian group $\Lambda$ (with addition denoted by ``$+$'') 
equipped with a linear order ``$\leqslant$''  such that the following axiom holds:
\begin{enumerate}
\item[(OA)] for all $\alpha, \beta, \gamma \in \Lambda$, $\alpha \leqslant \beta$ implies $\alpha +
\gamma \leqslant \beta + \gamma$.
\end{enumerate}

An abelian group $\Lambda$ is called {\em orderable} if there exists a linear order ``$\leqslant$'' 
on $\Lambda$, satisfying the condition (OA) above. In general, $\Lambda$ can be ordered in many 
different ways. In what follows $\Lambda$ always denotes an ordered abelian group.

If $A$ and $B$ are ordered abelian group then their direct sum $A \oplus B$  can be ordered with 
the {\em right lexicographic order}, where  one compares first the right components of two pairs 
and if they are equal than the left ones, that is, $(a,b) \leqslant (c,d)$ if and only if either $b 
< d$, or $b = d$ and $a \leqslant c$. Similarly, one can define the {\em left lexicographic} order 
on $A \oplus B$. Throughout the paper we consider only the right lexicographic order. Furthermore, 
the direct powers $\Z^n$ and $\R^n$, if not said otherwise, are always considered in the right 
lexicographic order.

An ordered abelian group $\Lambda$ is called {\em discretely ordered} or {\em discrete} if $\Lambda$
has a minimal positive element, which we denote by $1$. It will be always  clear from the context 
whether $1$ represents a natural number, or the minimal positive element of $\Lambda$. If $\Lambda$ 
is discrete then  for any $\alpha \in \Lambda$ the following hold:
\begin{enumerate}
\item[(1)] $\alpha + 1 = \min\{\beta \mid \beta > \alpha\}$,
\item[(2)] $\alpha - 1 = \max\{\beta \mid \beta < \alpha\}$.
\end{enumerate}
Notice, that $\Z^n$ is discretely ordered for any $n > 0$, but $\R^n$ is not.

Sometimes we would like to be able to divide elements of $\Lambda$ by non-zero integers. To this 
end we fix a canonical order-preserving embedding of $\Lambda$ into an ordered {\em divisible} 
abelian group $\Lambda_\Q$  and identify $\Lambda$ with its image in $\Lambda_\Q$. The group 
$\Lambda_\Q$ is the tensor product $\Q \otimes_\Z \Lambda$ of two abelian groups (viewed as 
$\Z$-modules) over $\Z$. One can represent elements of $\Lambda_\Q$ by fractions $\tfrac{\lambda}{m}$, 
where $\lambda \in \Lambda, m \in \Z, m \neq 0$,  and two fractions $\tfrac{\lambda}{m}$ and 
$\tfrac{\mu}{n}$ are equal if and only if $n \lambda = m \mu$. Addition of fractions is defined as 
usual, and the embedding is given by the map $\lambda \to \tfrac{\lambda}{1}$. The order on $\Lambda_\Q$
is defined by $\tfrac{\lambda}{m} \geqslant 0 \Longleftrightarrow  m \lambda \geqslant 0 \text{  
in  } \Lambda$.  Obviously, the embedding $\Lambda \to \Lambda_\Q$ preserves the order. It is easy 
to see that $\R_\Q = \R$ and $\Z_\Q = \Q$. Furthermore, it is not hard to  show that $(A \oplus B)_\Q 
\simeq A_\Q \oplus B_\Q$, so $(\R^n)_\Q = \R^n$ and $(\Z^n)_\Q = \Q^n$. Notice also, that for every
$\Lambda$ the group $\Z \oplus \Lambda$ is discrete.

For elements $\alpha, \beta \in \Lambda$ the {\em closed segment} $[\alpha, \beta]$ is defined by
$$[\alpha, \beta] = \{\gamma \in \Lambda \mid \alpha \leqslant \gamma \leqslant \beta \}.$$
Now a  subset $C \subset \Lambda$ is called {\em convex} if for every $\alpha, \beta \in C$ the set
$C$ contains $[\alpha, \beta]$. In particular, a subgroup $C$ of $\Lambda$ is convex if $[0, \beta]
\subset C$ for every positive $\beta \in C$. Observe, that  the set of all convex subgroups of 
$\Lambda$ is linearly ordered by inclusion. In the case when $\Lambda = \R^n$, or $\Lambda = \Z^n$
the convex subgroups form a chain: $ 0 < \Lambda_1 < \cdots < \Lambda_n = \Lambda$, where 
$\Lambda_i = \{(\lambda_1, \ldots, \lambda_i, 0, \ldots, 0) \mid \lambda_j \in \R \ (\text{or } 
\Z)\}$. In this case $\Lambda$ has a (unique) minimal non-trivial convex subgroup $\Lambda_1$.

For any $a \in \Lambda$ we define $|a| = a$ if $a \geq 0$ and $|a| = -a$ otherwise.

\subsection{$\Lambda$-metric spaces}
\label{subs:lambda_metric_space}

In \cite{Morgan_Shalen:1984} Morgan and Shalen defined $\Lambda$-metric spaces for an arbitrary 
ordered abelian group $\Lambda$.

Let $X$ be a non-empty set and $\Lambda$ an ordered abelian group. A {\em $\Lambda$-metric} on $X$ 
is a mapping $d: X \times X \rightarrow X$ such that:
\begin{enumerate}
\item[(LM1)] $\forall\ x,y \in X:\ d(x,y) \geqslant 0$;
\item[(LM2)] $\forall\ x,y \in X:\ d(x,y) = 0 \Leftrightarrow x = y$;
\item[(LM3)] $\forall\ x,y \in X:\ d(x,y) = d(y,x)$;
\item[(LM4)] $\forall\ x,y,z \in X:\ d(x,y) \leqslant d(x,z) + d(y,z)$.
\end{enumerate}

A {\em $\Lambda$-metric space} is a pair $(X,d)$, where $X$ is a non-empty set and $d$ is a 
$\Lambda$-metric on $X$. Usually, unless specified otherwise, we always assume that there is no 
convex subgroup $\Lambda_0$ of $\Lambda$ such that $d(x, y) \in \Lambda_0$ for every $x, y \in X$ 
(otherwise we can replace $\Lambda$ by $\Lambda_0$).

\begin{example}
For any ordered abelian group $\Lambda$ the map  $d(a,b) = |a - b|$ is a metric, so  $(\Lambda, d)$
is a $\Lambda$-metric space.
\end{example}

We fix a $\Lambda$-metric space $(X,d)$ and  a convex subgroup  $\Lambda_0$ of $\Lambda$. For any 
point $x \in X$ the  subset
$$X_{x, \Lambda_0} = \{y \in X \mid d(x, y) \in \Lambda_0\}$$
of $X$ is a $\Lambda_0$-metric space with respect to the metric $d_0 = d_{\mid_{X_0}}$, called a 
{\em $\Lambda_0$-metric subspace} of $X$.

If $x \in X$ and $\varepsilon \in \Lambda$ is positive then we define the {\em ball of radius
$\varepsilon$ centered at  $x$} as usual by
$$B_\varepsilon(x) = \{y \in X \mid d(x,y) \leqslant \varepsilon\}.$$
A subset $Y \subseteq X$ is {\em bounded} if $Y \subseteq B_\varepsilon(x)$ for some $x \in X$ and
$\varepsilon \geqslant 0$. If $\Lambda_0 \neq  \Lambda$ then any $\Lambda_0$-metric subspace 
$X_{x, \Lambda_0} $ of $X$  is bounded (it is contained in $B_\varepsilon(x)$ for any  $0 < \varepsilon 
\in  \Lambda \ssm \Lambda_0$).

If $(X,d)$ and $(X',d')$ are $\Lambda$-metric spaces, an {\em isometry} from $(X,d)$ to $(X',d')$ 
is a mapping $f: X \rightarrow X'$ such that $d(x,y) = d'(f(x), f(y))$ for all $x, y \in X$.

A mapping $f: X \rightarrow X'$ is called a {\em $(\lambda, c, L)$-local-quasi-isometry} from $(X,d)$ 
to $(X',d')$, where $\lambda \in \Z,\ c, L \in \Lambda$ are such that $\lambda \geqslant 1,\ c, L 
\geqslant 0$, if
$$\frac{1}{\lambda} d(x,y) - c \leqslant d'(f(x), f(y)) \leqslant \lambda d(x,y) + c$$
for all $x, y \in X$ such that $d(x,y) \leqslant L$. Here, as usual, we understand 
$\frac{1}{\lambda} d(x,y)$ as an element in $\Lambda_\Q$.

Similarly, $f$ is a {\em $(\lambda, c)$-quasi-isometry} if the inequalities above hold for any $x,y 
\in X$ (the condition $d(x,y) \leqslant L$ is dropped).

A {\em segment} in a $\Lambda$-metric space $X$ is the image of an isometry $\alpha: [a,b]
\rightarrow X$ for some $a,b \in \Lambda$.  In this case  $\alpha(a),  \alpha(b)$ are called the 
endpoints of the segment. By $[x,y]$ we denote any segment with endpoints $x,y$.

We call a $\Lambda$-metric space $(X,d)$ {\em geodesic} if for all $x,y \in X$, there is a segment
in $X$ with endpoints $x,y$. $(X,d)$ is {\em geodesically linear} if for all $x,y \in X$, there is 
a unique segment in $X$ with endpoints $x,y$.

\begin{lemma} \cite[Lemma 1.2.2]{Chiswell:2001}
\label{le:1.2.2}
Let $(X, d)$ be a $\Lambda$-metric space.
\begin{enumerate}
\item Let $\sigma$ be a segment in $X$ with endpoints $x, z$ and let $\tau$ be a segment in $X$ with
endpoints $y, z$.
\begin{enumerate}
\item[(a)] Suppose that, for all $u \in \sigma$ and $v \in \tau$, $d(u, v) = d(u, z) + d(z, v)$. Then
$\sigma \cup \tau$ is a segment with endpoints $x, y$.
\item[(b)] if $\sigma \cap \tau = \{ z \}$ and $\sigma \cup \tau$ is a segment, then its endpoints
are $x, y$.
\end{enumerate}
\item Assume that $(X, d)$ is geodesically linear. Let $x, y$ and $z \in X$, and let $\sigma$ be the
segment with endpoints $x, y$. Then $z \in \sigma$ if and only if $d(x, y) = d(x, z) + d(z, y)$.
\end{enumerate}
\end{lemma}

\subsection{Definition of hyperbolic $\Lambda$-metric space}
\label{subs:hyp_lambda_metric_space}

In \cite{Chiswell:2001} Chiswell, generalizing Gromov's approach to hyperbolicity \cite{Gromov:1987},
introduced hyperbolic $\Lambda$-metric spaces. We briefly discuss this notion below.

Let $(X,d)$ be a $\Lambda$-metric space. Fix  a point $v \in X$  and for $x,y \in X$ define the 
Gromov's product
$$(x \cdot y)_v = \frac{1}{2} (d(x,v) + d(y,v) - d(x,y)),$$
as an element of $\Lambda_\Q$.
A straightforward  computation shows that, if $t$ is another point from $X$ then
$$(x \cdot y)_t = d(t, v) + (x \cdot y)_v - (x \cdot t)_v - (y \cdot t)_v.$$
This and the triangle inequality implies the following result.

\begin{lemma}
\label{le:1.2.4}
Let $(X, d)$ be a $\Lambda$-metric space. Then the following hold:
\begin{enumerate}
\item For any $v, x, y \in X$
$$0 \leqslant (x \cdot y)_v \leqslant \min\{d(x, v), d(y, v)\}.$$
\item If for some $v \in X$ and all $x, y \in X,\ (x \cdot y)_v \in \Lambda$ then for all $v, x, y 
\in X,\ (x \cdot y)_v \in \Lambda$.
\end{enumerate}
\end{lemma}

Now, following Gromov (see \cite{Gromov:1987}) one can define a hyperbolic $\Lambda$-metric space.

\begin{defn}
Let $\delta \in \Lambda$ with $\delta \geqslant 0$. Then $(X, d)$ is {\em $\delta$-hyperbolic with 
respect to $v$} if, for all $x, y, z \in X$
$$(x \cdot y)_v \geqslant \min\{(x \cdot z)_v, (z \cdot y)_v\} - \delta.$$
\end{defn}

\begin{lemma}\cite[Lemma 1.2.5]{Chiswell:2001}
\label{le:1.2.5}
If $(X, d)$ is $\delta$-hyperbolic with respect to $v$, and $t$ is any other point of $X$, then
$(X, d)$ is $2\delta$-hyperbolic with respect to $t$.
\end{lemma}

In view of Lemma \ref{le:1.2.5}, we call a $\Lambda$-metric space $(X, d)$ {\em $\delta$-hyperbolic}
if it is $\delta$-hyperbolic with respect to any point of $X$.

The definition of $\delta$-hyperbolicity can be reformulated as follows.

\begin{lemma}\cite[Lemma 1.2.6]{Chiswell:2001}
\label{le:1.2.6}
The $\Lambda$-metric space $(X, d)$ is $\delta$-hyperbolic if and only if  any $x, y, z, t \in X$
satisfy the following {\em $4$-point condition}:
$$d(x, y) + d(z, t) \leqslant \max\{d(x, z) + d(y, t), d(y, z) + d(x, t)\} + 2\delta.$$
\end{lemma}

\subsection{Geodesic hyperbolic $\Lambda$-metric spaces}
\label{subs:geod_hyp_lambda_metric_space}

Geodesic hyperbolic $\Lambda$-metric spaces have some nice geometric properties, which can be 
expressed in various forms of ``thinness'' of geodesic triangles.  In this section $(X,d)$ is a 
geodesic $\Lambda$-metric space.

$\Lambda$-trees give an important class of $0$-hyperbolic geodesic $\Lambda$-metric spaces. They 
were introduced by Morgan and Shalen in \cite{Morgan_Shalen:1984}. Recall that a $\Lambda$-metric 
space is a {\em $\Lambda$-tree} if it satisfies  the following axioms:
\begin{enumerate}
\item[(T1)] $(X,d)$ is geodesic,
\item[(T2)] if two segments of $(X,d)$ intersect in a single point, which is an endpoint of both, 
then their union is a segment,
\item[(T3)] if the intersection of two segments with a common endpoint is also a segment.
\end{enumerate}

If $X$ is a $\Lambda$-tree and $x, y, z \in X$ then $[x,y] \cap [x,z] = [x,w]$ for some $w \in X$. 
In this case we write  $w = Y(y,x,z)$.

The following theorem was proved in \cite{Chiswell:2001} (Lemmas 2.1.6 and  2.4.3).
\begin{theorem}
A geodesic $\Lambda$-metric space $(X,d)$ is a $\Lambda$-tree if and only if it satisfies the 
following conditions:
\begin{enumerate}
\item[(1)] for all $x,y,v \in X$ $(x\cdot y)_v \in \Lambda$,
\item[(2)] $(X,d)$ is $0$-hyperbolic.
\end{enumerate}
\end{theorem}

In particular, if $\Lambda$ is a divisible ordered abelian group (for instance $\Lambda = \R^n$) 
then the first condition in the theorem is always satisfied, so in this case $\Lambda$-trees are 
precisely geodesic $0$-hyperbolic $\Lambda$-metric spaces.

Now we give a characterization of hyperbolic geodesic $\Lambda$-metric spaces in terms of thin 
triangles.

A {\em $\Lambda$-tripod} in a $\Lambda$-metric space is a $\Lambda$-tree spanned by three points 
(including degenerate cases when the points coincide, or are collinear). Here is an analog of 
Proposition 2.2 from \cite{Ghys_delaHarpe:1991}. The proof is straightforward.

\begin{lemma}
\label{le:1.2.7}
Let $(X, d)$ be a $\Lambda$-metric space such that for all $x, y, z \in X,\ (x \cdot y)_z \in
\Lambda$. Then for all $x,y,z \in X$ there exists a $\Lambda$-tripod $T$ and an isometry $\phi :
\{x,y,z\} \to T$, where $T$ is spanned by $\phi(x), \phi(y)$ and $\phi(z)$ such that $(x \cdot y)_z$
is equal to the length of the intersection $[\phi(z), \phi(x)] \cap [\phi(z), \phi(y)]$ in $T$.
\end{lemma}

We call the $\Lambda$-tripod $T$ from the lemma above a {\em comparison $\Lambda$-tripod} for the
triple $\{x,y,z\}$ and denote $T = T(x,y,z)$.

If $(X,d)$ is geodesic then  the isometry $\phi : \{x,y,z\}$ $\to$ $T(x,y,z)$ extends to an isometry 
of the geodesic triangle $\Delta(x,y,z) = [x,y] \cup [x,z] \cup [y,z]$ in $(X,d)$ to $T$ whose 
restriction to $\{x,y,z\}$ is exactly $\phi$ (we denote this extension again by $\phi$). Now, 
$\Delta(x,y,z)$ is called {\em $\delta$-thin} for some $\delta \in \Lambda$ if $d(u,v) \leqslant 
\delta$ for all $u,v \in \Delta(x,y,z)$ such that $\phi(u) = \phi(v)$.

\begin{lemma}
\label{le:1.2.8}
Let $(X, d)$ be a geodesic $\Lambda$-metric space such that for all $x, y, z \in X,\ (x \cdot y)_z
\in \Lambda$. Then
\begin{enumerate}
\item[(i)] $(x \cdot y)_z \leqslant d(z, [x,y])$,
\item[(ii)] if $\Delta(x,y,z)$ is $\delta$-thin then $d(z, [x,y]) \leqslant (x \cdot y)_z + \delta$
\end{enumerate}
\end{lemma}
\begin{proof} The proof repeats the one of Lemma 2.17 \cite{Ghys_delaHarpe:1991}.

\smallskip

Let $p \in [x,y],\ q \in [x,z],\ r \in [y,z]$ be such that $\phi(p) = \phi(q) = \phi(r) = Y(\phi(x),
\phi(y), \phi(z))$.

Let $w \in [x,y]$ be such that $d(z, [x,y]) = d(z,w)$. Then there exists $w' \in [x,z] \cup [y,z]$
such that $\phi(w) = \phi(w')$. Without loss of generality assume that $w' \in [x,z]$. Then
$$(x \cdot y)_z \leqslant d(w', z) = d(x,z) - d(x,w') = d(x,z) - d(x,w) \leqslant d(z,w) = d(z, [x,y])$$
which proves (i).

Finally, if $\Delta(x,y,z)$ is $\delta$-thin then we have
$$d(z, [x,y]) \leqslant d(z,q) + d(p,q) \leqslant (x \cdot y)_z + \delta.$$
\end{proof}

\begin{prop}
\label{pr:1.2.9}
Let $(X, d)$ be a geodesic $\Lambda$-metric space such that for all $x, y, z \in X,\ (x \cdot y)_z
\in \Lambda$. Consider the following properties of $(X,d)$:
\begin{enumerate}
\item[(H1,$\delta$)] $(X,d)$ is $\delta$-hyperbolic,
\item[(H2,$\delta$)] $\Delta(x,y,z)$ is $\delta$-thin for any $x,y,z \in X$,
\item[(H3,$\delta$)] $d(u, [x,z] \cup [y,z]) \leqslant \delta$ for any $x,y,z \in X$ and $u \in
[x,y]$.
\end{enumerate}
Then the following implications hold
$$(H1,\delta) \Longrightarrow (H2,4\delta),\ \ \ (H2,\delta) \Longrightarrow (H1,2\delta)$$
$$(H2,\delta) \Longrightarrow (H3,\delta),\ \ \ (H3,\delta) \Longrightarrow (H2,4\delta)$$
$$(H1,\delta) \Longrightarrow (H3,4\delta),\ \ \ (H3,\delta) \Longrightarrow (H1,8\delta)$$
\end{prop}
\begin{proof} We follow the proof of Proposition 2.21 \cite{Ghys_delaHarpe:1991}.

\smallskip

$(H1,\delta) \Longrightarrow (H2,4\delta)$:

\smallskip

Let $x,y,z \in X$ and $T(x,y,z)$ a comparison $\Lambda$-tripod (denote by $d'$ the metric on
$T(x,y,z)$). Let $u \neq v \in \Delta(x,y,z)$ be such that $\phi(u) = \phi(v)$. We have to show
that $d(u,v) \leqslant 4\delta$. Without loss of generality we can assume that $u \in [x,y],\ v \in
[x,z]$. If $t = d(x,u)$ then
$$d'(\phi(x),\phi(u)) = d'(\phi(x), \phi(v)) = t \leqslant (y \cdot z)_x$$
$$(u \cdot y)_x = (v \cdot z)_x = t$$
which implies
$$(u \cdot v)_x \geqslant \min\{(u \cdot y)_x, (v \cdot y)_x\} - \delta \geqslant
\min\{(u \cdot y)_x, (y \cdot z)_x, (z \cdot v)_x\} - 2\delta = t - 2\delta.$$
Since $(u \cdot v)_x = t - \frac{1}{2}d(u,v)$ then
$$d(u,v) = 2t - (u \cdot v)_x \leqslant 2t - 2(t-2\delta) = 4\delta.$$

\smallskip

$(H2,\delta) \Longrightarrow (H1,2\delta)$:

\smallskip

Suppose that all triangles in $X$ are $\delta$-thin and let $x_0, x_1, x_2, x_3 \in X$. We have to
show that
$$(x_1 \cdot x_2)_{x_0} \geqslant  \min\{(x_1 \cdot x_3)_{x_0}, (x_2 \cdot x_3)_{x_0}\} - 2\delta$$
Denote $t = \min\{(x_1 \cdot x_3)_{x_0}, (x_2 \cdot x_3)_{x_0}\}$. If $t \leqslant (x_1 \cdot x_2)_{x_0}$
then there is nothing prove, so let $t > (x_1 \cdot x_2)_{x_0}$.

For $i \in \{1,2,3\}$ let $x'_i \in [x_0, x_i]$ be such that $d(x_0, x'_i) = t$. Let $\phi_{i,j},\
i \neq j \in \{1,2,3\}$ be the isometry of the triangle $[x_0, x_i] \cup [x_i, x_j] \cup [x_j, x_0]$ to
the comparison $\Lambda$-tripod $T(x_0, x_i, x_j)$.

In the case when $i = 1, 2$ we have $d(x_0, x'_i) = d(x_0, x'_3) \leqslant (x_i \cdot x_3)_{x_0}$
since $\phi_{i,3}(x'_i) = \phi_{i,3}(x'_3)$ and $d(x_3, x'_i) \leqslant \delta$. Thus we have
$$d(x'_1, x'_2) \leqslant 2\delta.$$

Since $t > (x_1 \cdot x_2)_{x_0}$, there exists $y_j \in [x_1, x_2]$ such that $\phi_{1,2}(x'_j) =
\phi_{1,2}(y_j)$ and $d(x'_j, y_j) \leqslant \delta$. Hence,
$$2\delta \geqslant d(x'_1, x'_2) \geqslant d(y_1, y_2) - 2\delta = d(x_1, x_2) - d(x_1, y_1) -
d(x_2, y_2) - 2\delta$$
$$= d(x_1, x_2) - (d(x_1, x_0) - d(x'_1, x_0)) - (d(x_2, x_0) - d(x'_2, x_0))
- 2\delta$$
$$ = 2t - 2(x_1 \cdot x_2)_{x_0} - 3\delta.$$
So, $(x_1 \cdot x_2)_{x_0} \geqslant t - 2\delta$.

\smallskip

$(H2,\delta) \Longrightarrow (H3,\delta)$: obvious

\smallskip

$(H3,\delta) \Longrightarrow (H2,4\delta)$:

\smallskip

Suppose that $(H2,4\delta)$ does not hold, that is, there exist $x,y,z \in X$ and $u \in [x,y],\
v \in [x,z]$ such that $d(x,u) = d(x,v) < (y \cdot z)_x$ but $d(u,v) > 4\delta$. By Lemma \ref{le:1.2.8}
we have
$$d(v, [x,y]) = \min\{d(v, [x,u]), d(v, [u,y])\} \geqslant \min\{(x \cdot u)_v, (u \cdot y)_v\}.$$
Next, $2(x \cdot u)_v = d(u,v)$ and
$$2(u \cdot y)_v = d(u,v) + d(y,v) - (d(x,y) - d(x,u))$$
$$ = d(u,v) + (d(y,v) + d(x,v) - d(x,y)) \geqslant d(u,v),$$
hence
$$d(v, [x,y]) \geqslant \frac{1}{2} d(u,v) > 2\delta.$$
In particular, $d(x,v) > 2\delta$ and there exists $p \in [x,v]$ such that $d(p,v) = \delta$. Now
we have
$$d(p, [x,y]) \geqslant d(v, [x,y]) - d(v, p) > \delta$$
$$d(p, [y,z]) \geqslant d(x, [y,z]) - d(x,p) \geqslant (y \cdot z)_x - d(x,p) > t - d(x,p)$$
$$= d(v,x) - d(x,p) = d(p,v) = \delta.$$
It follows that $d(p, [x,y] \cup [y,z]) > \delta$ which contradicts our assumption.

\smallskip

$(H1,\delta) \Longrightarrow (H3,4\delta)$: follows from $(H1,\delta) \Longrightarrow (H2,4\delta)$
and $(H2,\delta) \Longrightarrow (H3,\delta)$.

\smallskip

$(H3,\delta) \Longrightarrow (H1,8\delta)$: follows from $(H3,\delta) \Longrightarrow (H2,4\delta)$
and $(H2,\delta) \Longrightarrow (H1,2\delta)$.
\end{proof}

If $(X, d)$ is geodesic then for $x, y, z \in X$ we denote by $\Delta_I(x,y,z)$ the geodesic triangle
whose vertices are the points $p, q, r$ on the sides of $\Delta(x,y,z)$ which are sent to the center 
point of $T(x,y,z)$ under the isometry $\phi : \Delta(x,y,z) \to T(x,y,z)$. From Proposition 
\ref{pr:1.2.9} it follows that if $(X,d)$ is $\delta$-hyperbolic then the length of the sides of 
$\Delta_I(x,y,z)$ is bounded by $4\delta$.

\subsection{Boundaries of hyperbolic $\Lambda$-metric spaces}
\label{subs:boundary}

We say that a sequence $\{\lambda_i\}$ of elements of $\Lambda$ {\em converges to infinity}, and 
write 
$$\lim_{i \to \infty} \lambda_i = \infty$$ 
if for every $\alpha \in \Lambda$ there is a natural number $n_\alpha$ such that $\lambda_i \geqslant
\alpha$ for every $i \geqslant n_\alpha$. Similarly, a double-indexed family $\{\lambda_{i j}\} 
\subseteq \Lambda$ converges to infinity, that is, 
$$\lim_{\substack{i \to \infty\\ j \to \infty}}  \lambda_{i j} = \infty$$ 
if for every $\alpha \in \Lambda$ there is a natural number $n_\alpha$ such that $\lambda_{i j} 
\geqslant \alpha$ for every $i, j \geqslant n_\alpha$.

Since $\Lambda$ is an arbitrary ordered abelian group, the notion of convergence to infinity can be 
applied with respect to any convex subgroup $\Lambda_0 \subseteq \Lambda$ by replacing $\Lambda$
with $\Lambda_0$.

Let $(X, d)$ be a $\Lambda$-metric space. Fix a base point $v \in X$. We say that a sequence of points
$\{x_i\} \subseteq X$ {\em converges to infinity} if
$$\lim_{\substack{i \to \infty\\ j \to \infty}} (x_i \cdot x_j)_v = \infty.$$
Observe that if a sequence $\{x_i\}$ converges to infinity  with respect to $v \in X$ then it converges 
to infinity with respect to any other $v' \in X$, since (by the triangle inequality)
\begin{equation} 
\label{eq:base-points}
|(x_i \cdot x_j)_v - (x_i \cdot x_j)_{v'}| \leqslant  d(v, v')
\end{equation}
for any $x_i, x_j, v, v'$.

Now, we term two convergent to infinity sequences $\{x_i\}, \{y_j\} \subseteq X$ {\em close} or 
{\em equivalent} (and write $\{x_i\} \sim \{y_i\}$) with respect to $v$ if
$$\lim_{\substack{i \to \infty\\ j \to \infty}} (x_i \cdot y_j)_v  = \infty.$$
Notice that (\ref{eq:base-points}) shows again that if $\{x_i\}, \{y_j\}$ are equivalent with respect 
to $v$ then they are equivalent with respect to any base point $v' \in X$.

\begin{lemma}
\label{le:1.2.10}
If $(X,d)$ is a hyperbolic $\Lambda$-metric space then the property ``to be close'' defines an 
equivalence relation on the set of all sequences in $(X,d)$ that converge to infinity.
\end{lemma}
\begin{proof} Reflexivity and symmetry of ``$\sim$'' follow immediately from definitions. To prove 
transitivity consider convergent to infinity sequences $\{x_i\} \sim \{y_j\}$ and $\{y_j\} \sim 
\{z_k\}$. Suppose that $X$ is $\delta$-hyperbolic for some $\delta \in \Lambda$. Then for a given 
arbitrary $\alpha \in \Lambda$ choose $n_\alpha$ be such that
$$(x_i \cdot y_j)_v \geqslant \alpha + \delta,\ \ \ (y_j \cdot z_k)_v \geqslant \alpha + \delta$$
for some $v \in X$ and all $i,j,k \geqslant n_\alpha$. Hence,
$$(x_i \cdot z_k)_v \geq \min\{(x_i \cdot y_j)_v, (y_j \cdot z_k)_v\} - \delta \geqslant
(\alpha + \delta) - \delta = \alpha$$
which shows that $\{x_i\}$ and $\{z_k\}$ are $\Lambda$-close.
\end{proof}

For a hyperbolic $\Lambda$-metric space $(X,d)$ we define the {\em boundary at infinity} of $X$ as 
the set of equivalence classes of close convergent to infinity sequences in $(X,d)$ and denote it 
$\partial X$. Observe that if $\Lambda = \R$ then $\partial X$ is the hyperbolic boundary of $(X,d)$. 
If $a \in \partial X$ and $\{x_i\} \in a$ then we write $x_i \to a$ as $i \to \infty$.

Now let $\Lambda_0$ be a convex non-trivial subgroup of $\Lambda$ and $v \in X$. Take a 
$\Lambda_0$-subspace $X_{v, \Lambda_0}$ with the base point $v$ of $X$, which is a 
$\Lambda_0$-metric space. If $X$ is $\delta$-hyperbolic and $\delta \in \Lambda_0$ then $X_{v, 
\Lambda_0}$ is $\delta$-hyperbolic, so the argument above applies and one gets the boundary at 
infinity $\partial X_{v, \Lambda_0}$ of $X_{v, \Lambda_0}$, which we call the {\em $\Lambda_0$-boundary 
of $X$ with respect to the base point $v$}. Notice that in the case of $X_{v, \Lambda_0}$ one can 
consider sequences not only from $X_{v, \Lambda_0}$ but also $\{x_i\}$ such that $x_i \in X_{v, 
\Lambda_0}$ for all sufficiently large $i$.

Recall that in the case of hyperbolic $\R$-metric spaces there exists a notion of the Gromov product 
on the boundary. Similarly, for $x, y \in \partial X$ we define the {\em Gromov product on the 
$\Lambda$-boundary} as follows
$$(\alpha \cdot \beta)_v = \sup_{\substack{x_i \to \alpha\\ y_j \to \beta}} \liminf_{\substack{i \to
\infty\\ j \to \infty}} (x_i \cdot y_j)_v$$ 
provided the limit exists (it depends on $\Lambda$).

\begin{lemma}
\label{le:1.2.11}
If $\Lambda = \R^n$ with the right lexicographic order and $\delta = (d, 0, \ldots, 0)$, then 
$(\alpha \cdot \beta)_v$ exists for any distinct $\alpha, \beta \in \partial X$.
\end{lemma}
\begin{proof}
We define 
$$(\{x_i\} \cdot \{y_j\})_v = \liminf_{\substack{i \to \infty\\ j \to \infty}}(x_i \cdot y_j)_v.$$ 
We have to prove first that $(\{x_i\} \cdot \{y_j\})_v$ exists for any two sequences converging at 
infinity which are not equivalent and then prove that $(\alpha \cdot \beta)_v$ exists for any 
distinct $\alpha, \beta \in \partial X$.

Suppose then that $\{x_i\}$ and $\{y_j\}$ are two sequences converging to infinity. If for any $a 
\in \R^n$ there exist some indices $k_a, k_a'$ such that $(x_{k_a} \cdot y_{k_a'})_v > a$, then 
the sub-sequences $\{x_{k_a}\}$ and $\{y_{k_a'}\}$, where $a = (0, \ldots, 0, n)$, converge to the 
same point in $\partial X$, which implies that $\alpha$ and $\beta$ coincide which is a contradiction.
Hence, we can assume that there exists some $a \in \R^n$ such that $(x_i \cdot y_j)_v \leqslant a$ 
for any $i, j$. Since $\{x_i\}$ and $\{y_j\}$ converge to infinity, there exist $m, n$ such that 
$(x_m \cdot x_{m+k})_v, (y_n \cdot y_{n+k})_v \geqslant a$ for any $k \geqslant 0$.

Take $M, N \in \N$ such that $M > m$ and $N > n$. Hence,
$$(x_M \cdot y_N)_v \geqslant \min\{(x_M \cdot x_m)_v, (x_m \cdot y_N)_v\} - \delta = (x_m \cdot y_N)_v
- \delta \geqslant (x_m \cdot y_n)_v - 2\delta.$$
A similar argument shows that $(x_m, y_n)_v \geqslant (x_M \cdot y_N)_v - 2\delta$.

Suppose now that $(x_m \cdot y_n)_v = (c_1, \ldots, c_n)$. It follows that 
$$(c_1, \ldots, c_{n-1}, c_n - 2d) = (x_m \cdot y_n)_v - 2\delta \leqslant (x_M \cdot y_N)_v \leqslant 
(x_m \cdot y_n)_v + 2\delta$$
$$ = (c_1, \ldots, c_{n-1}, c_n + 2d)$$ 
for any $M > m,\ N > n$. The completeness of $\R$ implies then that $(\{x_i\} \cdot \{y_j\})_v$ 
exists for any two non-equivalent sequences which converge at infinity.

Proving existence of $(\alpha \cdot \beta)_v$ is quite similar. Let $\{x_i'\}$ tend to $\alpha$ and 
$\{y_j'\}$ tend to $\beta$ and $(x_i' \cdot y_j')_v < a$ for any $i, j$. There exist $m,n$ such that 
$(x_m \cdot x_{m+k}')_v, (y_n \cdot y_{n+k}')_v > a$ for any $k \geqslant 0$. We then use the same 
argument as above to prove that 
$$(\{x_i\} \cdot \{y_j\})_v - 2\delta \leqslant (\{x_i'\} \cdot \{y_j'\})_v \leqslant (\{x_i\} \cdot 
\{y_j\})_v +2\delta,$$ 
so the supremum must again exist by completeness of $\R$.
\end{proof}

\subsection{Isometries of $\Lambda$-metric spaces}
\label{subs:isometries}

Let $(X,d)$ be a $\delta$-hyperbolic $\Lambda$-metric space and $\Lambda_\delta$ be the minimal 
convex subgroup of $\Lambda$ containing $\delta$ (that is, for every $\alpha \in \Lambda_\delta$ 
there exists $k \in \N$ such that $\alpha < k \delta$). If $\delta = 0$ then we set $\Lambda_\delta$
to be the trivial subgroup of $\Lambda$. Consider the set of all isometric mappings from $X$ to itself
which we denote by $Isom(X)$. The image of $x \in X$ under $\gamma \in Isom(X)$ we denote by 
$\gamma x$. 

Observe that if $\{x_i\} \subseteq X$ converges to infinity then $\{\gamma x_i\}$ also converges to
infinity, and if $\{x_i\}, \{y_i\}$ are close then $\{\gamma x_i\}, \{\gamma y_i\}$ are also close. 
This shows that $\gamma$ extends to the mapping $\partial X \to \partial X$ which we by abuse of 
notation again denote $\gamma$.

$\gamma \in Isom(X)$ is {\em minimal on $X$} if it does not have any invariant $\Lambda_0$-subspace
of $X$ for some non-trivial proper convex subgroup $\Lambda_0$ of $\Lambda$. If $\gamma$ is not 
minimal then it induces an isometry $\gamma_0$ on a $\Lambda_0$-subspace $X_0$ of $X$ and the 
definitions below apply to the case when $\Lambda = \Lambda_0$ and $X = X_0$. Note that if the set 
$\{\gamma^n x \mid n \in \Z\}$ for some $x \in X$ is bounded by $\alpha \in \Lambda$ and 
$\Lambda_0$ is the minimal convex subgroup of $\Lambda$ containing $\alpha$ then $\gamma$ stabilizes  
$$X_0 = \{y \in X \mid d(x, y) \in \Lambda_0\}$$
which is a $\Lambda_0$-subspace of $X$ containing $x$. Indeed, for every $y \in X_0$ we have
$$d(y, \gamma y) \leqslant d(x, y) + d(x, \gamma x) + d(\gamma x, \gamma y) = 2d(x, y) + d(x, 
\gamma x) \in \Lambda_0$$ 
Hence, $\gamma$ is minimal only if $X_0 = X$. 

An isometry $\gamma : X \to X$ is called {\em elliptic} if there exists $x \in X$ such that the set 
$\{\gamma^n x \mid n \in \Z\}$ has diameter of at most $K \delta$ for some $K \in \N$. This 
definition generalizes the standard definition of elliptic isometry in $\R$-metric spaces. Indeed,
recall that in a proper geodesic $\R$-metric space an isometry is elliptic if it has a fixed point. 
Our definition in this case (when $X$ is a proper geodesic $\R$-metric space) also implies a fixed 
point: since the set $\{\gamma^n x \mid n \in \Z\}$ is bounded and $X$ is proper, there exists a 
subsequence $\{n_i\}$ such that $\gamma^{n_i} x \to y$ for some $y \in X$ and $\gamma y = y$. But 
in general, for an arbitrary $\Lambda$-metric space $X$, even if $\Lambda = \R$ our definition does 
not imply that there exists a fixed point of $\gamma$. 

Suppose $\gamma$ is elliptic and $x \in X$ is such that the set $\{\gamma^n x \mid n \in \Z\}$ is 
bounded by $K \delta$ for some $K \in \N$. Let $X_0$ be the $\Lambda_\delta$-subspace of $X$ 
containing $x$. Hence, for every $y \in X_0$ and $m, n \in \Z$ there exists $M \in \N$ such that
$$d(\gamma^n y, \gamma^m y) \leqslant d(\gamma^n y, \gamma^n x) + d(\gamma^n x, \gamma^m x) + 
d(\gamma^m x, \gamma^m y)$$
$$ = 2d(x, y) + d(\gamma^n x, \gamma^m x) \leqslant 2 d(x, y) + K \delta \leqslant M \delta$$ 
In particular, it follows that $\gamma$ fixes $X_0$ and it cannot be minimal unless $X = X_0$. 

The following fact follows immediately from the definition above.

\begin{lemma}
\label{le:elliptic_R^n}
Let $\Lambda = \R^n$ with the right lexicographic order and let $(X,d)$ be a geodesic 
$\delta$-hyperbolic $\Lambda$-metric space, where $\delta = (d, 0, \ldots, 0)$. Then $\gamma \in
Isom(X)$ is elliptic if and only if for some $x \in X$ the set $\{\gamma^n x \mid n \in \Z\}$ is 
bounded inside of some $\R$-subspace of $X$.
\end{lemma}

An isometry $\gamma : X \to X$ is called {\em parabolic with respect to $x \in X$} if the diameter 
of the set $\{\gamma^n x \mid n \in \Z\}$ is not bounded by any $\alpha \in \Lambda$ and there 
exists $a_x \in \partial X$ such that for any subsequence of integers $\{n_i\}$ with the property 
$d(x, \gamma^{n_i} x) \to \infty$ we have that $\{\gamma^{n_i} x\} \to a_x$. 

\begin{lemma}
\label{le:parabolic}
Let $(X, d)$ be a $\delta$-hyperbolic $\Lambda$-metric space. If $\gamma \in Isom(X)$ is parabolic
with respect to $x \in X$ then $\gamma$ is parabolic with respect to any other $y \in X$ and $a_x
= a_y$.
\end{lemma}
\begin{proof} Suppose $\gamma$ is parabolic with respect to $x$ and fix $y \in X$. Notice that for
a subsequence of integers $\{n_i\}$ we have
$$d(x, \gamma^{n_i} x) \to \infty\ \Longleftrightarrow\ d(y, \gamma^{n_i} y) \to \infty.$$
Indeed, assuming $d(x, \gamma^{n_i} x) \to \infty$, we get
$$d(y, \gamma^{n_i} y) \geqslant d(x, \gamma^{n_i} y) - d(x,y) \geqslant d(x, \gamma^{n_i} x) -
d(\gamma^{n_i} x, \gamma^{n_i} y) - d(x,y)$$
$$= d(x, \gamma^{n_i} x) - 2 d(x,y) \to \infty$$
and the converse implication can be obtained similarly.

\smallskip

Let $v \in X$ be a base-point. If $d(x, \gamma^{n_i} x) \to \infty$ then $(\gamma^{n_i} x \cdot
\gamma^{n_j} x)_v \to \infty$ as $n_i, n_j \to \infty$. Next, from hyperbolicity of $X$ we get
$$(\gamma^{n_i} y \cdot \gamma^{n_j} y)_v \geqslant \min\{(\gamma^{n_i} x \cdot \gamma^{n_i} y)_v,
(\gamma^{n_i} x \cdot \gamma^{n_j} y)_v\} - \delta$$
$$\geqslant \min\{(\gamma^{n_i} x \cdot \gamma^{n_i} y)_v, (\gamma^{n_i} x \cdot \gamma^{n_j} x)_v,
(\gamma^{n_j} x \cdot \gamma^{n_j} y)_v\} -2 \delta$$
which implies that $(\gamma^{n_i} y \cdot \gamma^{n_j} y)_v \to \infty$ as $n_i, n_j \to \infty$ since
$$(\gamma^{n_i} x \cdot \gamma^{n_i} y)_v = \frac{1}{2}(d(v, \gamma^{n_i} x) + d(v, \gamma^{n_i} y) -
d(\gamma^{n_i} x, \gamma^{n_i} y))$$
$$ = \frac{1}{2}(d(v, \gamma^{n_i} x) + d(v, \gamma^{n_i} y) - d(x, y)) \to \infty\ {\rm as}\ n_i \to \infty$$
and similarly $(\gamma^{n_j} x \cdot \gamma^{n_j} y)_v \to \infty$ as $n_j \to \infty$. It follows that
there exists $a_y \in \partial X$ such that $\{\gamma^{n_i} y\} \to a_y$ and $\gamma$ is parabolic 
with respect $y$. Moreover, the fact that $(\gamma^{n_i} y \cdot \gamma^{n_j} y)_v \to \infty$ as 
$n_i, n_j \to \infty$ implies that $a_x = a_y$.
\end{proof}

In view of Lemma \ref{le:parabolic} we say that $\gamma \in Isom(X)$ is {\em parabolic} if it is
parabolic with respect to some $x \in X$. Note that a parabolic isometry cannot fix any 
$\Lambda_0$-subspace $X_0$ of $X$ for $\Lambda_0 \subsetneq \Lambda$, so, it is minimal.

Observe that if $\gamma$ is parabolic and a subsequence of integers $\{n_i\}$ is such that $d(x,
\gamma^{n_i} x) \to \infty$ then $\{\gamma^{n_i} x\}$ can be taken as a representative of $a$. But
then from $d(x, \gamma^{n_i+1} x) \geqslant d(x, \gamma^{n_i+1} x) - d(x, \gamma x)$ we get
$d(x, \gamma^{n_i+1} x) \to \infty$, so, $\{\gamma^{n_i+1} x\} \to a$. But at the same time,
$\{\gamma^{n_i+1} x\} \to \gamma a$ and we get $\gamma a = a$.

\smallskip

$\gamma$ is called {\em hyperbolic with respect to $x \in X$} if the diameter of the set $\{\gamma^n
x \mid n \in \Z\}$ is not bounded by any $\alpha \in \Lambda$ and there exist distinct $a_x, b_x \in 
\partial X$ such that for any subsequence of natural numbers $\{n_i\}$, if $d(x, \gamma^{n_i} x) 
\to \infty$ then $\{\gamma^{-n_i} x\} \to a_x$ and $\{\gamma^{n_i} x\} \to b_x$. Similar to the 
parabolic case one can show that if $\gamma$ is hyperbolic with respect to $x \in X$ then it is 
hyperbolic with respect to any other $y \in X$ and $a_x = a_y,\ b_x = b_y$ (up to a permutation). 
Hence, we say that $\gamma$ is {\em hyperbolic} if it is hyperbolic with respect to any $x \in X$ 
and we denote $a_- = a_x$ and $a_+ = b_x$. Again, similar to the parabolic case one can easily 
show that $\gamma a_- = a_-$ and $\gamma a_+ = a_+$.

It is easy to see that if $\gamma$ is elliptic, parabolic, or hyperbolic then so is $\gamma^k$ for
every fixed $k \in \Z$.

\smallskip

Finally, an isometry $\gamma$ is called an {\em inversion} if $\gamma$ does not fix any 
$\Lambda_\delta$-subspace of $X$, but $\gamma^2$ fixes a $\Lambda_\delta$-subspace of $X$.
Obviously, inversions can exist only in the case when $\Lambda \neq \Lambda_\delta$. Also,
observe that if $\delta = 0$ then any $\Lambda_\delta$-subspace of $X$ is a single point, hence, 
in the case when $X$ is a proper geodesic $0$-hyperbolic $\Lambda$-metric space, that is, a
$\Lambda$-tree, our definition coincides with the definition of inversion for $\Lambda$-trees.

\smallskip

Observe that if $\Lambda = \R$, a geodesic $\Lambda$-metric space $(X,d)$ is an ordinary hyperbolic 
space, so, every isometry of $X$ is either elliptic, or parabolic, or hyperbolic (see \cite[Theorem 
9.2.1]{CoornaertDelzantPapadopoulos:1990}). If $X$ is not geodesic, or if $\Lambda = \Z$ then the
case of an inversion adds.

Next, if $\delta = 0$ then a $\Lambda$-metric space $(X,d)$ is a $\Lambda$-tree and classification
of its isometries is also known (see \cite[Section 3.1]{Chiswell:2001}): any isometry of $X$ is
either elliptic, or hyperbolic, or an inversion in the case when $\Lambda \neq 2\Lambda$.

Our next goal is to classify isometries of $\delta$-hyperbolic $\Lambda$-metric spaces.

The following lemma is similar to \cite[Lemma 2.2]{CoornaertDelzantPapadopoulos:1990}.

\begin{lemma}
\label{le:1.3.1}
Let $\gamma$ be a minimal isometry of a geodesic $\delta$-hyperbolic $\Lambda$-metric space $(X, d)$.
If there exists $x \in X$ such that
$$d(x, \gamma^2 x) > d(x, \gamma x) + 3 \delta$$
then $\gamma$ is hyperbolic in the case when $\Lambda = 2 \Lambda$, and $\gamma$ is either hyperbolic,
or an inversion if $\Lambda \neq 2 \Lambda$.
\end{lemma}
\begin{proof} Consider the points $x,\ \gamma x,\ \gamma^2 x$ and $\gamma^n x$, where $n \in \N$. 
By the $4$-point condition (see Lemma \ref{le:1.2.6}) we have
$$d(x, \gamma^2 x) + d(\gamma x, \gamma^n x) \leqslant \max\{d(x, \gamma x) + d(\gamma^2 x, 
\gamma^n x),\ d(x, \gamma^n x) + d(\gamma^2 x, \gamma x)\} + 2\delta$$
or, if we denote $\alpha_k = d(\gamma^k x, x)$ for every $k \in \N$ then
$$\alpha_2 + \alpha_n \leqslant \max\{\alpha_1 + \alpha_{n-2},\ \alpha_n + \alpha_1\} + 2\delta$$
since $d(\gamma^k x, \gamma^m x) = \alpha_{|k-m|}$. In other words we get
$$\max\{\alpha_{n-2},\ \alpha_n\} \geqslant \alpha_{n-1} + \alpha_2 - \alpha_1 - 2\delta$$
and from the assumption $\alpha_2 > \alpha_1 + 3\delta$ we obtain
$$\max\{\alpha_{n-2},\ \alpha_n\} > \alpha_{n-1} + \delta$$
which holds for any $n \in \N$. Next, we prove by induction on $n$ that 
$$\alpha_n + \delta < \alpha_{n+1}$$
If $n = 0$ then we have
$$\alpha_1 + 3\delta < \alpha_2 \leqslant 2\alpha_1$$
and $\alpha_0 + \delta < \alpha_1$. Suppose the inequality holds for $n$, that is, $\alpha_{n+1} > 
\alpha_n + \delta$. Since we have
$$\max\{\alpha_{n+2},\ \alpha_n\} > \alpha_{n+1} + \delta$$
it implies that $\max\{\alpha_{n+2},\ \alpha_n\} = \alpha_{n+2}$ and
$$\alpha_{n+2} > \alpha_{n+1} + \delta$$
as required. 

In particular, it follows that $\alpha_n > n\delta$. Consider two cases.

\smallskip

{\bf Case I.} There exists $K \in \N$ such that $K \delta > \alpha_1 = d(x, \gamma x)$ ($\delta$
and $d(x, \gamma x)$ are ``comparable'').

\smallskip

In this case, the diameter of the set $\{\gamma^n x \mid n \in \Z\}$ is not bounded by any $\alpha 
\in \Lambda$ (since $\gamma$ is minimal). Moreover,  for any $v \in X$ we have
$$\lim_{\substack{i \to \infty\\ j \to \infty}} (\gamma^i x \cdot \gamma^j x)_v = \infty$$
and
$$\lim_{\substack{i \to \infty\\ j \to \infty}} (\gamma^{-i} x \cdot \gamma^{-j} x)_v = \infty$$
since $\gamma$ is minimal and $\delta$ and $d(x, \gamma x)$ are ``comparable''. It shows that 
there exist distinct $a_x, b_x \in \partial X$ such that for any subsequence of natural numbers 
$\{n_i\}$ we have $\{\gamma^{-n_i} x\} \to a_x$ and $\{\gamma^{n_i} x\} \to b_x$. That is, $\gamma$ 
is hyperbolic with respect to $x$, hence, hyperbolic.
 
\smallskip
 
{\bf Case II.} $d(x, \gamma x) = \alpha_1 > K \delta$ for every $K \in \N$ ($d(x, \gamma x)$ is 
``infinitely large'' with respect to $\delta$).

\smallskip

Recall that $\Lambda_\delta$ is the minimal convex subgroup of $\Lambda$ containing $\delta$. Notice 
that by the assumption we have $\Lambda_\delta \neq \Lambda$. Define an equivalence relation 
``$\sim$'' on $X$ by setting
$$y \sim z\ \ \ \ \Longleftrightarrow\ \ \ \ d(y, z) \in \Lambda_\delta,\ {\rm for\ any}\ y, z \in X$$
Observe that $X_1 = X / \sim$ is a $\Lambda_1$-metric space, where $\Lambda_1 = \Lambda / 
\Lambda_\delta$,
with respect to the metric
$$d_1([y], [z]) = d(y, z) + \Lambda_\delta$$ 
where $[y], [z]$ are the images of $y, z \in X$ in $X_1$. Since $X$ is geodesic, from the definition
of $X_1$ it follows that $X_1$ is also geodesic. Moreover, $X_1$ is $0$-hyperbolic since $\delta \in 
\Lambda_\delta$, and it follows that $X_1$ is a $\Lambda_1$-tree.

The isometry $\gamma$ of $X$ induces an isometry $\gamma_1$ of $X_1$ and we have
$$d_1([x], \gamma_1^2 [x]) \geqslant d_1([x], \gamma_1 [x])$$
Recall that the {\em translation length} $l(\gamma_1)$ of $\gamma_1$ (see, for example, 
\cite{Chiswell:2001}) is defined as
$$l(\gamma_1) = \min\{ d_1([y], \gamma_1 [y]) \mid [y] \in X_1\}$$
According to Lemma 3.1.8 \cite{Chiswell:2001}, for any $[y] \in X_1$
$$l(\gamma_1) = \max\{d_1([y], \gamma_1^2 [y]) - d_1([y], \gamma_1 [y]), 0\}$$
and in particular
$$l(\gamma_1) = \max\{d_1([x], \gamma_1^2 [x]) - d_1([x], \gamma_1 [x]), 0\} = d_1([x], \gamma_1^2 
[x]) - d_1([x], \gamma_1 [x]) \geqslant 0$$

If $l(\gamma_1) > 0$ then $\gamma_1$ is a hyperbolic isometry of $X_1$ which fixes a pair of ends
$a_1, b_1 \in \partial X_1$ of full $\Lambda_1$-type (the ends of the axis of $\gamma_1$). Indeed, 
if $a_1, b_1$ are not of full $\Lambda_1$-type, then $\gamma_1$ fixes a $\Lambda_2$-subspace $X_2$
of $X_1$, where $\Lambda_2 \neq \Lambda_1$. But in this case, $X_2$ lifts to a $\Lambda_3$-subspace
$X_3$ of $X$, where $\Lambda_3 \neq \Lambda$ which is stabilized by $\gamma$ - a contradiction with
minimality of $\gamma$. Hence, $a_1, b_1 \in \partial X_1$ have preimages $a_x, b_x \in \partial X$
which are fixed by $\gamma$ and $\gamma$ is hyperbolic.

If $l(\gamma_1) = 0$ then $\gamma_1$ is either elliptic or an inversion. If $\gamma_1$ is elliptic
then there exists $[y] \in X_1$ such that $\gamma_1 [y] = [y]$, that is, $\gamma$ stabilizes the 
$\Lambda_0$-subspace of $X$ containing $y$ - a contradiction with minimality of $\gamma$. Suppose
$\gamma_1$ is an inversion, that is, $\gamma_1$ does not fix a point in $X_1$ but $\gamma_1^2$ has
a fixed point $[y] \in X_1$. Observe that this is possible only if $\Lambda_1 \neq 2 \Lambda_1$ which
implies that $\Lambda \neq 2 \Lambda$. Now, $\gamma$ does not fix any $\Lambda_\delta$-subspace 
of $X$, but $\gamma^2$ fixes the $\Lambda_\delta$-subspace of $X$ containing $y$. Hence, $\gamma$ 
is an inversion.

\end{proof}

The lemma below is similar to \cite[Lemma 2.3]{CoornaertDelzantPapadopoulos:1990}.

\begin{lemma}
\label{le:1.3.2}
Let $(X, d)$ be a geodesic $\delta$-hyperbolic $\Lambda$-metric space, where and $\delta > 0$. Suppose 
$\gamma_1, \gamma_2$ are isometries of $X$ which are neither hyperbolic, nor inversions, and such 
that for some $x \in X$
$$d(x, \gamma_1 x) \geqslant 2 (\gamma_1 x \cdot \gamma_2 x)_x + 6 \delta,\ \ d(x, \gamma_2 x) 
\geqslant 2 (\gamma_1 x \cdot \gamma_2 x)_x + 6 \delta$$
Then $\gamma_2 \gamma_1$ and $\gamma_1 \gamma_2$ are hyperbolic if $\Lambda = 2 \Lambda$, and
are either hyperbolic, or inversions if $\Lambda \neq 2 \Lambda$.
\end{lemma}
\begin{proof} We follow the scheme of proof of \cite[Lemma 2.3]{CoornaertDelzantPapadopoulos:1990}
and adopt the same terminology: for any isometries $\alpha,\ \beta$ of $X$ denote $|\alpha - \beta| 
= d(\alpha x, \beta x)$ and $|\alpha| = d(x, \alpha x)$.

Since $\gamma_1, \gamma_2$ are neither hyperbolic, nor inversions, by Lemma \ref{le:1.3.1} we have
$$|\gamma_1^2| \leqslant |\gamma_1| + 3\delta,\ |\gamma_2^2| \leqslant |\gamma_2| + 3\delta$$
Next, from
$$|\gamma_1| \geqslant 2 (\gamma_1 x \cdot \gamma_2 x)_x + 6 \delta,\ \ |\gamma_2| \geqslant 
2 (\gamma_1 x \cdot \gamma_2 x)_x + 6 \delta$$
by definition of Gromov product we get
$$|\gamma_1 - \gamma_2| \geqslant |\gamma_1| + 6\delta,\ |\gamma_1 - \gamma_2| \geqslant |\gamma_2| 
+ 6\delta$$
Now we apply the $4$-point condition (see Lemma \ref{le:1.2.6}) to $x,\ \gamma_1 x,\ \gamma_1^2 x$,
and $(\gamma_1 \gamma_2) x$:
$$|\gamma_1| + |\gamma_1^2 - \gamma_1 \gamma_2| \leqslant \max\{|\gamma_1^2| + |\gamma_1 - 
\gamma_1 \gamma_2|,\ |\gamma_1 \gamma_2| + |\gamma_1 - \gamma_1^2|\} + 2\delta$$
or
$$|\gamma_1| + |\gamma_1 - \gamma_2| \leqslant \max\{|\gamma_1^2| + |\gamma_2|,\ |\gamma_1 \gamma_2| 
+ |\gamma_1|\} + 2\delta$$
From $|\gamma_1^2| \leqslant |\gamma_1| + 3\delta$ and $|\gamma_1 - \gamma_2| \geqslant |\gamma_2|
+ 6\delta$ we obtain
$$|\gamma_1| + |\gamma_1 - \gamma_2| \geqslant |\gamma_1^2| + |\gamma_2| + 3\delta$$
which implies that $|\gamma_1| + |\gamma_1 - \gamma_2|$ cannot not be smaller than $|\gamma_1^2| 
+ |\gamma_2| + 2\delta$ and 
$$\max\{|\gamma_1^2| + |\gamma_2|,\ |\gamma_1 \gamma_2| + |\gamma_1|\}  = |\gamma_1 \gamma_2| + 
|\gamma_1|$$
Hence, 
$$|\gamma_1| + |\gamma_1 - \gamma_2| \leqslant |\gamma_1 \gamma_2| + |\gamma_1| + 2\delta$$
or
$$|\gamma_1 - \gamma_2| \leqslant |\gamma_1 \gamma_2| + 2\delta$$
Similar argument produces
$$|\gamma_1 - \gamma_2| \leqslant |\gamma_2 \gamma_1| + 2\delta$$
Combining the above inequalities with $|\gamma_1 - \gamma_2| \geqslant |\gamma_1| + 6\delta,\ 
|\gamma_1 - \gamma_2| \geqslant |\gamma_2| + 6\delta$ we get
$$|\gamma_1 \gamma_2| \geqslant |\gamma_1| + 4\delta,\ \ \ \ |\gamma_2 \gamma_1| \geqslant |\gamma_1| 
+ 4\delta,\ \ \ \ |\gamma_1 \gamma_2| \geqslant |\gamma_2| + 4\delta,$$
$$|\gamma_2 \gamma_1| \geqslant |\gamma_2| + 4\delta$$
Next we apply the $4$-point condition to $x,\ \gamma_1 x,\ (\gamma_1 \gamma_2) x$, and $(\gamma_1 
\gamma_2 \gamma_1) x$:
$$|\gamma_2 \gamma_1| + |\gamma_1 \gamma_2| \leqslant \max\{|\gamma_1 \gamma_2 \gamma_1| + 
|\gamma_2|,\ 2|\gamma_1|\} + 2\delta$$
But we have 
$$|\gamma_2 \gamma_1| + |\gamma_1 \gamma_2| \geqslant 2|\gamma_1| + 8\delta$$
so
$$|\gamma_2 \gamma_1| + |\gamma_1 \gamma_2| \leqslant |\gamma_1 \gamma_2 \gamma_1| + |\gamma_2| + 
2\delta$$
and since $|\gamma_2 \gamma_1| \geqslant |\gamma_2| + 4\delta$, it follows that
$$|\gamma_1 \gamma_2| + 2\delta \leqslant |\gamma_1 \gamma_2 \gamma_1|$$
We combine the above inequality with $|\gamma_1 \gamma_2| \geqslant |\gamma_1| + 4\delta$ and 
$|\gamma_1 \gamma_2| \geqslant |\gamma_2| + 4\delta$ to get
$$|\gamma_2| + 6\delta \leqslant |\gamma_1 \gamma_2 \gamma_1|,\ \ \ |\gamma_1| + 6\delta \leqslant 
|\gamma_1 \gamma_2 \gamma_1|$$
Apply the $4$-point condition to $x,\ (\gamma_1 \gamma_2) x,\ (\gamma_1 \gamma_2 \gamma_1) x$, and
$(\gamma_1 \gamma_2)^2 x$:
$$|\gamma_1 \gamma_2| + |\gamma_1 \gamma_2 \gamma_1| \leqslant \max\{|(\gamma_1 \gamma_2)^2| + 
|\gamma_1|,\ |\gamma_2| + |\gamma_1 \gamma_2|\} + 2\delta$$
From $|\gamma_2| + 6\delta \leqslant |\gamma_1 \gamma_2 \gamma_1|$ we get
$$|\gamma_2| + |\gamma_1 \gamma_2| + 6\delta \leqslant |\gamma_1 \gamma_2 \gamma_1| + |\gamma_1 
\gamma_2|$$
and it follows that
$$\max\{|(\gamma_1 \gamma_2)^2| + |\gamma_1|,\ |\gamma_2| + |\gamma_1 \gamma_2|\} = |(\gamma_1 
\gamma_2)^2| + |\gamma_1|$$
Hence,
$$|\gamma_1 \gamma_2| + |\gamma_1 \gamma_2 \gamma_1| \leqslant |(\gamma_1 \gamma_2)^2| + 
|\gamma_1| + 2\delta$$
and from $|\gamma_1| + 6\delta \leqslant |\gamma_1 \gamma_2 \gamma_1|$ we obtain
$$|\gamma_1 \gamma_2| + 4\delta \leqslant |(\gamma_1 \gamma_2)^2|$$
and
$$|\gamma_1 \gamma_2| + 3\delta < |(\gamma_1 \gamma_2)^2|$$
From Lemma \ref{le:1.3.1} it follows that $\gamma_1 \gamma_2$ is hyperbolic if $\Lambda = 2 
\Lambda$, and it is either hyperbolic, or an inversion if $\Lambda \neq 2 \Lambda$. The argument for 
$\gamma_2 \gamma_1$ is similar.
\end{proof}

Now, using the above lemmas, we are ready classify minimal isometries of a geodesic 
$\delta$-hyperbolic $\Lambda$-metric space.

\begin{theorem}
\label{th:Lambda_isom}
Let $(X,d)$ be a a geodesic $\delta$-hyperbolic $\Lambda$-metric space. Then every minimal isometry 
of $X$ is either elliptic, or parabolic, or hyperbolic in the case when $\Lambda = 2 \Lambda$, and is 
either elliptic, or parabolic, or hyperbolic, or an inversion when $\Lambda \neq 2 \Lambda$.
\end{theorem}
\begin{proof}  If $\delta = 0$ then $X$ is a $\Lambda$-tree and any isometry of $X$ (not necessarily
a minimal one) is either hyperbolic, or elliptic, or an inversion (see, \cite{Alperin_Bass:1987, 
Chiswell:2001}).

\smallskip

Suppose $\delta > 0$ and let $\gamma$ be a minimal isometry of $X$. Suppose $\gamma$ is neither 
elliptic, nor parabolic. It follows that for any $x \in X$ the diameter of the set $\{\gamma^n x 
\mid n \in \Z\}$ is not bounded by $K\delta$ for any $K \in \N$.

\smallskip

Next, suppose for any $x \in X$, the diameter of the set $\{\gamma^n x \mid n \in \Z\}$ is bounded
by some $\alpha \in \Lambda$. We would like to show that $\gamma$ is an inversion in this case. 
Indeed, observe that the minimal convex subgroup $\Lambda' \subseteq \Lambda$ containing $\alpha$ 
must coincide with $\Lambda$ (otherwise $\gamma$ stabilizes a $\Lambda'$-subspace of $X$ and this 
is a contradiction with minimality of $\gamma$). Next, by our assumption $\Lambda_\delta \neq 
\Lambda$. Define an equivalence relation ``$\sim$'' on $X$ by setting
$$y \sim z\ \ \ \ \Longleftrightarrow\ \ \ \ d(y, z) \in \Lambda_\delta,\ {\rm for\ any}\ y, z \in X$$
Hence, $X_1 = X / \sim$ is a $\Lambda_1$-metric space, where $\Lambda_1 = \Lambda / \Lambda_\delta$,
with respect to the metric
$$d_1([y], [z]) = d(y, z) + \Lambda_\delta$$ 
where $[y], [z]$ are the images of $y, z \in X$ in $X_1$. Since $X$ is geodesic, from the definition
of $X_1$ it follows that $X_1$ is also geodesic. Moreover, $X_1$ is $0$-hyperbolic since $\delta \in 
\Lambda_\delta$, and it follows that $X_1$ is a $\Lambda_1$-tree. 

The isometry $\gamma$ of $X$ induces an isometry $\gamma_1$ of $X_1$ and the diameter of the set 
$\{\gamma_1^n [x] \mid n \in \Z\}$ is bounded by $\alpha + \Lambda_\delta$. Consider the translation 
length $l(\gamma_1)$ of $\gamma_1$. If $l(\gamma_1) > 0$ then $\gamma_1$ is hyperbolic and the 
diameter of the set $\{\gamma_1^n [x] \mid n \in \Z\}$ cannot be bounded by any $\beta \in \Lambda_1$. 
Hence, $\gamma_1$ is either an inversion, or elliptic. If $\gamma_1$ is elliptic then it fixes a 
point $[y] \in X_1$ which implies that $\gamma$ stabilizes a $\Lambda_\delta$-subspace of $X$ - a 
contradiction with minimality of $\gamma$. Hence, $\gamma_1$ is an inversion (which is possible 
only if $\Lambda \neq 2 \Lambda$) and it follows that $\gamma$ is also an inversion.

\smallskip

Finally, suppose that for any $x \in X$, the diameter of the set $\{\gamma^n x \mid n \in \Z\}$ is 
not bounded by any $\alpha \in \Lambda$. It follows that there exists a sequence of integers $\{n_i\}$ 
such that $d(x, \gamma^{n_i} x) \to \infty$. Hence, there exists $a \in \partial X$ such that 
$\{\gamma^{n_i} x\} \to a$. Since we assume that $\gamma$ is not parabolic, there must be at least 
one other sequence of integers $\{m_j\}$ such that $d(x, \gamma^{m_j} x) \to \infty$ and a point 
$b \in \partial X$ such that $\{\gamma^{m_j} x\} \to b$ such that $a \neq b$. 

Since $a \neq b$, it follows that the Gromov product $(a \cdot b)_x$ of $a$ and $b$ is finite. At the
same time $d(x, \gamma^{n_i} x) \to \infty,\ \ d(x, \gamma^{m_j} x) \to \infty$, so, there exist $N \in
\{n_i\}$ and $M \in \{m_j\}$ such that $N \neq M$ and
$$d(x, \gamma^N x) \geqslant 2(\gamma^N x \cdot \gamma^M x)_x + 6\delta,\ \ \ d(x, \gamma^M x) 
\geqslant 2(\gamma^N x \cdot \gamma^M x)_x + 6\delta$$
By Lemma \ref{le:1.3.2}, the isometry $\gamma^{N - M}$ is hyperbolic if $\Lambda = 2 \Lambda$, and
is either hyperbolic, or an inversion if $\Lambda \neq 2 \Lambda$. Hence, the required statement for
$\gamma$ follows.
\end{proof}

The theorem above immediately can be applied in the case when $\Lambda = \R^n,\ \Z^n$.

\begin{theorem}
\label{th:R^n_isom}
Let $(X,d)$ be a geodesic $\delta$-hyperbolic $\R^n$-metric space, where $\R^n$ is taken with the 
right lexicographic order. Then every minimal isometry of $X$ is either elliptic, or parabolic, or 
hyperbolic.
\end{theorem}

\begin{theorem}
\label{th:Z^n_isom}
Let $(X,d)$ be a geodesic $\delta$-hyperbolic $\Z^n$-metric space, where $\Z^n$ is taken with the 
right lexicographic order. Then every minimal isometry of $X$ is either elliptic, or parabolic, or 
hyperbolic, or an inversion.
\end{theorem}

Using Theorem \ref{th:Lambda_isom} we can give a nice characterization of hyperbolic isometries in
the case when $\Lambda$ is either $\R^n$, or $\Z^n$.

\begin{corollary}
\label{co:character_hyp}
Let $(X,d)$ be a geodesic $\delta$-hyperbolic $\Lambda$-metric space, where $\Lambda$ is either 
$\R^n$, or $\Z^n$ with the right lexicographic order and $\delta = (\delta_0, 0, \ldots, 0)$. Let 
$\gamma$ be a minimal isometry of $X$. If $n > 1$ then $\gamma$ can be only either hyperbolic,
or an inversion. Moreover, for any $n$, $\gamma$ is hyperbolic if and only if there exist $x \in X$ 
and $\lambda, c \in \Lambda$ such that $ht(\lambda) = ht(d(x,\gamma x))$ and $d(x, \gamma^k x) 
\geqslant k \lambda + c$ for any $k \in \N$.
\end{corollary}
\begin{proof} If $n = 1$ then there exists $x \in X$ such that $k \to \gamma^k x$ is a quasi-isometry
(see, for example, \cite{CoornaertDelzantPapadopoulos:1990}) of $\Z$ into $X$. Hence, there exist 
$\lambda, c \in \Lambda$ (which is either $\R$, or $\Z$ in this case) such that $d(x, \gamma^k x) 
\geqslant k \lambda + c$ for any $k \in \N$.

\smallskip

Suppose $n > 1$. By Theorem \ref{th:Lambda_isom}, $\gamma$ is either hyperbolic, or elliptic, or 
parabolic or an inversion, so, consider all these possibilities. Recall that $\Lambda_\delta$ is the 
minimal convex subgroup of $\Lambda$ containing $\delta$ (in our case, $\Lambda_\delta = \R$, or 
$\Lambda_\delta = \Z$). Define $X_1 = X / \sim$, where $x \sim y$ if $d(x, y) \in \Lambda_\delta$. 
Since $X$ is geodesic, $X_1$ is a $\Lambda_1$-tree (here, either $\Lambda_1 = \R^{n-1}$, or 
$\Lambda_1 = \Z^{n-1}$). Observe that $\gamma$ induces an isometry $\gamma_1$ of $X_1$ which can be
either elliptic, or hyperbolic, or an inversion if $\Lambda_1 = \Z^{n-1}$.

If $\gamma$ is elliptic then it fixes a $\Lambda_0$-subspace - a contradiction with minimality.

If $\gamma$ is parabolic then $\gamma_1$ cannot be elliptic because then $\gamma$ is not minimal. 
$\gamma_1$ cannot be hyperbolic because in this case $\gamma_1$ fixes two distinct points on the 
boundary $\partial X_1$ (since $\gamma$ is minimal) which can be lifted to two distinct points on the
boundary $\partial X$ fixed by $\gamma$ - a contradiction since we assume that $\gamma$ is parabolic.
Eventually, if $\gamma_1$ is an inversion then $\gamma_1$ fixes a point in $X_1$, that is, $\gamma^2$
fixes a $\Lambda_\delta$-subspace of $X$. But since $\gamma$ is parabolic, $\gamma^2$ is also parabolic
and the diameter of $\{\gamma^2 x\}$ is unbounded by any $\alpha \in \Lambda$ for any $x \in X$ - a
contradiction. 

Hence, we can conclude that $\gamma$ can be neither elliptic, nor parabolic if $n > 1$. That is, it
can be only either hyperbolic, or an inversion.

\smallskip

Finally, in the case when $n > 1$, the isometry $\gamma_1$ is either hyperbolic, or an inversion. 
Moreover (see \cite{Chiswell:2001}), $\gamma_1$ is hyperbolic if and only if there exists $x_1 \in 
X_1$, which belongs to the axis of $\gamma_1$, and $\lambda_1, c_1 \in \Lambda_1$ such that $d(x_1, 
\gamma_1^k x_1) \geqslant k \lambda_1 + c_1$ for any $k \in \N$. So, $x_1, \lambda_1$, and $c_1$ 
can be lifted back to $X$ and $\Lambda$ respectively, and we get the required result for $\gamma$.
\end{proof}

Finally, we conclude this section with an investigation of the behavior of non-minimal isometries in
the case when $\Lambda$ is either $\R^n$, or $\Z^n$.

\begin{prop}
\label{class_is_unique}
Let $(X,d)$ be a geodesic $\delta$-hyperbolic $\R^n$-metric space, where $\R^n$ is taken with the
right lexicographic order and $\delta = (\delta_0, 0, \ldots, 0)$. Let $\gamma$ be a non-minimal 
isometry of $X$ fixing two distinct $\R^{n-i}$-subspaces $X_0$ and $X_1$, where $i \in [1,n-1]$. 
Then the action of $\gamma$ on $X_0$ and $X_1$ is of the same type.
\end{prop}
\begin{proof} Since $X$ is geodesic, there exist unique $\alpha_0 \in \partial X_0$ and $\alpha_1
\in \partial X_1$ such that if $\{x_k\} \to \alpha_0,\ \{y_k\} \to \alpha_1$ and $x \in X_0,\ y \in
X_1$ then
$$(x_k \cdot y)_x \to \infty,\ \ (y_k \cdot x)_y \to \infty\ \ {\rm with\ respect\ to}\ \R^{n-i}$$
It follows that $\gamma \alpha_0 = \alpha_0$ and $\gamma \alpha_1 = \alpha_1$.

Suppose $\gamma|_{X_0}$ is elliptic and $x$ is such that for any $k \in \Z$ we have $d(x, \gamma^k
x) \leqslant M \delta$ for some $M \in \N$. If $\gamma|_{X_1}$ is hyperbolic, it follows that either 
$\{\gamma^k y\} \to \alpha_1$, or $\{\gamma^{-k} y\} \to \alpha_1$. Without loss of generality assume 
that $\{\gamma^k y\} \to \alpha_1$. We have
$$d(x, y) = d(\gamma^k y, \gamma^k x) = d(\gamma^k y, y) + d(\gamma^k x, y) - 2 (\gamma^k y \cdot
\gamma^k x)_y$$
$$\leqslant d(\gamma^k y, y) + d(x, y) + d(\gamma^k x, x) - 2(\gamma^k y \cdot \gamma^k x)_y$$
$$\leqslant d(\gamma^k y, y) + d(x, y) + M\delta - 2(\gamma^k y \cdot \gamma^k x)_y$$
Next, since $X_0 \neq X_1$, we have $(\gamma^k y \cdot x)_y \in \R^{n-i}$ and $(\gamma^k x \cdot
x)_y \in \R^{n-j}$ for some $j < i$. It follows that $(\gamma^k y \cdot x)_y < (\gamma^k x \cdot x)_y$
and from
$$(\gamma^k y \cdot \gamma^k x)_y \geqslant \min\{(\gamma^k y \cdot x)_y, (\gamma^k x \cdot x)_y\}
- \delta$$
we get
$$(\gamma^k y \cdot \gamma^k x)_y \geqslant (\gamma^k y \cdot x)_y - \delta$$
or
$$-2(\gamma^k y \cdot \gamma^k x)_y \leqslant -2(\gamma^k y \cdot x)_y + 2\delta$$
Thus,
$$d(x,y) \leqslant d(\gamma^k y, y) + d(x, y) + M \delta - 2(\gamma^k y \cdot \gamma^k x)_y$$
$$\leqslant d(\gamma^k y, y) + d(x, y) + (M + 2) \delta - 2(\gamma^k y \cdot x)_y$$
and eventually we obtain
\begin{equation}
\label{eq:class_is_unique}
0 \leqslant d(\gamma^k y, y) + (M + 2) \delta - 2(\gamma^k y \cdot x)_y
\end{equation}
for any $k$.

Using a similar argument but starting with $d(\gamma^k x, y) \geqslant d(x,y) - d(x, \gamma^k x)$, 
we can eventually obtain that
\begin{equation}
\label{eq:class_is_unique_2}
0 \geqslant d(\gamma^k y, y) - (M + 2) \delta - 2(\gamma^k y \cdot x)_y
\end{equation}
for any $k$. Indeed we have
$$d(x, y) = d(\gamma^k y, \gamma^k x) = d(\gamma^k y, y) + d(\gamma^k x, y) - 2 (\gamma^k y \cdot
\gamma^k x)_y$$
$$\geqslant d(\gamma^k y, y) + d(x, y) - d(\gamma^k x, x) - 2(\gamma^k y \cdot \gamma^k x)_y$$
$$\geqslant d(\gamma^k y, y) + d(x, y) - M \delta - 2(\gamma^k y \cdot \gamma^k x)_y$$
Next, since $X_0 \neq X_1$, we have $(\gamma^k y \cdot \gamma^k x)_y \in \R^{n-i}$ and $(\gamma^k x 
\cdot x)_y \in \R^{n-j}$ for some $j < i$. It follows that $(\gamma^k y \cdot \gamma^k x)_y < 
(\gamma^k x \cdot x)_y$ and from
$$(\gamma^k y \cdot x)_y \geqslant \min\{(\gamma^k y \cdot \gamma^k x)_y, (\gamma^k x \cdot x)_y\}
- \delta$$
we get
$$(\gamma^k y \cdot x)_y \geqslant (\gamma^k y \cdot \gamma^k x)_y - \delta$$
or
$$-2(\gamma^k y \cdot \gamma^k x)_y \geqslant -2(\gamma^k y \cdot x)_y - 2\delta$$
Thus,
$$d(x,y) \geqslant d(\gamma^k y, y) + d(x, y) - M \delta - 2(\gamma^k y \cdot \gamma^k x)_y$$
$$\geqslant d(\gamma^k y, y) + d(x, y) - (M + 2) \delta - 2(\gamma^k y \cdot x)_y$$
from which we obtain (\ref{eq:class_is_unique_2}).

\smallskip

Now, we assume $(\gamma^k y \cdot \gamma^{k+1} y)_y \leqslant (\gamma^{k+1} y \cdot 
x)_y$ and deduce
$$d(y, \gamma^{k+1} y) \leqslant 2d(y, \gamma y) + (M + 2) \delta$$ 
From our assumption we get
$$-2(\gamma^k y \cdot \gamma^{k+1} y)_y \geqslant - 2(\gamma^{k+1} y \cdot x)_y$$
Next, we have
$$d(y, \gamma y) = d(\gamma^k y, \gamma^{k+1} y) = d(y, \gamma^k y) + d(y, \gamma^{k+1} y) -
2(\gamma^k y \cdot \gamma^{k+1} y)_y$$
$$\geqslant 2d(y, \gamma^{k+1} y) - d(y, \gamma y) - 2(\gamma^k y \cdot \gamma^{k+1} y)_y$$
where the latter inequality follows from the triangle inequality. We can rewrite the latter inequality 
in the form
$$2d(y, \gamma y) \geqslant 2d(y, \gamma^{k+1} y) - 2(\gamma^k y \cdot \gamma^{k+1} y)_y$$
so, combining it with $-2(\gamma^k y \cdot \gamma^{k+1} y)_y \geqslant - 2(\gamma^{k+1} y \cdot 
x)_y$ we obtain
$$2d(y, \gamma y) \geqslant 2d(y, \gamma^{k+1} y) - 2(\gamma^{k+1} y \cdot x)_y$$
Eventually, since
$$d(\gamma^{k+1} y, y) - 2(\gamma^{k+1} y \cdot x)_y \geqslant - (M + 2) \delta$$
(we replaced $k$ by $k+1$ in (\ref{eq:class_is_unique})), it follows that
$$2d(y, \gamma y) \geqslant d(y,\gamma^{k+1} y) - (M + 2) \delta$$
or
$$d(y, \gamma^{k+1} y) \leqslant 2d(y, \gamma y) + (M + 2) \delta$$
Observe that the latter inequality gives a contradiction since we assume that $\gamma$ acts as a 
hyperbolic isometry on $X_1$. It follows that the inequality 
$$(\gamma^н y \cdot \gamma^{k+1} y)_y \leqslant (\gamma^{k+1} y \cdot x)_y$$ 
cannot hold for arbitrarily large $k$ and there exists $N \in \N$ such that $(\gamma^k y \cdot 
\gamma^{k+1} y)_y > (\gamma^{k+1} y \cdot x)_y$ for any $k > N$.

It implies that 
$$(\gamma^k y\cdot x)_y \geqslant \min\{(\gamma^k y \cdot \gamma^{k+1} y)_y, (\gamma^{k+1} y \cdot
x)_y\} - \delta = (\gamma^{k+1} y \cdot x)_y - \delta$$ 
Therefore, there exists $L \in \R^{n-i}$ such that $(\gamma^k y \cdot x)_y \leqslant L + k \delta$ 
for any $k > N$.

However, since $\gamma$ acts hyperbolically on $X_1$, by Corollary \ref{co:character_hyp}, there 
exist $\lambda, c \in \R^{n-i}$ such that $d(y, \gamma^k y) > k \lambda + c$. We can assume that 
$\lambda > 5 \delta$ since we can replace $\gamma$ with $\gamma' = \gamma^i$ for $i$ large enough 
so that $\lambda' > 5 \delta$. However, according to (\ref{eq:class_is_unique_2}) we have
$$0 \geqslant d(\gamma^k y, y) - (M + 2) \delta - 2(\gamma^k y \cdot x)_y \geqslant d(\gamma^k y, y) 
- (M + 2) \delta - 2 L - 2 k \delta$$
which implies that
$$0 > k \lambda + c - (M + 2) \delta - 2 L - 2 k \delta > 5 k \delta + c - (M + 2) \delta - 2 L - 
2 k \delta$$
$$ = 3 k \delta - (2 L + (M + 2) \delta - c)$$
Since $2 L + (M + 2) \delta - c$ is a constant, we have a contradiction.

\smallskip

The same argument can be used to prove that if $\gamma|_{X_0}$ is elliptic, then $\gamma|_{X_1}$
cannot be parabolic since then both $\{\gamma^k y\}$ and $\{\gamma^{-k} y\}$ converge to the
same point on the boundary while $\{\gamma^k x \mid k \in \Z\}$ stays within a fixed distance from
$x$.

\smallskip

We can use a similar argument to show that, if $\gamma|_{X_0}$ is hyperbolic, then $\gamma|_{X_1}$ 
cannot be parabolic. Indeed, if $\gamma|_{X_1}$ is parabolic then it has a unique fixed point in 
$\partial X_1$ which both $\{\gamma^k y\}$ and $\{\gamma^{-k} y\}$ converge to. At the same time, 
$\gamma|_{X_0}$ has two fixed points in $\partial X_0$: $\{\gamma^k x\}$ converges to one of them 
and $\{\gamma^{-k} x\}$ to the other one. Suppose, without loss of generality, that $\{\gamma^k x\}$ 
converges to $\alpha_0$ and $\{\gamma^k y\}$ converges to $\alpha_1$. The we have
$$d(x,y) = d(\gamma^k x, \gamma^k y) = d(x,y) + d(x, \gamma^k x) + d(y, \gamma^k y) - 2(\gamma^k x 
\cdot y)_x - 2(\gamma^k y \cdot x)_y$$
which implies that
$$0 = d(x, \gamma^k x) + d(y, \gamma^k y) - 2(\gamma^k x \cdot y)_x - 2(\gamma^k y \cdot x)_y$$
Finally, we use the obtained equality as a analog of (\ref{eq:class_is_unique}) and repeat the argument 
given above for both $(\gamma^k y \cdot \gamma^{k+1} y)_y$ and $(\gamma^k x \cdot \gamma^{k+1} x)_x$.
\end{proof}

\begin{corollary}
\label{true_for_Z^n}
Let $(X,d)$ be a geodesic $\delta$-hyperbolic $\Z^n$-metric space, where $\Z^n$ is taken with the 
right lexicographic order and $\delta = (\delta_0, 0, \ldots, 0)$. Let $\gamma$ be a non-minimal 
isometry of $X$ fixing $\Z^{n-i}$-subspaces $X_0$ and $X_1$. Then the action of $\gamma$ on $X_0$ 
and $X_1$ is of the same type.
\end{corollary}
\begin{proof} If the actions of $\gamma$ on $X_0$ and $X_1$ are either elliptic, or hyperbolic, or 
parabolic then the proof is a straightforward adaptation of that of Proposition \ref{class_is_unique}.

Suppose that $\gamma$ is an inversion on $X_0$. Observe that by definition $\gamma$ does not fix 
any $\Z$-subspace of $X$.

Let $Y$ be the $\Z^{n-1}$-tree obtained by contracting all $\Z$-subspaces of $X$ to points (more
precisely, $Y = X / \sim$, where $x \sim y$ if and only if $d(x, y) \in \Z$, and since $X$ is geodesic,
$Y$ is a geodesic $0$-hyperbolic $\Z^{n-1}$-metric space). Let $Y_0$ and $Y_1$ be the subtrees of
$Y$ corresponding to $X_0$ and $X_1$. Observe that $\gamma$ induces an isometry $\gamma_1$ of $Y$
such that $\gamma_1$ fixes both $Y_0$ and $Y_1$ and $\gamma_1$ acts on $Y_0$ as an inversion.
Hence, let $a, b \in Y_0$ such that $\gamma_1 a = b,\ \gamma_1 b = a$ and take an arbitrary $z \in 
Y_1$. If $[a, b] \cup [b, z]$ is a geodesic in $Y$ then so is $\gamma_1 [b, z] = [a, \gamma_1 z]$ 
and $b \notin [a, \gamma_1 z]$. But the unique geodesic from $a$ to any element of $Y_1$ must 
contain $b$, hence, a contradiction. If $[b, a] \cup [a, z]$ is a geodesic then we get a contradiction
in a similar way. Finally, if there is some $c \in [a, b]$ such that $[a, c] \cup [c, z]$ and $[b, c] 
\cup [c, z]$ are both geodesics then we get a contradiction by using similar considerations with $c$ 
and $\gamma_1 c$. 

It follows that $\gamma_1$ cannot fix distinct $Y_0$ and $Y_1$ and the same applies to $\gamma$
in $X$.
\end{proof}

\subsection{Examples}
\label{subs:examples}

\begin{example}
\label{ex:hyperbolic_1a}
Let $X$ be a proper, geodesic $\delta$-hyperbolic $\R$-metric space, $\ast \in X$ and $a,b \in 
\partial X$ such that $(a \cdot b)_\ast = 0$. Let $\{X_i \mid i \in \Z\}$ be a set of copies of $X$
with the copy of $x \in X$ in $X_i$ denoted $x_i$, and define $Y = \bigcup X_i$

Let us define a metric $d$ on $Y$. By abuse of notation, we are also going to use $d$ for the metric 
in $X$. For any $i \in \Z$, define $d(x_i, y_i) = (d(x, y), 0)$. If $i < j$ then define
$$d(x_i, y_j) = (d(x, \ast) + d(y, \ast) - 2(x \cdot a)_\ast - 2(y \cdot b)_\ast, |i - j|)$$
First of all, $(Y, d)$ is a $(8 \delta, 0)$-hyperbolic metric space. To see that, take $\rho$ to be a
geodesic line joining $a$ and $b$ such that $\ast \in \rho$ (such a geodesic line exists since $X$
is geodesic). Define $\rho_i$ to be the image of $\rho$ in $X_i$, $a_i$ and $b_i$ to be the images 
of $a$ and $b$ in $\partial X_i$, and $[x,\omega)$ to be a geodesic ray between some $x \in X$ and 
$\omega \in \partial X$.

For any $x_i, y_j \in Y$ with $i \neq j$, let 
$$[x_i, a_i) \cup \rho_{i+1} \cup \cdots \cup \rho_{j-1} \cup (b_j, y_j]$$ 
will be a geodesic embedding of $[0, d(x_i, y_j)]$ into $Y$. It is then easy to prove that $(Y, d)$ 
is $(8 \delta, 0)$-hyperbolic by using a geometric argument.

Let now $\gamma$ be an isometry of $X$ which preserves $a$ and $b$. We would like to extend it to 
a mapping $\overline{\gamma} : Y \to Y$ by $\overline{\gamma}(x_i) = (\gamma x)_{i+1}$. For every 
$x_i, y_j \in Y$, consider the first component of $d(x_i, y_j)$ which we denote by $D(x, y)$. Explicitly,
$$D(x, y) = d(x, \ast) + d(y, \ast) - 2(x \cdot a)_\ast - 2(y \cdot b)_\ast$$ 
It is easy to see that $\overline{\gamma}$ is an isometry of $Y$ if and only if $D(x, y) = D(\gamma 
x, \gamma y)$ for any $x, y \in X$. We have
$$D(x,y) = d(x, \ast) + d(y, \ast) - 2 (x \cdot a)_\ast - 2 (y \cdot b)_\ast$$
$$= d(x, \ast) + d(y, \ast) - \sup_{\substack{a_i \to a\\ b_i \to b}} \lim_{i \to \infty} d(x, \ast) 
+ d(a_i, \ast) - d(x, a_i) + d(y, \ast) + d(b_i, \ast) - d(y, b_i)$$
$$= -\sup_{\substack{a_i \to a\\ b_i \to b}} \lim_{i \to \infty} d(a_i, \ast) - d(x, a_i) + d(b_i, 
\ast) - d(y, b_i)$$
Recall now that $\gamma$ fixes $a$ and $b$, so we have that $\{a_i \mid a_i \to a\} = \{\gamma a_i 
\mid a_i \to a\}$ and $\{b_i \mid b_i \to b\} = \{\gamma b_i \mid b_i \to b\}$. Hence,
$$D(x,y) - D(\gamma x, \gamma y) = \sup_{\substack{a_i \to a\\ b_i \to b}} \lim_{i \to \infty} 
(d(a_i, \ast) - d(\gamma x, a_i) + d(b_i, \ast) - d(\gamma y, b_i))$$
$$- \sup_{\substack{a_i \to a\\ b_i \to b}} \lim_{i \to \infty} (d(a_i, \ast) - d(x, a_i) + d(b_i, 
\ast) - d(y, b_i))$$
$$= \sup_{\substack{a_i \to a\\ b_i \to b}} \lim_{i \to \infty} (d(\gamma a_i, \ast) - d(\gamma x, 
\gamma a_i) + d(\gamma b_i, \ast) - d(\gamma y, \gamma b_i))$$
$$- \sup_{\substack{a_i \to a\\ b_i \to b}} \lim_{i \to \infty} (d(a_i, \ast) - d(x, a_i) + d(b_i, 
\ast) - d(y, b_i))$$
$$\leqslant \sup_{\substack{a_i \to a\\ b_i \to b}} \lim_{i \to \infty} d(\gamma a_i, \ast) + 
d(\gamma b_i, \ast) - d(a_i, \ast) - d(b_i, \ast)$$
$$= \sup_{\substack{a_i \to a\\ b_i \to b}} \lim_{i \to \infty} d(\gamma a_i, \ast) + d(\gamma b_i, 
\ast) - d(\gamma a_i, \gamma \ast) - d(\gamma b_i, \gamma \ast)$$
$$\leqslant \sup_{\substack{a_i \to a\\ b_i \to b}} \lim_{i \to \infty} d(\gamma a_i, \ast) + 
d(\gamma b_i, \ast) - d(\gamma a_i, \gamma b_i)$$
$$= 2 (\gamma a \cdot \gamma b)_\ast = 0$$
We use a similar reasoning to prove that $D(\gamma x, \gamma y) - D(x,y) \leqslant 0$, which 
implies that $D(\gamma x, \gamma y) = D(x,y)$.

\smallskip

Let us reuse the same spaces $X$ and $Y$ but assume $\gamma b = a$ and $\gamma a = b$.
Since $\gamma^2$ has more fixed points on the boundary than $\gamma$, it is easy to see that
$\gamma$ is an elliptic isometry of $X$.

This time, we extend $\gamma$ to $Y$ by using $\overline{\gamma}(x_i) = (\gamma x)_{-i}$. We can 
use the same argument as above to prove that $\overline{\gamma}$ is an elliptic isometry of $Y$ by 
using the fact that $\{a_i \mid a_i \to a\} = \{\gamma b_i \mid b_i \to b\}$ and $\{b_i \mid b_i \to 
b\} = \{\gamma a_i \mid a_i \to a\}$, allowing an analog of the previous computations.
\end{example}

\begin{example}
\label{ex:hyperbolic_2}
Let $X$ be a $\delta$-hyperbolic metric space in $\Lambda_1$, $T$ a $\Lambda_2$-tree, $d_X$ and 
$d_T$ the associated metrics and $Y \subseteq X$ bounded and $\gamma$ its diameter. Define $d_Y : 
X \times T \to \Lambda_1 \oplus \Lambda_2$ by 
$$d_Y((x_1, t),\ (x_2, t)) = (d_X(x_1, x_2),\ 0)$$ 
and 
$$d_Y((x_1, t_1),\ (x_2, t_2)) = (d_X(x_1, Y) + d_X(x_2, Y),\ d_T(t_1, t_2))$$ 
To see that $d_Y$ is a $(\delta+\gamma,0)$-hyperbolic metric, notice that $(X \times T, d_Y)$ can be 
embedded into a space where we connect each pair $X \times t_1,\ X \times t_2$, where $d_T(t_1, t_2) 
= 1$ by attaching a copy of $Y \times \Lambda_2$ to $Y \times t_1$ and $Y \times t_2$ and imagining 
they meet at the end.

In particular, suppose $G$ acts on $X$ and $T$ and there exists an $x \in X$ such that $G x$ is bounded. 
We can then consolidate both actions into an action on $(X \times T,\ d_{Gx})$ which is a $(\delta + 
\gamma, 0)$-hyperbolic $\Lambda_1 \oplus \Lambda_2$-metric space.

Note that one could use $T$ as a $\delta'$-hyperbolic space resulting in $d_Y$ being $(\delta + 
\gamma,\delta')$-hyperbolic.
\end{example}

\section{Group actions and hyperbolic length\\ functions}
\label{subs:group_actions}

In this section we introduce hyperbolic length functions on groups. This gives an equivalent approach 
to study group actions on hyperbolic $\Lambda$-metric spaces.

In \cite{Lyndon:1963} Lyndon introduced a notion of an abstract length function $l:G \to \Lambda$ 
on a group $G$ with values in $\Lambda$. This started  the whole study of the group length functions 
and actions. Following Lyndon we call a function $l:G \to \Lambda$ a {\em length function} if it 
satisfies the following axioms
\begin{enumerate}
\item[($\Lambda 1$)] $\forall\ g \in G:\ l(g) \geqslant 0$ and $l(1) = 0$,
\item[($\Lambda 2$)] $\forall\ g \in G:\ l(g) = l(g^{-1})$,
\item[($\Lambda 3$)] $\forall\ g,h \in G:\ l(g h) \leqslant l(g) + l(h)$.
\end{enumerate}

Now we introduce the following crucial definition.

A length function $l : G \to \Lambda$ is called {\em hyperbolic} if there is $\delta \in \Lambda$ 
such that
\begin{enumerate}
\item[($\Lambda 4, \delta$)] $\forall\ f, g, h \in G:\ c(f,g) \geqslant \min\{c(f,h),c(g,h)\} - \delta$,
\end{enumerate}

where $c(g,h) = \frac{1}{2}\left(l_v(g) + l_v(h) - l_v(g^{-1} h)\right)$.

Usually, a length function satisfying ($\Lambda 4, \delta$), is called {\em $\delta$-hyperbolic}.

Lyndon himself considered a much stronger form of the axiom ($\Lambda 4, \delta$), the one with 
$\delta = 0$. After him length functions $l : G \to \Lambda$ are called {\em Lyndon length functions}. 
In our terminology these are $0$-hyperbolic length functions. Chiswell in \cite{Chiswell:2001} showed 
that groups with Lyndon length functions $L : G \to \R$ (and an extra axiom) are precisely those 
that act freely on $\R$-trees, and later Morgan and Shalen generalized his construction to arbitrary 
$\Lambda$ \cite{Morgan_Shalen:1984}. For more details we refer to the book \cite{Chiswell:2001}.

In Section \ref{subsec:actions-length} we show how an action of a group $G$ by isometries on a 
(hyperbolic) $\Lambda$-metric space induces naturally a (hyperbolic) length function on $G$ with 
values in $\Lambda$. And in Section \ref{subsec:length-actions} we prove the converse, thus 
establishing equivalence of these two approaches.

\subsection{From actions - to length functions}
\label{subsec:actions-length}

Let $X = (X, d)$ be a $\Lambda$-metric space. By $Isom(X)$ we denote the group of bijective isometries 
of $X$. We say that a group $G$ acts on a $X$ if for any $g \in G$ there is an isometry $\phi_g \in 
Isom(X)$ such that for any $x \in X$ and any $g,h \in G$ one has $\phi_{gh}(x) = \phi_g(\phi_h(x))$, 
that is, the map $g \to \phi_g$ is a group homomorphism $G \to Isom(X)$.  In this case, for every 
$x \in X$ and $g \in G$ we denote $\phi_g(x)$ by $g x$.

If $G$ acts on $(X,d)$ then one  can fix a point $v \in X$ and consider a function $l_v : G \to 
\Lambda$ defined as $l_v(g) = d(v, g v)$, called  a {\em length function based at $v$}. The basic 
properties of based length functions come from the metric properties of $(X,d)$.

\begin{theorem}
\label{le:1}
If a group $G$ acts on a $\Lambda$-metric space $(X, d)$ and $v \in X$ then the length function
$l_v$ based at $v$ is a length function on $G$ with values in $\Lambda$. Moreover, if $(X, d)$ is 
$\delta$-hyperbolic with respect to $v$ for some $\delta \in \Lambda$ then $l_v$ is $\delta$-hyperbolic.
\end{theorem}
\begin{proof} ($\Lambda 1$) is obvious since $l_v(1) = d(v,v) = 0$. Also, ($\Lambda 2$) follows since 
$d(v, gv) = d(g^{-1} v, v)$.

Next, since $d(v, (gh)v) = d(g^{-1}v, hv)$ then by definition of the metric we have
$$d(g^{-1}v, hv) \leqslant d(g^{-1}v, v) + d(v, hv)$$
and ($\Lambda 3$) follows from the equality $d(g^{-1}v, v) = d(v, gv)$.

Finally, assume that $(X,d)$ is $\delta$-hyperbolic. Observe that
$$c(f, g) = \frac{1}{2}\left(l_v(f) + l_v(g) - l_v(f^{-1} g)\right) = \frac{1}{2}(d(v, fv) + d(v, gv) 
- d(v, (f^{-1}g) v))$$
$$= \frac{1}{2}\left(d(v, fv) + d(v, gv) - d(fv, gv)\right) = (fv \cdot gv)_v.$$
In the same way we have $c(f,h) = (fv \cdot hv)_v,\ c(g,h) = (gv \cdot hv)_v$ and since $(X,d)$ is 
$\delta$-hyperbolic then we have
$$(f v \cdot g v)_v \geqslant \min \{(f v \cdot h v)_v, (g v \cdot h v)_v\} - \delta,$$
which proves ($\Lambda 4$) for $l_v$.
\end{proof}

\subsection{From length functions - to actions}
\label{subsec:length-actions}

Let $l:G \to \Lambda$ be a length function. The set
$$\ker(l) = \{ g \in G\ |\ l(g) = 0 \}$$
is called the {\em kernel} of $l$. It is easy to see that $\ker(l)$ is a subgroup of $G$ (this follows
from ($\Lambda 3$)).

\begin{lemma} 
\label{le:ker-cosets}
Let $l:G \to \Lambda$ be a length function. Then for any $a \in \ker(l), g \in G$
$$l(a g) = l(g a) = l(g),$$
that is, $l$ is a constant function on each coset of $\ker(l)$.
\end{lemma}
\begin{proof}
Let  $a \in \ker(l), g \in G$. Then $l(ag) \leqslant l(a) + l(g) = l(g)$ and $l(g) = l(a^{-1} a g) 
\leqslant l(a^{-1}) + l(ag) = l(ag)$, so $l(ag) = l(g)$.  A similar argument works for $l(ga) = 
l(g)$.
\end{proof}

\begin{theorem}
\label{th:delta_length}
If $l : G \rightarrow \Lambda$ is a length function, then there are a $\Lambda$-metric space $(X, 
d)$, an action of $G$ on $X$, and a point $v \in X$ such that $l = l_v$. Moreover, if $l : G \rightarrow
\Lambda$ is $\delta$-hyperbolic then the space $(X,d)$ is also $\delta$-hyperbolic.
\end{theorem}
\begin{proof} Denote $A = \ker(l)$ and consider the set $X = G / A$ of all left cosets of $A$ in $G$. 
Define a function $d_A : G / A \times G / A \rightarrow \Lambda$ so that $d_A(g A, h A) = l(g^{-1} 
h)$. Observe that by Lemma \ref{le:ker-cosets}, $d_A$ is well-defined. Indeed, if $g', h' \in G$ 
such that $g' A = g A,\ h' A = h A$ then $g' = g a_1$ and $h' = h a_2$ and $d_A(g' A, h' A) = 
l(a_1^{-1} g^{-1} h a_2) = l(g^{-1} h) = d_A(g A, h A)$.

We claim that $(X, d_A)$ is a $\Lambda$-metric space. Indeed, axioms $(LM1)$ and $(LM3)$ of 
$\Lambda$-metric space are evident. Next, $d_A(g A, h A) = l(g^{-1} h) = 0$ if and only if $g^{-1} h 
\in A$ if and only if $g A = h A$, so $(LM2)$ follows. Finally, the triangle inequality $(LM4)$ 
follows from the corresponding property of the length function $l$.

Now we show that $(X, d_A)$ is $\delta$-hyperbolic with respect to the point $A$, provided $l$ is 
$\delta$-hyperbolic. Notice that for $f A, g A \in G / A$ we have
$$(f A \cdot g A)_A = \frac{1}{2}\left(d_A(f A, A) + d_A(g A, A) - d_A(f A, g A)\right)$$
$$ = \frac{1}{2}\left(l(f) + l(g) - l(f^{-1}g)\right) = c(f, g).$$
Hyperbolicity is then a consequence of $G$ having a $\delta$-hyperbolic length function.

Now, $G$ acts on $G / A$ in a natural way, that is, if $g \in G,\ h A \in G / A$ then $g \cdot (h A)
= (gh) A$. This action is isometric since $d_A(f A, h A) = l(f^{-1} h ) = l(f^{-1}g^{-1} g h) =
d_A(g \cdot (f A), g \cdot (h A))$. Finally, $l(g) = d_A(A, gA) = l_A(g)$.
\end{proof}

Given a group $G$ and a $\delta$-hyperbolic length function $l: G \rightarrow \Lambda$, denote by
$(X_l, d_l)$ the $\delta$-hyperbolic $\Lambda$-metric space constructed in Theorem
\ref{th:delta_length}. Note that the stabilizer of the point $v$ is exactly the kernel of $l = l_v$.

In general, given an action of $G$ on an arbitrary $\delta$-hyperbolic $\Lambda$-metric space $(X, 
d)$ and a point $x \in X$, the stabilizer $G_x$ of $x$ is exactly the kernel of the $\delta$-hyperbolic 
length function $l_x$ based at $x$.

We are going to use the notion of the kernel later in Section \ref{sec:kernel}.

\subsection{Examples of group actions on hyperbolic $\Lambda$-metric \\ spaces}
\label{subsec:group_action_examples}

Here are some examples of groups acting on hyperbolic $\Lambda$-metric spaces for various $\Lambda$.

\begin{example}
\label{ex:hyperbolic_3}
Given a torsion-free word-hyperbolic group $G$ and its generating set $S$, the Cayley graph $X =
Cay(G, S)$ with the word metric (with respect to $S$) is a $\delta$-hyperbolic $\Z$-metric space for 
some $\delta \in \Z$. $G$ acts on $X$ by isometries, in particular, no element of $G$ fixed a point 
in $X$.
\end{example}

\begin{example}
\label{ex:hyperbolic_4}
Since any $\Lambda$-tree is a $\delta$-hyperbolic $\Lambda$-metric space with $\delta = 0$, the class
of groups acting on hyperbolic $\Lambda$-metric spaces contains all groups acting on $\Lambda$-trees
(in particular, all $\Lambda$-free groups).
\end{example}

\begin{example}
\label{ex:hyperbolic_5}
Any subgroup of a group acting on a $\Lambda$-metric space $(X,d)$ also acts on this space. So, the
class of groups acting on hyperbolic $\Lambda$-metric spaces is closed under taking subgroups.
\end{example}

\begin{example}
\label{ex:hyperbolic_6}
Given two groups $G_1$ and $G_2$ acting respectively on $\Lambda$-metric spaces $(X_1,d_1)$ and
$(X_2,d_2)$, there exists a $\Lambda \oplus \Z$-metric space $(X,d)$, where $\Lambda \oplus
\Z$ is taken with the right lexicographic order, such that $G_1 \ast G_2$ acts on $X$. To see
this, take some $x_1 \in X_1$ and $x_2 \in X_2$ and define $l(g) = d_n(x_n, g x_n)$ for $g \in G_n$.

Any element $g$ of $G_1 \ast G_2$ has a unique normal form $g = g_1 g_2 \cdots g_n$ with $g_i \in 
G_1 \cup G_2$ for any $i$ and if $g_i \in G_1$, then $g_{i+1} \in G_2$ and vice versa. Define then 
$l(g) = (l(g_1) + \cdots + l(g_n), 1)$ and $X = (G_1 \ast G_2, d_l)$. It is easy to show that this 
is a metric space.

Finally, it is not hard to see that, if both $X_1$ and $X_2$ are hyperbolic, then so is $X$.
\end{example}

\begin{example}
\label{ex:hyperbolic_7}
Let $H$ be a torsion-free hyperbolic group, $u$ a cyclically reduced word in generators of $H$, and 
$$G = \langle H, t \mid t^{-1} u t = u \rangle$$ 
We would like to construct a $\Z^2$-metric space $X$ on which $G$ acts as follows. 

Define $T$ as the subset of the boundary of $G$ made up of limits of sequences of the form 
$$\{w, w u, wu^2, w u^2, \ldots\}\ \ \ {\rm and}\ \ \ \{w, w u^{-1}, w u^{-2}, w u^{-3}, \ldots\},$$ 
where $w$ is an arbitrary element of $G$. Take a rooted tree $Y$ whose vertices have valence $|T|$ 
and label them by elements of $G / H$ in such a way that the adjacent vertices are the pairs of the 
form $(g H, (h t^{\pm 1} g) H)$, where $h \in H,\ g \in G$.

Denote the metric on $Y$ by $d_Y$, the word metric on $H$ (with respect to some fixed generating 
set) b $d_H$, and the hyperbolicity constant of $(H, d_H)$ by $\delta$. Note that any element of $G$ 
either is of the form $h$, if it is in $H$, or is a product of elements of the type $g t^{\pm1} h$ with 
$g \in G$ and $h \in H$ otherwise.

Now, let $X = Y \times H$ and for $(g H, h_1), (g H, h_2) \in Y \times H$ we define the distance
$d((g_1 H, h_1), (g_2 H, h_2))$ as follows. If $g_1 H = g_2 H$ then we set
$$d((g_1 H, h_1), (g_2 H, h_2)) = (d_H(h_1, h_2), 0)$$ 
If $g_1 H \neq g_2 H$ then let $e_1 e_2 \ldots e_N$, where $N = d_Y(g_1 H, g_2 H)$, be the path 
from $g_1 H$ to $g_2 H$ in $Y$. Observe that every edge $e_i$ is associated with a pair of ends
$(\alpha(e_i), \omega(e_i))$ in the corresponding copies of $(H, d_H)$. Hence, define 
$$d((g_1 H, h_1), (g_2 H, h_2)) = (d_H(1, g_1) - 2 ((g H, h_1) \cdot \alpha(e_1)) + d_H(1, g_2)$$
$$- 2 ((g H, h_1) \cdot \omega(e_n)) - 2 \sum_{i=1}^N (\omega(e_i) \cdot \alpha(e_{i+1})), d_Y(g_1 H, 
g_2 H)),$$ 
where ``$\ \cdot$'' represents the Gromov product of two ends, or an end and a point in a hyperbolic
space (computed in the appropriate copy of $(H, d_H)$).

In the light of the alternative given in the definition of a hyperbolic length function, it suffices to 
prove that $X$ is $(\delta, 0)$-hyperbolic and the action 
$$(g t^{\pm1} h) \cdot (g' H, h') = ((g t^{\pm 1} h g') H, h')$$ 
if $g' \notin H$, or 
$$(g t^{\pm 1} h) \cdot (g' H, h') = ((g t^{\pm 1}) H, h h')$$ 
if $g' \in H$, is an action by isometries to have that 
$$l(g t^{\pm 1} h) = d((H, 1), ((g t^{\pm 1}) H, h))$$ 
is a $(\delta, 0)$-hyperbolic length function on $G$. 
\end{example}

\section{Kernels and hyperbolicity constants}
\label{subs:delta}

Let $\Lambda$ be an ordered abelian group and $G$ a group with a length function $l : G \to \Lambda$
(that is, $l$ satisfies the axioms $(\Lambda 1) - (\Lambda 3)$). Fix a convex subgroup $\Lambda_0 
\leqslant \Lambda$. The order on $\Lambda$ induces an order on the quotient abelian group $\bar{\Lambda} 
= \Lambda / \Lambda_0$, so $\bar{\Lambda}$ becomes an ordered abelian group with an order preserving 
the quotient epimorphim $\eta : \Lambda \to \bar{\Lambda}$.

The $\Lambda_0$-kernel of $G$ (with respect to $l : G \to \Lambda$) is the set
$$G_{\Lambda_0} = \{g \in G \mid l(g) \in \Lambda_0\}$$
By the axiom ($\Lambda 3$), the kernel $G_{\Lambda_0}$ is a subgroup of $G$ (though not necessary 
normal).

\begin{lemma}
Let $l : G \to \Lambda$ be a length function, $\Lambda_0$ be a convex subgroup of $\Lambda$, and 
$0 < \delta \not\in \Lambda_0$. Then the restriction $l_0$ of the length function $l$ to $G_{\Lambda_0}$ 
is a $\delta$-hyperbolic length function $l_0 : G_{\Lambda_0} \to \Lambda$.
\end{lemma}
\begin{proof} Since $\Lambda_0$ is convex, it follows that $\delta > \lambda$ for any $\lambda \in 
\Lambda_0$. To prove the result one needs only to show that for any $f, g, h \in G_{\Lambda_0)}$, 
the following inequality holds
$$c(f,g) \geqslant \min\{c(f,h), c(g,h)\} - \delta,$$
which is, of course, equivalent to (avoiding working in $\mathbb{Q} \otimes \Lambda$):
$$2c(f,g) \geqslant \min\{2c(f,h), 2c(g,h)\} - 2\delta.$$
Since $\Lambda_0$ is convex, it follows that $2c(f,h), 2c(g,h) \in \Lambda_0$, so 
$$\min\{2c(f,h), 2c(g,h)\} - 2\delta < 0$$ 
But $2c(f,g) \geqslant 0$, which finishes the proof.
\end{proof}

The whole idea behind $\delta$-hyperbolic length functions is a generalization of hyperbolic
properties from usual metric spaces to $\Lambda$-metric spaces for an arbitrary $\Lambda$.
But when $\Lambda$ is not archimedean the choice of $\delta$ becomes very important as
the following construction shows.

\smallskip

Let $G$ be a word-hyperbolic group, that is, the geodesic word length function $l = | \cdot |_S :
G \to \Z$ with respect to some finite generating set $S$ is $\delta$-hyperbolic for some
$\delta \in \Z$.

Consider the group $H = G \times G$ and a map $l_H : H \to \Z^2$ defined by $l_H(f,g) =
(l(f), l(g))$, where $\Z^2$ is considered with the right lexicographic order.

\begin{lemma}
\label{le:bad_delta}
The function $l_H : H \to \Z^2$ is a $\delta_H$-hyperbolic length function on $H$, where
$\delta_H = (\delta_1, \delta_1) \in \Z^2$ and $\delta_1  \in \Z$ is such that
$\delta_1 > \delta$.
\end{lemma}
\begin{proof} It is easy to see that the axioms ($\Lambda 1$) and ($\Lambda 2$) hold immediately.

The triangle inequality ($\Lambda 3$) is also straightforward. Indeed, let $\textbf{h}_1, \textbf{h}_2
\in H$ be such that $\textbf{h}_1 = (f_1, g_1), \textbf{h}_2 = (f_2, g_2)$. Then
$$l_H(\textbf{h}_1 \textbf{h}_2) = l_H(f_1 f_2, g_1 g_2) = (l(f_1 f_2), l(g_1 g_2)) \leqslant
(l(f_1) + l(f_2), l(g_1) + l(g_2))$$
$$ = l_H(\textbf{h}_1) + l_H(\textbf{h}_2).$$

To prove ($\Lambda 4$) let $\textbf{h}_1, \textbf{h}_2, \textbf{h}_3 \in H$ be such that $\textbf{h}_1
= (f_1, g_1), \textbf{h}_2 = (f_2, g_2)$ and $\textbf{h}_3 = (f_3, g_3)$. We have to show that
$$c(\textbf{h}_1, \textbf{h}_2) \geqslant \min\{c(\textbf{h}_1, \textbf{h}_3), c(\textbf{h}_2,
\textbf{h}_3)\} - \delta_H.$$
We have
$$c(\textbf{h}_i, \textbf{h}_j) = \frac{1}{2}(l(\textbf{h}_i) + l(\textbf{h}_j) -
l(\textbf{h}_i^{-1} \textbf{h}_j)) = (c(f_i, f_j), c(g_i, g_j)).$$
Hence, we have to show that
$$(c(f_1, f_2), c(g_1, g_2)) \geqslant \min\{(c(f_1, f_3), c(g_1,g_3)) , (c(f_2, f_3), c(g_2, g_3))\}
- (\delta_1, \delta_1),$$
where $\delta_1 \in \Z$ is such that $\delta_1 > \delta$. Since $G$ is $\delta$-hyperbolic
we have
$$c(f_1, f_2) \geqslant \min\{c(f_1, f_3), c(f_2, f_3)\} - \delta > \min\{c(f_1, f_3), c(f_2, f_3)\} -
\delta _1.$$
Without loss of generality suppose that $c(g_1, g_3) \leqslant c(g_2, g_3)$, so that
$$c(g_1, g_2) > c(g_1, g_3) - \delta_1.$$
Thus, we need to prove that
$$(c(f_1, f_2), c(g_1, g_2)) \geqslant (\kappa, c(g_1, g_3)) - (\delta_1, \delta_1),$$
where $\kappa \in \Z^2$ is defined as follows
\[
\kappa = \left\{\begin{array}{ll}
\mbox{$c(f_1, f_3)$}  & \mbox{$c(g_1, g_3) < c(g_2, g_3)$} \\
\mbox{$\min\{c(f_1, f_3), c(f_2, f_3)\}$} & \mbox{$c(g_1,g_3) = c(g_2, g_3)$}
\end{array}
\right.
\]
The above inequality holds since $c(g_1, g_2) > c(g_1, g_3) - \delta_1$.

\end{proof}

\begin{remark}
\label{rem:bad_delta}
The result of Lemma \ref{le:bad_delta} does not hold for $\delta_1 = \delta$. Indeed, in the case 
when $c(g_1, g_3) < c(g_2, g_3)$ and $c(g_1, g_2) = c(g_1, g_3) - \delta$ we must have $c(f_1, f_2)
\geqslant c(f_1, f_3) - \delta$ which may not be true.
\end{remark}

The construction above shows that if $G$ has a $\delta$-hyperbolic length function $l$ in a
non-archimedean $\Lambda$ and if there are many elements of $G$ such that $l(g) < \delta$ then
the properties of $G$ may be quite far from properties of word-hyperbolic groups.

\section{Various types of actions and functions}
\label{sec:hyp_length_func}

In this section we study properties of $\delta$-hyperbolic length functions. In particular, we
consider new axioms in addition to $(\Lambda 1)-(\Lambda 4, \delta)$ which shed some light on the 
structure of the underlying group. We also try to characterize the corresponding group actions (which 
exist in view of Theorem \ref{th:delta_length}).

Below we  use the following notation. If $l : G \to \Lambda$ is a $\delta$-hyperbolic length function 
and $\alpha \in \Lambda$ then for $g,h \in G$ we write
$$g h = g \circ_\alpha h$$
if $c(g^{-1}, h) \leqslant \alpha$. Respectively, $g h = g \circ h$ means that $c(g^{-1}, h) = 0$, 
that is, $l(gh) = l(g) + l(h)$.

\subsection{Regularity}
\label{subs:reg_length_func}

Let $l : G \to \Lambda$ be a $\delta$-hyperbolic length function. We introduce several conditions
on $l$, all of which lead to the same property called {\em regularity}. Both conditions depend on 
a parameter $k \in \N$. Below is the first condition on $l$, which we denote ($R1, k$).

\smallskip

{\noindent ($R1, k$):} $\exists\ k \in \N\ \ \forall\ g,h \in G\ \ \exists\ u \in G$
$$g = u \circ_{k \delta} g_1\ \ \&\ \ h = u \circ_{k \delta} h_1\ \ \&\ \
g^{-1} h = g_1^{-1} \circ_{k \delta} h_1$$

Observe that if $\delta = 0$ then $l$ is a Lyndon length function on $G$ and it satisfies ($R1, k$) 
for some $k$ (and hence for any) if and only if $l$ is regular (see \cite{KMRS:2012, KMS_Survey:2013} for
all definitions and properties of regular Lyndon length functions). For an arbitrary $\delta$, clearly 
($R1, k$) implies ($R1, m$) for any $m > k$.

\smallskip

Here is another condition on $l$, we call it ($R2, k$).

\smallskip

{\noindent ($R2, k$):} $\exists\ k \in \N\ \ \forall\ g,h \in G\ \ \exists\ u \in G$
$$l(u) \leqslant c(g,h) + k \delta\ \ \&\ \ l(u^{-1} g) \leqslant c(g^{-1}, g^{-1} h) + k \delta$$
$$\&\ \ l(u^{-1} h) \leqslant c(h^{-1}, h^{-1} g) + k \delta$$

Again, as in the case of ($R1, k$), the condition ($R2, k$) obviously implies ($R2, m$) for any $m 
> k$.

The following result makes a connection between these conditions.

\begin{lemma}
\label{le:reg_len_equiv}
Let $G$ be a group and $l : G \to \Lambda$ be a $\delta$-hyperbolic length function. Then the
following implications hold for $l$:
$$(R1, k) \Longrightarrow (R2, k+1),\ \ \ (R2, k) \Longrightarrow (R1, k)$$
\end{lemma}
\begin{proof} $(R1, k) \Longrightarrow (R2, k+1):$

\smallskip

We have $g = u \circ_{k \delta} g_1$ which is equivalent to $c(u^{-1}, g_1) \leqslant k \delta$, so
$l(u) + l(g_1) - l(g) \leqslant 2k \delta$. It follows that $2l(u) - 2c(u,g) \leqslant 2k \delta$, or
$l(u) \leqslant c(u,g) + k \delta$. Using the same argument we get $l(u) \leqslant c(u,h) + k \delta$.
Hence,
$$c(g,h) \geqslant \min\{c(u,g), c(u,h)\} - \delta \geqslant l(u) - (k+1)\delta$$
and we have $l(u) \leqslant c(g,h) + (k+1) \delta$.

Similarly, from $g^{-1} = g_1^{-1} \circ_{k \delta} u^{-1}$ we get
$$l(g_1) = l(g_1^{-1}) \leqslant c(g^{-1}, g_1^{-1}) + k \delta,$$
from $g^{-1} h = g_1^{-1} \circ_{k \delta} h_1$ we get
$$l(g_1) = l(g_1^{-1}) \leqslant c(g^{-1}, g^{-1} h) + k \delta,$$
which imply
$$l(u^{-1} g) = l(g_1) \leqslant c(g^{-1}, g^{-1} h) + (k+1) \delta.$$
The inequality  $l(u^{-1} h) = l(h_1) \leqslant c(h^{-1}, h^{-1} g) + (k+1) \delta$ is derived in 
the same way.

\smallskip

$(R2, k) \Longrightarrow (R1, k):$

\smallskip

We have that $l(u) \leqslant c(g,h) + k \delta$ and $l(u^{-1} g) \leqslant c(g^{-1}, g^{-1} h) + 
k \delta$ which imply that $l(u) + l(u^{-1} g) \leqslant l(g) + 2k \delta$. Therefore,
$$c(u^{-1}, u^{-1} g) = \frac{1}{2}\left(l(u) + l(u^{-1} g) - l(g)\right) \leqslant k \delta,$$
that is, $g = u (u^{-1} g) =  u \circ_{k \delta} (u^{-1} g)$. Similar argument proves the inequalities
$h = u \circ_{k \delta} (u^{-1} h)$ and $g^{-1} h = (u^{-1} g)^{-1} \circ_{k \delta} (u^{-1} h)$.
\end{proof}

\begin{defn}
\label{de:regular_len}
A $\delta$-hyperbolic length function $l : G \to \Lambda$ is called {\em regular} if it satisfies 
either $(R1, k)$, or $(R2, k)$ for some $k \in \N$.
\end{defn}

Observe that Definition \ref{de:regular_len} agrees with the notion of regularity for Lyndon length
functions in the case when $\delta = 0$.

\smallskip

Now, we say that the action of $G$ on a $\delta$-hyperbolic space $(X,d)$ is {\em regular} if for 
some $x \in X$ (hence, for any) the length function $l$ based at $x$ is regular. Since in this case
for every $g, h \in G$ we have $c(g,h) = (g x \cdot h x)_x$ then regularity of the action can be 
explicitly characterized by the following condition (which is just a reformulation of $(R2, k)$ in
terms of actions):

\smallskip

{\noindent $\exists\ k \in \N\ \ \forall\ g,h \in G\ \ \exists\ u \in G$}
$$d(x, u x) \leqslant (g x \cdot h x)_x + k \delta\ \ \&\ \ d(x, (u^{-1} g) x) \leqslant (g^{-1} x
\cdot (g^{-1} h) x)_x + k \delta$$
$$\&\ \ d(x, (u^{-1} h) x) \leqslant (h^{-1} x \cdot (h^{-1} g) x)_x + k \delta$$

In the case when $(X, d)$ is geodesic we introduce the following condition on the action of $G$.

\smallskip

{\noindent $(RA, k)$:} there exists $k \in \N$ such that for any $g,h \in G$ there exists $u \in G$
with the property that $u x$ belongs to the $k \delta$-neighborhood of the interior of
$\Delta_I(x, gx, hx)$.

\smallskip

Obviously, if $\delta = 0$ then the interior of $\Delta_I(x, gx, hx)$ is a single point $Y(x, gx, 
hx) = [x, gx] \cap [x, hx] \cap [gx, hx]$ and $(RA, k)$ is equivalent to the regularity condition
for group actions on $\Lambda$-trees (see \cite{KMS:2011, KMRS:2012}). In general, equivalence of
$(RA, k)$ and regularity of the action follows from the lemma below.

\begin{lemma}
\label{le:reg_act_equiv}
Let $G$ be a group acting on a $\delta$-hyperbolic space $(X,d)$ and $l : G \to \Lambda$ a
$\delta$-hyperbolic length function based at $x \in X$. If $X$ is geodesic then the following
implications hold
$$(RA, k) \Longrightarrow (R2, k+4),\ \ \ (R2, k) \Longrightarrow (RA, 3k+4)$$
\end{lemma}
\begin{proof} $(RA, k) \Longrightarrow (R2, k+4):$

\smallskip

Since $X$ is $\delta$-hyperbolic then for any $g, h \in G$, the triangle $\Delta_I(x, gx, hx)$ has
diameter at most $4 \delta$. Hence, one of its vertices is at a distance $c(g, h) = (g x \cdot 
h x)_x$ from $x$, another vertex is at a distance $c(g^{-1}, g^{-1} h)$ from $gx$ (one can see 
this by translating the whole picture by $g^{-1}$), and the third one is at a distance $c(h^{-1}, 
h^{-1} g)$ from $hx$ (again, one can see this by translating the whole picture by $h^{-1}$). Thus, 
if $ux$ is in the $k \delta$-neighborhood of the interior of $\Delta_I(x, gx, hx)$, then it is in the
$(k+4) \delta$-neighborhood of all three of its corners which implies $(R2, k+4)$.

\smallskip

$(R2, k) \Longrightarrow (RA, k+4):$

\smallskip

Suppose $(R2, k)$ holds, so, the action of $G$ on $(X,d)$ (that is, $(X_l, d_l)$) is regular, and we
want to show that it satisfies $(RA, 3k+4)$ with respect to the base-point $x \in X$.

First of all, it is easy to see that
$$(g x \cdot h x)_x + (g^{-1} x \cdot (g^{-1} h) x)_x = d_l(x, gx)$$
Next, we have
$$d_l(x, gx) \leqslant d_l(x, ux) + d_l(ux, gx) \leqslant (gx \cdot hx)_x + k \delta + d_l(ux, gx)$$
from which it follows that
$$d_l(x, gx) - d_l(ux, gx) \leqslant d_l(x, ux) \leqslant (gx \cdot hx)_x + k \delta$$
which, combined with the inequalities $d_l(ux, gx) \leqslant (g^{-1}x \cdot (g^{-1}h) x) + k\delta$
and $d_l(x, gx) - (g^{-1}x \cdot (g^{-1}h) x)_x = (gx \cdot hx)_x$, implies that
$$(gx \cdot hx)_x - k \delta \leqslant d_l(x, ux) \leqslant (gx \cdot hx)_x + k \delta$$
A similar argument shows that
$$(g^{-1} x \cdot (g^{-1}h) x)_x - k \delta \leqslant d_l(gx, ux) \leqslant (g^{-1}x \cdot (g^{-1}h) x)_x
+ k \delta$$

Consider $\Delta(x, gx, hx)$ and $\Delta(x, gx, ux)$ which are both geodesic triangles in $X$. The
point on $[x, gx]$ situated at a distance $(gx \cdot hx)_x$ from $x$ is a vertex of $\Delta_I(x, gx,
hx)$, and the points at a distance $(gx \cdot ux)_x$ on $[x, gx]$ and $[x, ux]$ are at a distance at
most $4 \delta$ from one another. Next, we have
$$(ux \cdot gx)_x = \frac{1}{2}(d(x, ux) + d(x, gx) - d(ux, gx))$$
$$\geqslant \frac{1}{2}(((gx \cdot hx)_x - k \delta) + d(x, ux) - ((g^{-1}x \cdot (g^{-1}h) x)_x + 
k \delta))$$
$$= \frac{1}{2}(2 (gx \cdot hx)_x - 2k \delta) = (gx \cdot hx)_x - k \delta$$
It follows that a point on $[x, gx]$ at a distance $(gx \cdot ux)_x$ from $x$ is at a distance at most
$k \delta$ from a vertex of $\Delta_I(x, gx, hx)$, a point on $[x, ux]$ at a distance $(gx \cdot ux)_x$
from $x$ is at a distance at most $2 k \delta$ from $ux$, so $ux$ is at a distance at most
$(3k + 4) \delta$ from one of the vertices of $\Delta_I(x, gx, hx)$. Hence, $(RA, 3k+4)$ holds for
the action of $G$ on $X$.
\end{proof}

\subsection{Completeness and hyperbolic groups}
\label{subs:compl_length_func}

A $\delta$-hyperbolic length function $l : G \to \Lambda$ is called {\em complete} if the following
axiom holds:
\begin{enumerate}
\item[(C)] $\forall\ g \in G\ \ \forall\ \alpha \leqslant l(g)\ \ \exists\ u \in G:\ \ g = u \circ g_1\ \
\&\ \ l(u) = \alpha$.
\end{enumerate}
Note that such an element $u$ may not be unique.

\begin{lemma}
\label{le:30}
Let $l : G \rightarrow \Lambda$ be a complete $\delta$-hyperbolic length function. Then
 for all $g,h \in G$, if $g = u \circ g_1,\ h = v \circ h_1$, with $l(u) = l(v) \leqslant
c(g,h)$, it follows that $l(u^{-1} v) \leqslant 4 \delta$.
\end{lemma}
\begin{proof}  Consider the triple $\{h_1, v^{-1}, v^{-1} u g_1\}$. Then by $\delta$-hyperbolicity
of $l$ we have
$$2c(v^{-1}, h_1) \geqslant \min\{2c(v^{-1} u g_1, v^{-1}), 2c(h_1, v^{-1} u g_1)\} - 2\delta,$$
so that
$$0 \geqslant \min\{l(v^{-1} u g_1) - l(g_1), l(v^{-1} u g_1) + l(h_1) - l(g^{-1} h)\} - 2\delta.$$
At the same time we have $l(u) = l(v) \leqslant c(g,h)$, which implies that $l(g^{-1} h) \leqslant
l(g_1) + l(h_1)$ so that
$$l(v^{-1} u g_1) - l(g_1) \leqslant l(v^{-1} u g_1) + l(h_1) - l(g^{-1} h),$$
which gives $0 \geq l(v^{-1} u g_1) - l(g_1) - 2 \delta$. Hence
$$0 \leqslant l(v^{-1} u g_1) - l(g_1) \leqslant 2 \delta.$$

Now consider the triple $\{u, v, g\}$. We have
$$2c(u, v) \geqslant \min\{2c(u, g), 2c(v, g)\} - 2\delta.$$
Next, $2c(u, g) = 2 l(u)$ and $2c(v,g) = 2l(u) + l(g_1) - l(v^{-1} u g_1)$. But $l(g_1) - l(v^{-1}
u g_1) \leqslant 0$ implies $2c(u,g) \geqslant 2c(v,g)$, so $2c(u,v) \geqslant 2c(v,g) - 2 \delta$
and therefore
$$2l(u) - l(u^{-1} v) \geqslant 2l(u) + l(g_1) - l(v^{-1} u g_1) - 2\delta,$$
so that we have the desired inequality
$$l(u^{-1} v) \leqslant l(v^{-1} u g_1) - l(g_1) + 2\delta \leqslant 4\delta.$$
\end{proof}

\begin{corollary}
\label{co:30}
Let $l : G \rightarrow \Lambda$ be a complete $\delta$-hyperbolic length function. Suppose also that
$c(g,h) \in \Lambda$ for all $g,h \in G$.
Then
 $\forall\ g,h \in G\ \ \exists\ u,v \in G:\ \ g = u \circ g_1\ \ \&\ \ h = v \circ h_1\ \
\&\ \ l(u) = l(v) = c(g,h)\ \ \&\ \ l(u^{-1} v) \leqslant 4 \delta$.
\end{corollary}
\begin{proof}
From completeness of $l$ it follows that for any $g,h \in G$ there exist $u, v \in G$ such that
$g = u \circ g_1,\ h = v \circ h_1$, with $l(u) = l(v) = c(g,h)$. Now, by Lemma  \ref{le:30} we have 
$l(u^{-1} v) \leqslant 4 \delta$ and the result follows.
\end{proof}

Here are examples of groups with complete $\delta$-hyperbolic length functions.

\begin{example}
\label{ex:complete_1}
Every $\delta$-hyperbolic group has a complete $\Z$-valued $\delta$-hyperbolic length function,
which is the geodesic length of its elements.
\end{example}

\begin{example}
\label{ex:complete_2}
If $G$ has a complete $\delta$-hyperbolic length function with values in $\Lambda$ and $C_G(u)$ is a
centralizer of $u \in G$, then the group $G' = \langle G, t \mid [C_G(u), t] = 1 \rangle$ has a
complete $\delta$-hyperbolic length function with values in $\Lambda \oplus \Z$ (with the
right lexicographic order).
\end{example}

We call a $\Lambda$-metric space $(X, d)$ {\em quasi-geodesic} if for every $x,y \in X$ there is a map 
$\gamma : [0, d(x,y)] \to X$ such that $\gamma(0) = x,\ \gamma(d(x,y)) = y$, and for any $0 \leqslant 
\alpha \leqslant \beta \leqslant d(x,y)$ we have
$$\beta - \alpha \leqslant d(\gamma(\alpha), \gamma(\beta)) \leqslant \beta - \alpha + C,$$
here $C$ is a constant which depends on $X$ only.

\begin{theorem}
\label{th:complete}
Let $l : G \rightarrow \Lambda$ be a complete $\delta$-hyperbolic length function. Then the 
$\delta$-hyperbolic $\Lambda$-metric space $(X_l, d_l)$ constructed from the pair $(G, l)$ is 
quasi-geodesic.
\end{theorem}
\begin{proof} Recall that $X_l$ is the set $G / A$, where $A = \ker l$, and the metric $d_l$ on $X_l$
is defined by $d_l(g A, h A) = l(g^{-1}h)$.

Let $x = g A,\ y = h A \in X_l$. Consider a map $\gamma : [0, d(x,y) ] \to X_l$ such that $\gamma(0) 
= x,\ \gamma(d(x,y)) = y$, and for any $0 \leqslant \alpha \leqslant d(x,y)$ we put $\gamma(\alpha) 
= (g u_1) A$, where $g^{-1} h = u_1 \circ u_2$ and $l(u_1) = \alpha$ (such $u_1, u_2 \in G$ exist by
completeness of $l$). Thus we have $d(\gamma(0), \gamma(\alpha)) = d(gA, (g u_1) A) = l(u_1) = 
\alpha$.

Now, for any $\alpha$ and $\beta$ such that $0 \leqslant \alpha \leqslant \beta \leqslant d(x,y)$ 
we have $\gamma(\alpha) = (g u_1) A,\ \gamma(\beta) = (g v_1) A$, where $g^{-1} h = u_1 \circ u_2 
= v_1 \circ v_2$ and $l(u_1) = \alpha,\ l(v_1) = \beta$. Since $\alpha \leqslant \beta$, by Lemma 
\ref{le:30}, $v_1 = (u_1 s) \circ v_2$, where $l(s) \leqslant 4\delta$. Hence, we have
$$\beta - \alpha \leqslant d(\gamma(\alpha), \gamma(\beta)) = l(u_1^{-1} v_1) \leqslant l(v_2) + 
l(s) \leqslant \beta - \alpha + 4\delta$$
\end{proof}

\begin{theorem}
\label{th:complete_hyp}
A group $G$ is $\delta$-hyperbolic if and only if there exists a $\delta$-hyperbolic length function
$l : G \to \Z$ with the following properties
\begin{enumerate}
\item[(a)] $l$ is complete,
\item[(b)] $|\{g \in G \mid l(g) \leqslant 1\}| < \infty$.
\end{enumerate}
\end{theorem}
\begin{proof} If $G$ is $\delta$-hyperbolic then its word metric $| \cdot |_S : G \to \Z$
with respect to some finite generating set $S$ is a $\delta$-hyperbolic length function, it is obvious.

\smallskip

Now, suppose on a group $G$ there exists a $\delta$-hyperbolic length function $l : G \to \Z$
which satisfies the conditions (a) and (b). Denote $S = \{g \in G \mid l(g) \leqslant 1\}$. Observe that
$S$ is finite and $G = \langle S \rangle$ since by completeness of $l$ every $g \in G$ can be decomposed
as a finite product of elements from $S$. Hence, $l$ can be viewed as a word metric $| \cdot |_S$ with
respect to $S$. Finally, the Cayley graph of $G$ with respect to $S$ is $\delta$-hyperbolic which follows
from $\delta$-hyperbolicity of $l$.
\end{proof}

\subsection{Free length functions}
\label{subs:free_length_func}

A $\delta$-hyperbolic length function $l: G \to \Lambda$ is called {\em free} if
\begin{enumerate}
\item [(F)] $\forall\ g \in G:\ g \neq 1 \rightarrow l(g^2) > l(g) + 3\delta$
\end{enumerate}
Observe that if $\delta = 0$, that is, the $\delta$-hyperbolic $\Lambda$-metric space $(X_l, d_l)$ 
is a $\Lambda$-tree then $l$ is a free Lyndon length function.

We say that the action of $G$ on a $\delta$-hyperbolic space $(X,d)$ is {\em free} if for some $x \in
X$ (hence, for any) the length function $l$ based at $x$ is free. Obviously, if a $\delta$-hyperbolic 
length function $l: G \to \Lambda$ is free then $\ker(l)$ is trivial.

\begin{example}
\label{ex:free_1}
Every torsion-free $\delta$-hyperbolic group has a free (and complete) $\Z$-valued
$\delta$-hyperbolic length function, which is the geodesic length of its elements.
\end{example}

Observe that, in view of Lemma \ref{le:1.3.1}, free action implies that every element of $G$ acts 
as either a hyperbolic isometry, or an inversion. Hence, we say that a group $G$ is {\em $\Lambda$-free} 
if it acts on a $\delta$-hyperbolic $\Lambda$-metric space $(X,d)$ freely and without inversions.

\section{Proper actions and hyperbolicity relative to the kernel}
\label{sec:kernel}

In this section we consider action of a finitely generated group $G$ on a geodesic $\delta$-hyperbolic
$\R$-metric space $(X,d)$ and show that if the action is ``nice'' (regular and proper)
then $G$ is weakly hyperbolic relative to the kernel of the associated length function.

\subsection{Proper actions}
\label{subs:proper}

We fix the group $G$ with a finite generating set $S$ and the $\delta$-hyperbolic
$\R$-metric space $(X,d)$ with a base-point $x \in X$. As usual we have a $\delta$-hyperbolic
length function $l : G \to \R$ based at $x$ and its kernel $\ker(l)$ is a subgroup of $G$.
Recall that $\ker(l) = G_x = Stab_G(x)$.

Recall that the action of $G$ on $(X,d)$ satisfies the axiom ($R2, k$) if there exists $k \in \N$
such that for all $g,h \in G$ there exists $u \in G$ with the property
$$d(x, u x) \leqslant (g x \cdot h x)_x + k \delta\ \ \&\ \ d(x, (u^{-1} g) x) \leqslant (g^{-1} x
\cdot (g^{-1} h) x)_x + k \delta$$
$$\&\ \ d(x, (u^{-1} h) x) \leqslant (h^{-1} x \cdot (h^{-1} g) x)_x + k \delta$$
See Subsection \ref{subs:reg_length_func} for all the details.

Next, we say the action is {\em proper relative to $x$} if there exists some $\alpha \in \R$
such that $l(g) > \alpha$ for any $g \in G \ssm G_x$ and the set
$$B_N = \{g \in G \mid d(x, gx) \leqslant N\}$$
is bounded for any $N \in \N$ in the Cayley graph $\Gamma(G, S \cup G_x)$ of $G$ relative
to $S \cup G_x$.


Since $S$ is finite, there exists $N \in \N$ such that $S \subseteq B_N$. Define the weighted
graph $\Gamma$ so that
$$V(\Gamma) = G / G_x,\ \ E(\Gamma) = \{(g G_x, (g h) G_x) \mid g \in G, h \in B_N\}$$
and the weight function $w : E(\Gamma) \to \R$ is defined by $w(g G_x, (gh) G_x) = l(h)$.
Note that $\Gamma$ is connected since $S \subseteq B_N$. Next, $\Gamma$ is a metric space with
respect to the metric $d_\Gamma$ defined by
$$d_\Gamma(g G_x, h G_x) = \min\{w(p) \mid p\ {\rm is\ a\ path\ in}\ \Gamma\ {\rm connecting}\
g G_x\ {\rm and}\ h G_x\}.$$
Notice that for $g, h \in G$, we have $d(gx, hx) \leqslant N$ if and only if $(g G_x, h G_x) \in
E(\Gamma)$. It follows that in a geodesic path in $\Gamma$ no two consecutive edges both have
weights less than or equal to $\frac{N}{2}$ (by the triangle inequality).

\begin{lemma}
\label{le:rel_hyp_2}
Suppose the action of $G$ on $(X,d)$ is proper relative to $x$ and that it satisfies $(R2, k)$ for
some $k \in \N$.
\begin{enumerate}
\item[(i)] If $(G_x, a G_x), (a G_x, b G_x) \in E(\Gamma)$ and $(G_x, a G_x) \cup (a G_x, b G_x)$ is
a geodesic in $\Gamma$ then $d(x, b x) \geqslant d(x, a x) + d(a x, b x) - 2k\delta$.
\item[(ii)] If $(G_x, a G_x), (a G_x, b G_x), (b G_x, c G_x) \in E(\Gamma)$ and $(G_x, a G_x) \cup
(a G_x, b G_x) \cup (b G_x, c G_x)$ is a geodesic in $\Gamma$ with $w(a G_x, b G_x) < \frac{N}{2}$
then $d(x, c x) \geqslant d(x, a x) + d(a x, b x) + d(b x, c x) - 5k\delta$.
\end{enumerate}
\end{lemma}
\begin{proof} (i) Consider the geodesic triangle $\Delta_I(x, ax, bx)$ with the vertices $p \in [x, ax]$,
$q \in [x, bx]$ and $r \in [ax, bx]$. From ($R2, 1$) it follows that there exists $u \in G$ such that
$$d(x, ux) \leqslant d(x, p) + k\delta,\ \ d(ax, ux) \leqslant d(ax, q) + k\delta,\ \
d(bx, ux) \leqslant d(bx, r) + k\delta$$
By definition of $\Delta_I(x, ax, bx)$ we have
$$d(x, bx) = d(x,q) + d(q, bx) = d(x, ax) + d(ax, bx) - d(p, ax) - d(r, ax).$$
If $d(p, ax) = d(r, ax) < k\delta$ then there is nothing to prove. Otherwise, we have
$$d(x, ux) \leqslant d(x, p) + k\delta \leqslant d(x, ax) - k\delta + k\delta = d(x, ax)
\leqslant N$$
and
$$d(ux, bx) \leqslant d(r, bx) + k\delta \leqslant d(ax, bx) - k\delta + k\delta = d(ax, bx)
\leqslant N,$$
so we have that $(G_x, u G_x) \cup (u G_x, bG_x)$ is a path in $\Gamma$. But then
$$d(x, ax) + d(ax, bx) = w(G_x, a G_x) + w(a G_x, b G_x) \leqslant w(G_x, u G_x) + w(u G_x, b G_x)$$
$$= d(x, ux) + d(ux, bx) \leqslant d(x, q) + d(q, bx) +2k\delta = d(x,bx) + 2k\delta$$

\smallskip

(ii) Consider the geodesic triangle $\Delta_I(x, bx, cx)$ with the vertices $p \in [x, bx]$, $q \in
[x, cx]$ and $r \in [bx, cx]$. Again, from ($R2, k$) it follows that there exists $u \in G$ such that
$$d(x, ux) \leqslant d(x, p) + k\delta,\ \ d(bx, ux) \leqslant d(bx, q) + k\delta,\ \
d(cx, ux) \leqslant d(cx, r) + k\delta$$
By definition of $\Delta_I(x, ax, cx)$ we have
$$d(x, cx) = d(x,q) + d(q, cx) = d(x, bx) + d(bx, cx) - d(p, bx) - d(r, bx).$$
If $d(p, bx) = d(r, bx) < k\delta$ then
$$d(x, cx) \geqslant d(x, bx) + d(bx, cx) - 2k\delta \geqslant d(x, ax) + d(ax, bx) - 2k\delta +
d(bx, cx) - 2k\delta,$$
so, there is nothing to prove. So, assume that $d(r, bx) \geqslant k\delta$. Then we have
$$d(ux, cx) \leqslant d(r, cx) + k\delta \leqslant d(bx, cx) - k\delta + k\delta \leqslant N,$$
that is, $(u G_x, c G_x) \in E(\Gamma)$. Note also that, if $d(x, ux) \leqslant N$ then $(G_x,
u G_x) \in E(\Gamma)$ which implies that
$$d(x, ax) + d(ax, bx) + d(bx, cx) \leqslant d(x, ux) + d(ux, cx) \leqslant d(x, cx) + 2k\delta,$$
so, again, there would be nothing to prove. Thus we can assume that $d(x, ux) > N$.

Observe that $ux$ is at a distance of at most $k\delta$ from a point on $[x, bx]$. If this point is
at a distance of at most $\delta$ from a point on $[ax, bx]$, then $ux$ is at a distance of at
most $\frac{N}{2} + (k+1)\delta$ from $ax$. If, on the other hand, it is at a distance of at most
$\delta$ from a point on $[x, ax]$, say $t$, then we have
$$d(x, ax) < d(x, ux) \leqslant d(x, t) + (k+1)\delta,$$
so, $d(t, ax) < (k+1)\delta$, and thus $d(ax, ux) < (2k+2)\delta$. In both cases, since we can assume
that $N > (2k+2)\delta$, we have $(aG_x, uG_x) \in E(\Gamma)$. It follows that
$$d(x, ax) + d(ax, bx) + d(bx, cx) \leqslant d(x, ax) + d(ax, ux) + d(ux, cx).$$

Finally, we have $d(x, bx) \geqslant d(x, ax) + d(ax, bx) - 2k\delta$, so, $(ax \cdot bx)_x \leqslant
k\delta$. It follows that if $v$ is a point on $[x, bx]$ such that
$$d(x, ax) - 2k\delta \leqslant d(v, x) \leqslant d(x, ax) + 2k\delta$$
then $d(ax, v) \leqslant 2k \delta$. Similarly, if $w$ is a point on $[x, bx]$ such that
$$d(x, ux) - 2k\delta \leqslant d(w, x) \leqslant d(x, ux) + 2k\delta$$
then $d(ux, w) \leqslant k\delta$. Since we assume that $d(x, ux) > d(x, ax)$, we have that
$$d(ax, ux) \leqslant d(x, ux) - d(x, ax) + 3k\delta.$$
It follows that
$$d(x, ax) + d(ax, ux) + d(ux, cx) \leqslant d(x, ax) + d(x, ux) - d(x, ax) + 3k\delta + d(ux, cx)$$
$$= d(x, ux) + d(ux, cx) + 3k\delta \leqslant d(x, cx) + 5k \delta.$$
\end{proof}

Now, define a map $\varphi : \Gamma \to X$ so that
$$\varphi(g G_x) = gx\ \ {\rm and}\ \ \varphi(g G_x, h G_x) = [gx, hx].$$

\begin{corollary}
\label{co:rel_hyp_2}
Suppose the action of $G$ on $(X,d)$ satisfies $(R2, k)$ for some $k \in \mathbb{N}$. Then 
$\Gamma$ is hyperbolic.
\end{corollary}
\begin{proof} Let $\gamma$ be a geodesic in $\Gamma$. By the construction, $\varphi(\gamma)$ is a
concatenation of geodesic segments of lengths at most $N$ so that no two consecutive segments both
have weights less than or equal to $\frac{N}{2}$. Hence, if $x \in \varphi(\gamma)$ then the
intersection $B_{\frac{N}{4}}(x) \cap \varphi(\gamma)$ is contained in the image $\varphi(\gamma_0)$
of a subpath $\gamma_0$ of $\gamma$ in $\Gamma$ which contains not more than three edges.
Moreover, if $\gamma_0$ consists of three edges then the length of the middle one is less than
$\frac{N}{2}$. Now, by Lemma \ref{le:rel_hyp_2}, $\varphi(\gamma)$ is a $(1, 5 \delta, 
\frac{N}{4})$-local-quasi-geodesic. It is a known result (see 
\cite[Theorem 3.1.4]{CoornaertDelzantPapadopoulos:1990}) that there exist constants $L(\delta, 
\lambda, c),\ \lambda'(\delta, \lambda, c)$, and $c'(\delta, \lambda, c)$ such that any 
$(\lambda,c,L)$-local-quasi-geodesic is a $(\lambda',c')$-quasi-geodesic.

We can assume that $\frac{N}{4} \geqslant L(\delta, 1, 5\delta)$ since $N$ is independent of $\delta$ 
and $L = L(\delta)$. Then we use $\lambda = \lambda'(\delta, 1, 5\delta)$ and $c = c'(\delta, 1, 
5\delta)$ and it is immediate that $\Gamma$ is $(\lambda \delta + c)$-hyperbolic.
\end{proof}


\begin{lemma}
\label{le:rel_hyp_1}
If the action of $G$ on $(X,d)$ is proper relative to $x$ then $\Gamma$ is quasi-isometric to the
Cayley graph $\Gamma(G, S \cup G_x)$ of $G$ relative to $S \cup G_x$.
\end{lemma}
\begin{proof} Denote the metric on $\Gamma(G, S \cup G_x)$ by $d'$. Consider the function
$\Gamma(G, S \cup G_x) \to \Gamma$ defined by $g \to g G_x$. Since $S \cup G_x \subseteq B_N$,
the image of a path from $\Gamma(G, S \cup G_x)$ is a path in $\Gamma$ of the same combinatorial
length and we have
$$d_\Gamma(g G_x, h G_x) \leqslant N d'(g, h).$$
Next, suppose that $N' \in \mathbb{N}$ is such that $B_N$ is contained in the ball of radius $N'$
centered at the identity in $\Gamma(G, S \cup G_x)$ (such $N'$ exists since the action is proper).
Any edge of $\Gamma$ has weight at least $\alpha$ and it lifts to a geodesic path in
$\Gamma(G, S \cup G_x)$ of length at most $N'$. Since in any geodesic word in $(S \cup G_x)$ at
least every other letter is in $S$, we have that the edge together with any elements of $S$ adjacent
to it will have length at most $2N'$. Hence,
$$d'(g,h) \leqslant \frac{2N'}{\alpha} d_\Gamma(g G_x, h G_x)$$
and the required statement follows.
\end{proof}

\begin{theorem}
\label{th:rel_hyp_1}
Let $G$ be a finitely generated group acting on a  $\delta$-hyperbolic $\mathbb{R}$-metric space
$(X, d)$ with a base-point $x \in X$. If the action of $G$ on $X$ is proper relative to $x$ and
that it satisfies $(R2, k)$ for some $k \in \mathbb{N}$ then $G$ is weakly hyperbolic relative to
$G_x$.
\end{theorem}
\begin{proof} This is a direct consequence of Lemma \ref{le:rel_hyp_1} and Corollary 
\ref{co:rel_hyp_2}.
\end{proof}

\subsection{Geometric alternative to relatively proper actions}
\label{subs:geom_alt_rel_proper_actions}

One problem with our definition of a relatively proper action is that it is very hard to detect, 
especially geometrically. Let us look at an alternative condition which gives the same result but 
it is more easy to interpret geometrically, though it gives less insight into the relation between 
the action and the relative Cayley graph.

Let $X$ be a geodesic $\delta$-hyperbolic space on which $G$ acts and choose a base-point $x$. As
before, $B_n = \{g \in G \mid d(x, g x) \leqslant n\}$. We say the action of $G$ on $X$ has {\em
property $P(n)$} if
\begin{enumerate}
\item[(i)] there exists $\alpha$ such that $B_\alpha = G_x$,
\item[(ii)] $B_n$ generates $G$,
\item[(iii)] the set of double cosets $\{G_x\, g\, G_x \mid g \in B_n\}$ is finite.
\end{enumerate}


\begin{prop}
\label{pr:weak_rel_hyp}
Suppose $G$ is finitely generated relatively to $G_x$ and the action of $G$ on $X$ has property 
$P(n)$ for some $n > 6144 \log_2(154) + 768 + 2288 \delta$. Then $G$ is weakly relatively 
hyperbolic relative to $G_x$.
\end{prop}
\begin{proof} It is suffficient to prove that Lemma \ref{le:rel_hyp_1} and Corollary \ref{co:rel_hyp_2}
hold in this case. Since Lemma \ref{le:rel_hyp_2} requires only regularity and that $n \geqslant
(2k + 2) \delta$, we have nothing to prove there.

Lemma \ref{le:rel_hyp_1} is rather easy to prove in this case. By assumption, $B_n$ generates $G$ and
that $\{G_x\, g\, G_x \mid g \in B_n\}$ is finite, so we can take $S$ to be a finite set of 
representatives of $\{G_x\, g\, G_x \mid g \in B_n\}$. Again, since both $S$ and $G_x$ are contained 
in $B_n$, paths in $\Gamma(G, G_x \cup S)$ are still paths in $\Gamma$ and the images of its edges 
have weight at most $N$, so $d_\Gamma(g G_x, h G_x) \leqslant N d_\Gamma(G, G_x \cup S)(g,h)$ for 
any $g, h$. On the other hand, any edge in $\Gamma$ has weight at least $\alpha$ and the length of 
any of its preimages in $\Gamma(G, G_x \cup S)$ is at most $3$, so 
$$d_{\Gamma(G, G_x \cup S)}(g,h) \leqslant \frac{3}{\alpha} d_\Gamma(g G_x, h G_x)$$

Finally, in order for the proof of Corollary \ref{co:rel_hyp_2} to work, it is sufficient to have 
$\frac{N}{4} \geqslant L(\delta, 1, 5\delta)$, and it can be inferred from the proofs of Lemma 3.1.9
and Theorem 3.1.4 of \cite{CoornaertDelzantPapadopoulos:1990} that $L(\delta, 1, 5\delta) = 1536 
\log_2(154) + 192 + 572 \delta$.
\end{proof}


\subsection{Application to actions on trees}
\label{subs:application_trees}

\begin{corollary}
\label{co:application_tree}
Let $G$ be a finitely generated relative to some $G_x$ group, which acts regularly on an $\R$-tree
$T$ with the property $P(n)$ for some $n \in \R, n > 0$. Then $G$ is weakly hyperbolic relative to 
$G_x$.
\end{corollary}
\begin{proof} Define $T^k$ to be the metric space $T$ with all distances multiplied by $k$ and 
take a basepoint $x^k$. It is obvious that $G$ acts on $T^k$ for any $k$, so we can define $B_n^k 
= \{g \in G \mid d(x^k, g x^k) \leqslant n\}$. It is easy to see that $B_n^k = B_{nk}$, thus, the 
action of $G$ on $T^k$ has property $P(nk)$.

Since $\delta = 0$ for all $T^k$, it follows that $a \delta^2 + b\delta + c = c$ for any such action, 
and for some $k$ we obtain $n k \geqslant c$, as required.
\end{proof}

\section{Completing hyperbolic $\Z$- and $\R$-metric \\ spaces}
\label{sec:geod_spaces}

In this section we investigate the question of ``completing'' a given non-geodesic hyperbolic
$\Z$-metric space $X$, that is, constructing a geodesic hyperbolic $\Z$-metric space
$\overline{X}$ in which $X$ (quasi-)isometrically embeds. Observe that any $\delta$-hyperbolic
$\Z$-metric space embeds isometrically into a complete geodesic $\delta$-hyperbolic $\R$-metric
space (see \cite{Bonk_Schramm:2000}) but this completion does not have to be a $\Z$-metric space.

Given a $\delta$-hyperbolic $\Z$-metric space $(X, d)$ which we fix for the rest of this section, 
below we introduce two $\Z$-completions of $X$ which we call $\Gamma_1(X)$ and $\Gamma_2(X)$. We 
shall also define an analogous construction $\beta(X)$ when $X$ is an $\R$-metric space.

Our constructions will have, compared to Bonk and Schramm's, the disadvantage that the hyperbolicity
constant will increase. However, they will have the advantage that isometries, embeddings and
quasi-isometries of $X$ extend easily and that boundaries are easy to work with.

\subsection{$\Gamma_1(X)$}
\label{subs:gamma_1}

Define a graph $\Gamma_1(X)$ as follows: to the set of points of $X$ which we call {\em essential
vertices} we add new vertices which fill ``gaps'' between essential vertices.
\begin{enumerate}
\item[(1)] Define $\Gamma_1(X) = X$, that is, all vertices of $\Gamma_1(X)$ initially are essential.

\item[(2)] For any pair $\{x, y\}$ of essential vertices with the property that there exists no $z \in
\Gamma_1(X)$ such that $d(x,y) = d(x,z) + d(z,y)$, add to $\Gamma_1(X)$ all vertices on a $\Z$-path 
of distance $d(x,y)$. The added $\Z$-paths we call {\em basic} and the new vertices we call {\em 
auxiliary vertices}. Observe that after this step, for every essential vertices of $\Gamma_1(X)$ 
there exists a $\Z$-geodesic segment (composed from auxiliary vertices) connecting them.

\item[(3)] For any triple $\{x,y,z\}$ of essential vertices, consider the projection of the triangle
$\Delta(x,y,z)$ onto the tripod $T(x,y,z)$. Every two auxiliary vertices of $\Delta(x,y,z)$ which
map into the same point of the tripod we connect by a $\Z$-path whose length is the smallest
integer larger or equal to $4\delta$ (that is, we add to $\Gamma_1(X)$ all vertices on this path)
unless there exists already a path of length less than or equal to $4\delta$ between them. The
added paths we call {\em bridges} and the new vertices we call {\em negligible vertices}.

\item[(4)] We extend the metric $d : X \to \Z$ to the metric $d : \Gamma_1(X) \to \Z$ as
follows:
\begin{enumerate}
\item[(a)] the distance between two essential vertices is inherited from $X$,

\item[(b)] the distance between an auxiliary vertex to the adjacent essential vertices is defined
by construction, hence, the distance from an auxiliary vertex to any other either essential, or
auxiliary vertex is also defined (as the minimum of lengths of paths connecting them),

\item[(c)] the distance from a negligible vertex to the adjacent auxiliary vertices is defined
by construction, so, the distance from a negligible vertex to any other vertex of $\Gamma_1(X)$ is
also defined.
\end{enumerate}
\end{enumerate}

\begin{remark}
\label{rem:gamma_1}
\begin{enumerate}
\item Observe that if $\delta = 0$ then the process of building $\Gamma_1(X)$ is equivalent to the
construction of a $\Lambda$-tree out of a $0$-hyperbolic $\Lambda$-metric space (see, for example,
\cite[Theorem 2.4.4]{Chiswell:2001}) since all bridges have length $0$.
\item If $X$ is geodesic then $\Gamma_1(X) = X$. Indeed, if $X$ is geodesic then for any essential
vertices $x$ and $y$ there is no essential vertex $z$ such that $d(x,y) = d(x,z) + d(y,z)$ only when
$d(x,y) = 1$. So, no auxiliary vertices are added. Finally, since $X$ is hyperbolic, for any pair of
auxiliary vertices there is already a path of length at most $4\delta$. Hence, no bridges are added
and, hence, no negligible vertices are added either.
\end{enumerate}
\end{remark}

\begin{remark}
\label{rem:gamma_1_unique}
Note that $\Gamma_1(X)$ is unique for a given $X$. Indeed it is easy to see that the steps $(1)$ and 
$(2)$ above do not depend on the order in which we process pairs of essential vertices. As for the 
step $(3)$, since bridges are exactly of length $4\delta$, the only cases where a bridge is not added 
is one where there already was a path which did not contain bridges, so the order in which we process 
triples of essential vertices does not matter.
\end{remark}

By \cite[Theorem 4.1]{Bonk_Schramm:2000}, $X$ isometrically embeds into a complete geodesic
$\delta$-hyperbolic $\R$-metric space $\overline{X}$. Denote the metric on $\overline{X}$ by 
$\bar{d}$. We are going to use $(\overline{X}, \bar{d})$ in our construction below.

Define a map $\varphi: \Gamma_1(X) \to \overline{X}$ as follows. First of all, observe that the set
of essential vertices of $\Gamma_1(X)$ is naturally identified with $X$, hence, it embeds into
$\overline{X}$. Next, for a pair $\{x, y\}$ of essential vertices, the basic path between them in
$\Gamma_1(X)$ can be mapped to some geodesic segment between $\varphi(x)$ and $\varphi(y)$.
Finally, for a pair $\{x, y\}$ of auxiliary vertices (whose images under $\varphi$ are already defined),
the bridge between them can also be mapped to a geodesic segment between $\varphi(x)$ and
$\varphi(y)$. Observe that $\varphi$ is not unique but it is well-defined by the construction.

\begin{lemma}
\label{le:gamma_1_1}
Let $v,w$ be vertices of $\Gamma_1(X)$.
\begin{enumerate}
\item[(i)] If $v$ and $w$ are essential then $d(v,w) = \bar{d}(\varphi(v), \varphi(w))$.
\item[(ii)] If $v$ and $w$ are auxiliary then
$$\bar{d}(\varphi(v), \varphi(w)) \leqslant d(v,w) \leqslant \bar{d}(\varphi(v), \varphi(w)) + 24\delta$$
\item[(iii)] If $v$ is essential and $w$ is auxiliary then
$$\bar{d}(\varphi(v), \varphi(w)) \leqslant d(v,w) \leqslant \bar{d}(\varphi(v), \varphi(w)) + 8\delta$$
\end{enumerate}
\end{lemma}
\begin{proof} First, notice that for any $v, w \in \Gamma_1(X)$ we have
$$d(v,w) \geqslant \bar{d}(\varphi(v), \varphi(w)).$$
Indeed, any edge of $\Gamma_1(X)$ belongs either to a basic path, or to a bridge. Basic paths are
embedded isometrically in $\overline{X}$ since $X$ is embedded isometrically in $\overline{X}$.
However, $\overline{X}$ is also $\delta$-hyperbolic, so for a pair of auxiliary vertices connected
by a bridge in $\Gamma_1(X)$, their images in $\overline{X}$ are also at a distance of at most
$4\delta$. It follows that $\varphi$ can only shorten distances.

\smallskip

(i) Obvious.

\smallskip

(iii) By the construction, $w$ is on the geodesic linking two essential vertices $w_1$ and $w_2$ in
$\Gamma_1(X)$. Consider the geodesic triangle $\Delta(v, w_1, w_2)$. Hence, $w$ is at a distance
of at most $4\delta$ from either $[v, w_1]$, or $[v, w_2]$. Let $\gamma$ be the path $[v, w'] \cup
[w', w]$, where $w' \in [v, w_i]$ for $i = 1,2$ and let $l(\gamma)$ the length of $\gamma$. Then
$[v, w']$ is isometrically embedded in $\overline{X}$ and $d(w,w') \leqslant 4\delta$, so
$\varphi|_\gamma$ is a $(1, 4\delta)$-quasi-isometry. Furthermore, $\varphi([v, w'])$ is a geodesic
and $\bar{d}(\varphi(w), \varphi(w')) \leqslant 4\delta$, so $\varphi(\gamma)$ is a
$(1,4\delta)$-quasi-geodesic. It follows that
$$d(v,w) \leqslant l(\gamma) \leqslant \bar{d}(\varphi(v),\varphi(w)) + 8\delta.$$

\begin{figure}[htbp]
\centering
\includegraphics[scale=1]{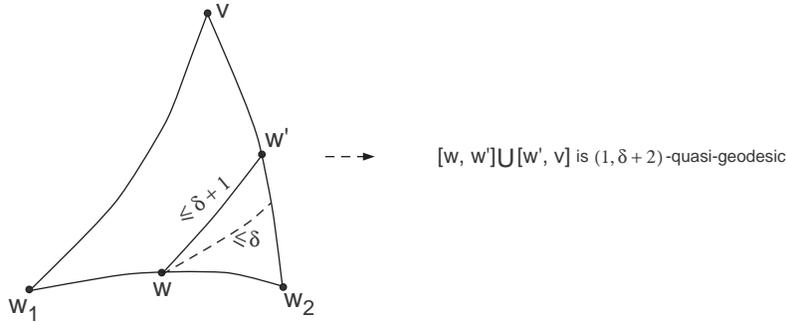}
\caption{Case (iii) in the proof of Lemma \ref{le:gamma_1_1}}
\label{pic2}
\end{figure}

(ii) By the construction, there exist essential vertices $v_1, v_2, w_1, w_2 \in \Gamma_1(X)$ such
that $v \in [v_1, v_2],\ w \in [w_1, w_2]$. Consider the geodesic square $\{v_1, v_2, w_2, w_1\}$
(linked together in the given order). Suppose $v$ is at a distance of $4\delta$ from $v' \in [v_1,
w_1]$. We can always assume this since $v$ must be at a distance of $4\delta$ from either $[v_1,
w_1]$, or $[v_2, w_1]$, and we can twist the square to fit this situation.

\begin{figure}[htbp]
\centering
\includegraphics[scale=0.8]{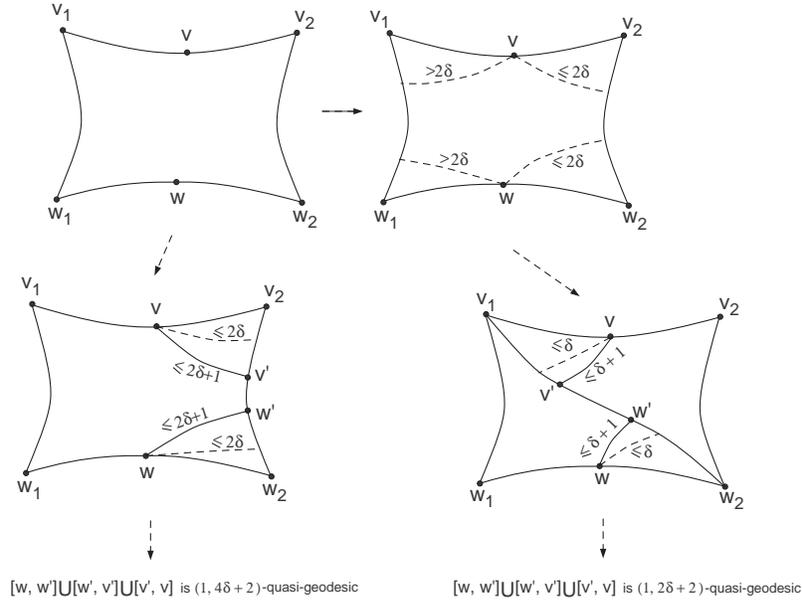}
\caption{Case (ii) in the proof of Lemma \ref{le:gamma_1_1}}
\label{pic1}
\end{figure}

If $w$ is at a distance $4\delta$ from $w' \in [v_1, w_1]$ then set $\gamma = [v, v'] \cup [v', w']
\cup [w', w]$. Otherwise, $w$ is at a distance of $4\delta$ from $w' \in [v_1, w_2]$ and $v'$ is at
a distance of $4\delta$ from either $v'' \in [v_1, w_2]$, or $v'' \in[w_1, w_2]$. Hence, we set
$\gamma = [v, v'] \cup [v', v''] \cup [v'',w'] \cup [w', w]$ in the former case, and $\gamma = [v, v']
\cup [v', v''] \cup [v'', w]$ in the latter one.

Suppose first that $\gamma = [v, v'] \cup [v', w'] \cup [w', w]$. Then $\varphi([v', w'])$ is an
isometrically embedded geodesic. Next, the other two segments of $\gamma$ have lengths of at most
$4\delta$, and so do their images. It follows that $\varphi|_\gamma$ is a $(1, 8\delta)$-quasi-isometry
and $\varphi(\gamma)$ is a $(1, 8\delta)$-quasi-geodesic. Hence,
$$d(v,w) \leqslant l(\gamma) \leqslant \bar{d}(\varphi(v), \varphi(w)) + 16\delta.$$
Now assume that $\gamma = [v, v'] \cup [v', v''] \cup [v'', w'] \cup [w', w]$. Then $\varphi([v'', 
w'])$ is an isometrically embedded geodesic. The other three segments of $\gamma$ have lengths of 
at most $4\delta$, and so do their images. It follows that $\varphi|_\gamma$ is a $(1, 
12\delta)$-quasi-isometry and $\varphi(\gamma)$ is a $(1, 12\delta)$-quasi-geodesic. Hence,
$$d(v,w) \leqslant l(\gamma) \leqslant \bar{d}(\varphi(v), \varphi(w)) + 24\delta.$$
Finally, suppose $\gamma = [v, v'] \cup [v', v''] \cup [v'', w]$. Then $\varphi([v'', w])$ is an
isometrically embedded geodesic. The other two segments of $\gamma$ have lengths of at most 
$4\delta$, and so do their images. Hence, $\varphi|_\gamma$ is a $(1, 8\delta)$-quasi-isometry and 
$\varphi(\gamma)$ is a $(1, 8\delta)$-quasi-geodesic. So
$$d(v,w) \leqslant l(\gamma) \leqslant \bar{d}(\varphi(v), \varphi(w)) + 16\delta.$$
\end{proof}

\begin{prop}
\label{pr:gamma_1_2}
$\Gamma_1(X)$ is $\delta'$-hyperbolic with $\delta '= 29\delta$.
\end{prop}
\begin{proof}  By Lemma \ref{le:gamma_1_1}, the restriction of $\varphi$ to any geodesic of
$\Gamma_1(X)$, whose endpoints are essential or auxiliary vertices, is a $(1, 24\delta)$-quasi-isometry.
Suppose then that a geodesic has negligible vertices as endpoints. Negligible vertices have always
valency $2$ and they belong to paths of length of at most $2\delta + 1$ linking auxiliary vertices
together. Thus, if $a, b$ are the endpoints, there exist auxiliary vertices $a', b'$ such that
$d(a, a'), d(b, b') \leqslant 2\delta$. It follows that, since $[a, b] \to [a', b']$ is a $(1,
4\delta)$-quasi-isometry by Lemma \ref{le:gamma_1_1} and $[a',b'] \to \overline{X}$ is a $(1,
24\delta)$-quasi-isometry, then $[a,b] \to \overline{X}$ is a $(1, 28\delta)$-quasi-isometry. Thus,
the embedding of any geodesic of $\Gamma_1(X)$ in $\overline{X}$ is a $(1, 28\delta)$-quasi-isometric.
It follows that $\varphi$ is a $(1, 28\delta)$-quasi-isometric embedding, and the result follows.
\end{proof}

\begin{lemma}
\label{le:gamma_1_3}
Let $g$ be an isometry of $X$. Then there exists an unique isometry $\bar{g}$ of $\Gamma_1(X)$ such 
that $\bar{g}|_X = g$.
\end{lemma}
\begin{proof} Define $\bar{g}$ simply by mapping basic paths to basic paths (since we know the action 
on their endpoints) and bridges to bridges. Since $g$ is an isometry of $X$, it preserves the length 
of basic paths and the sizes of triangles, so it also preserves the presence of bridges.

The uniqueness of $\bar{g}$ is pretty obvious from the construction. Since all auxiliary vertices 
have valence $3$ or $4$, unless they are at distance smaller than or equal to $2 \delta$ from an 
essential vertex, and all negligible vertices have valence $2$, and are at a distance greater than 
$2 \delta$ from any essential vertex, it follows that any isometry of $\Gamma_1(X)$ which preserves 
essential vertices must also preserve auxiliary and negligible vertices. Therefore, any extension 
of $g$ will map basic paths to basic paths and bridges to bridges, and so will be equivalent to 
$\bar{g}$.
\end{proof}

\begin{lemma}
\label{le:gamma_1_4}
Let $Y$ be a geodesic $\Delta$-hyperbolic metric space with $X$ isometrically embedded into $Y$.
Then $\Gamma_1(X)$ is quasi-isometrically embedded into $Y$ and the constants of the
quasi-isometry depend only on $\delta$ and $\Delta$
\end{lemma}
\begin{proof} Denote by $d^*$ the metric on $Y$ and let $\psi : X \to Y$ be the embedding of $X$
into $Y$. We can extend it to $\Gamma_1(X)$ in an obvious way by mapping geodesics to geodesics.
Basic paths are embedded isometrically, and the images of bridges cannot be longer than $4\Delta$.
If we have a path in $\Gamma_1(X)$, it maps to a path of length at least multiplied by $\min\left\{1,
\frac{\Delta}{\delta}\right\}$, so $\min\{1,\frac{\delta}{\Delta}\} d^*(\psi(x),\psi(y)) \leqslant d(x,y)$
for any $x,y \in \Gamma_1(X)$. On the other hand, if we have a path $\gamma$ which consists of $n$
bridges and only one segment of a basic path, since basic paths are isometrically embedded, we have
that $\psi|_\gamma$ is a $(1, n \cdot \max\{4\delta, 4(\Delta - \delta)\})$-quasi-isometric
embedding and $\psi(\gamma)$ is a $(1, n \cdot 4\Delta)$-quasi-geodesic. Using the same argument
as in the proof of Lemma \ref{le:gamma_1_1}, we have that $\psi$ extended to $\Gamma_1(X)$ is a
$(1, 12 (\Delta + \max\{\delta, \Delta - \delta\}))$-quasi-isometric embedding.
\end{proof}

\begin{lemma}
\label{le:gamma_1_5}
Let $Y$ be a $\Delta$-hyperbolic metric space with $X$ quasi-isometric to $Y$. Then $\Gamma_1(X)$ is
quasi-isometric to $\Gamma_1(Y)$ and the constants of the quasi-isometry depend only on $\delta$,
$\Delta$, and the constants of the quasi-isometry between $X$ and $Y$.
\end{lemma}
\begin{proof} Let $\psi : X \to Y$ be a $(\lambda, k)$-quasi-isometry. Recall that $\bar{d}$ is the
metric on $\Gamma_1(X)$ and denote by $\overline{d^*}$ the metric on $\Gamma_1(Y)$. We can build a
map $\Psi : \Gamma_1(X) \to \Gamma_1(Y)$ as follows. Define $\Psi(x) = \psi(x)$ for $x \in X$. Next,
we can approximate how auxiliary vertices should be mapped based on how far they are from the
endpoints of the basic paths they lie on. In other words, if $z$ is on a basic path $[x, y]$ at a
distance of $d$ from $x$, then $\Psi(z)$ should be the auxiliary vertex on $[\Psi(x), \Psi(y)]$ at a
distance of $d \cdot \frac{\bar{d}(x,y)}{\overline{d^*}(\Psi(x), \Psi(y))}$ from $\Psi(x)$. Finally,
we map bridges to bridges approximating how integer distances should be mapped. To see that $\Psi$
is a quasi-isometry, notice that for any $x,y \in X$ we have that $\Psi|_{[x,y]}$ is a
$(\lambda, k+1)$-quasi-isometry, then reuse the same paths we have above to obtain bounds for
$\bar{d}(a,b)$ in terms of $\overline{d^*}(\Psi(a), \Psi(b))$ by using the fact that bridges have
length $4\delta$, their images have length $4\Delta$ and basic paths are quasi-isometrically mapped.
\end{proof}

\begin{corollary}
\label{co:gamma_1_6}
Let $Y$ be a geodesic $\Delta$-hyperbolic metric space such that $X$ is quasi-isometrically embedded
into $Y$. Then $\Gamma_1(X)$ is quasi-isometrically embedded into $Y$.
\end{corollary}

\subsection{$\Gamma_2(X)$}
\label{subs:gamma_2}

The space $(\Gamma_1(X), d)$ introduced in the previous subsection is constructed so that geodesics
between any two essential vertices $x$ and $y$ almost never include any other essential vertices.
The only exception happens when there exists an essential vertex $z$ in $\Gamma_1(X)$ such that
$d(x,y) = d(x,z) + d(y,z)$. This property makes $\Gamma_1(X)$ an artificially ``thinned'' weighted
complete graph.

The goal of this subsection is to construct another completion $\Gamma_2(X)$ of $X$ in the case
when $X$ is {\em regular}, that is, when the following condition holds
\begin{enumerate}
\item [(RS)] $\forall\ x,y,z \in X,\ \exists\ v \in X$:
$$d(x, v) + d(v, y) \leqslant d(x, y) + 2\delta,\ \ \ \ d(x, v) + d(v, z) \leqslant d(x, z) + 2\delta,$$
$$d(y, v) + d(v, z) \leqslant d(y, z) + 2\delta.$$
\end{enumerate}
Any point $v$ from the definition above we call a {\em mid-point of $\{x,y,z\}$}.

\smallskip

The process of building $\Gamma_2(X)$ is considerably more involved but the graph itself appears
to be more natural than $\Gamma_1(X)$. At first, for $n \in \N$ we build an auxiliary graph
$\Gamma_2^n(X)$ using $\Gamma_1(X)$.

Recall that $\delta'$ is the hyperbolicity constant for $\Gamma_1(X)$ (see Proposition \ref{pr:gamma_1_2}).
Now, define $\mathcal{H}$ to be the maximal Hausdorff distance between a geodesic and a
$(4 \delta', 240 \delta'^3 +108 \delta'^2)$-quasi-geodesic in $\Gamma_1(X)$ and set $B = 2\mathcal{H}
+ 2\delta'$.

\smallskip

Set $\Gamma_2^n(X) = X$ and call all vertices of $X$ {\em essential}.
\begin{enumerate}
\item[(1)] For any essential vertices $x, y$ with $d(x,y) = n$ add to $\Gamma_2^n(X)$ a $\Z$-path
of length $n$ connecting $x$ and $y$. This path we call {\em basic} and its vertices {\em auxiliary}.

\item[(2)] For a pair of essential vertices $\{x, y\}$ such that $d(x,y) = n$ and another essential
vertex $z$, consider all mid-points $v$ for $\{x,y,z\}$. If there exists some $v$ such that $d(x,v),
d(y,v) \geqslant 2\delta$ then we remove the basic path $[x,y]$.

\item[(3)] Repeat step (2) for the pair $\{x, y\}$ and all essential vertices $z$.

\item[(4)] Repeat steps (2) and (3) for all pairs of essential vertices $\{x,y\}$ with $d(x,y) = n$.

\item[(5)] Repeat steps (1)-(4) for all integers smaller than $n$ in descending order.

\item[(6)] Finally, if the distance in $\Gamma_1(X)$ between an auxiliary vertex $x$ and a basic path
$p$, which has not been removed on previous steps, is smaller than $B$, add a $\Z$-path,
which we call a {\em bridge} and whose vertices we call {\em negligible}, connecting $x$ and the closest
to it vertex $y$ on $p$ (if there are two such vertices on $p$ then add a bridge for each one) in
$\Gamma_2^n(X)$. The length of the added bridge connecting $x$ and $y$ is equal to $d(x,y)$ in
$\Gamma_1(X)$. See Figure \ref{pic3}.

\begin{figure}[htbp]
\centering
\includegraphics[scale=0.9]{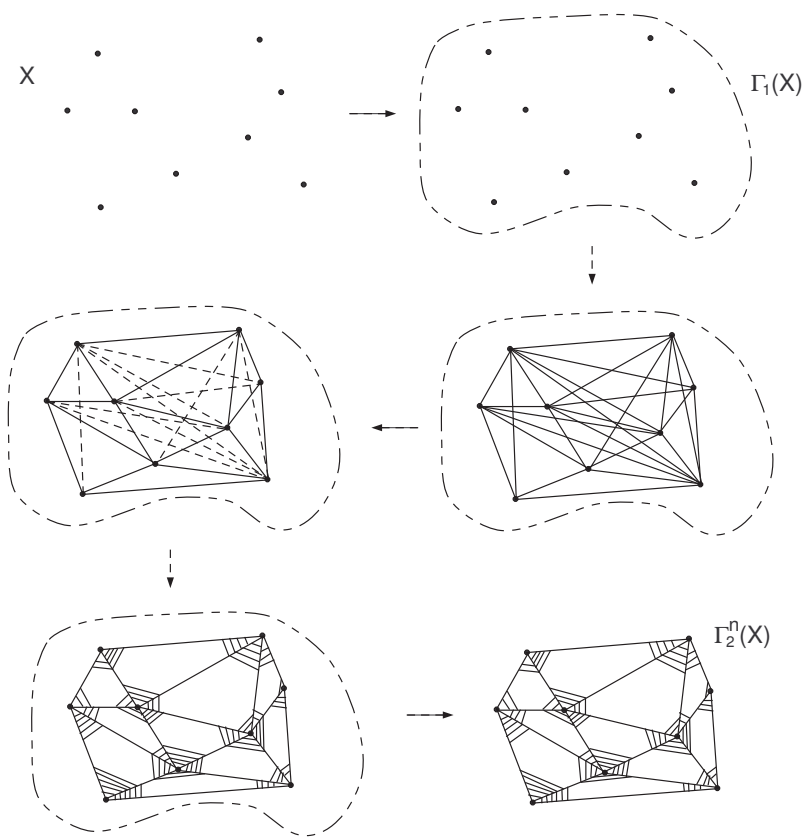}
\caption{Construction of $\Gamma_2^n(X)$}
\label{pic3}
\end{figure}

\item[(7)] We extend the metric $d : X \to \Z$ to the metric $d_2^n : \Gamma_2^n(X) \to
\Z$ as follows:
\begin{enumerate}
\item[(a)] the distance between two essential vertices is inherited from $X$,

\item[(b)] the distance between an auxiliary vertex to the adjacent essential vertices is defined
by construction, hence, the distance from an auxiliary vertex to any other either essential, or
auxiliary vertex is also defined (as the minimum of lengths of paths connecting them),

\item[(c)] the distance from a negligible vertex to the adjacent auxiliary vertices is defined
by construction, so, the distance from a negligible vertex to any other vertex of $\Gamma_2^n(X)$ 
is also defined.
\end{enumerate}
\end{enumerate}

\begin{lemma}
\label{le:gamma_2_1}
Suppose for some pair $x, y$ there exists $z$ such that there is a mid-point $v$ of $\{x,y,z\}$ 
chosen on step (2) above. Then $d(x,v) < d(x, y),\ d(y,v) < d(x,y)$.
\end{lemma}
\begin{proof} Suppose on the contrary that $d(y,v) \geqslant d(x,y)$. From regularity of $X$ it
follows that $d(x,v) \leqslant 2\delta$, hence, by construction, $v$ could not be chosen on step 2,
a contradiction.
\end{proof}

\begin{remark}
\label{rem:gamma_2_2}
\begin{enumerate}
\item From Lemma \ref{le:gamma_2_1} it follows that the algorithm of constructing $\Gamma_2^n(X)$
is correct in the sense that there is no risk of re-adding a geodesic that was previously removed and the
process ends since we always split geodesics into geodesics of strictly shorter integer length.

\item Lemma \ref{le:gamma_2_1} also explains why a geodesic is removed in the process only if there
exists a mid-point of $\{x,y,z\}$ which is sufficiently far away from $x, y$and $z$. Unfortunately,
it also means that no geodesic shorter than $2 \delta$ is ever removed (since any path through an
acceptable mid-point has length of at most $4\delta$, so it is contained in the $2\delta$-neighborhood
of the two points), so, $\delta$-neighborhoods of points in $\Gamma_1(X)$ and $\Gamma_2^n(X)$ are
essentially the same.

\item Finally, from Lemma \ref{le:gamma_2_1} it follows that $\Gamma_2^n(X)$ is connected since a
basic path is removed only if there are two shorter paths connecting the same vertices which are
added later on.
\end{enumerate}
\end{remark}

Observe that $\Gamma_2^n(X),\ n \in \N$ can be viewed as a graph whose vertices $V(\Gamma_2^n(X))$ 
are points of $\Gamma_2^n(X)$ and edges $E(\Gamma_2^n(X))$ are pairs of points at distance $1$ from 
each other.

\begin{lemma}
\label{le:gamma_2_3}
If $n < m$ then
$$V(\Gamma_2^n(X)) \subseteq V(\Gamma_2^m(X))\ \ {\rm and}\ \ E(\Gamma_2^n(X)) \subseteq
E(\Gamma_2^m(X)).$$
\end{lemma}
\begin{proof} In the process of building $\Gamma_2^m(X)$ we eventually run steps (1)-(5) for $m$.
Thus, we add all the geodesics we would add and remove all the geodesics we would remove in the 
process of building $\Gamma_2^n(X)$ (since we always split geodesics into shorter ones). Since the 
set of geodesics between essential vertices in $\Gamma_2^n(X)$ is contained in the set of geodesics 
between essential vertices of $\Gamma_2^m(X)$, the same relation is going to hold for the sets of 
bridges in both graphs. The required statement follows.
\end{proof}

Observe that Lemma \ref{le:gamma_2_3} does not imply that $\Gamma_2^n(X)$ isometrically embeds 
into $\Gamma_2^m(X)$. At the same time, it implies that if $X$ is not bounded then the sequence
$\{\Gamma_2^n(X)\}_{n \in \N}$ converges to a graph $\Gamma_2(X)$ whose vertices and edges are 
defined by
$$V(\Gamma_2(X)) = \bigcup_{n \in \N} V(\Gamma_2^n(X),\ \ \ E(\Gamma_2(X)) = \bigcup_{n \in
\N} E(\Gamma_2^n(X)).$$
If $X$ is bounded and has diameter $d$ then we set $\Gamma_2(X) = \Gamma_2^d(X)$.

\begin{remark}
\label{rem:gamma_2_3a}
If $X$ is geodesic then $\Gamma_2(X) = X$. Indeed, if $X$ is geodesic then all basic paths are 
removed except those of length $1$, and all bridges already exist since $X$ is $\delta$-hyperbolic.
\end{remark}

\begin{remark}
\label{rem:gamma_2_unique}
Given a specific $X$, then there will exist an unique extension $\Gamma_2(X)$. This is less obvious 
than in the case of $\Gamma_1$, but still pretty easy to see. Notice that, in the construction of 
$\Gamma_2^n(X)$, steps (1) and (2) do not depend on the order in which we consider pairs of 
essential vertices, steps (3)--(5) are a repetition of steps (1) and (2), and step (6) depends 
entirely on the image of the graph built up until then in $\Gamma_1(X)$, which we have already 
shown to be unique.

$\Gamma_2(X)$, being an union of uniquely determined extensions of $X$, will itself be an unique 
extension of $X$.
\end{remark}

Suppose $n \in \N$. Let ${\cal N}$ be a set which initially contains only $n$ and we apply to all 
elements of ${\cal N}$ the following steps. For each $k \in {\cal N}$, if $k > 2 \delta$ then replace 
it by two new numbers $k_1$ and $k_2$ such that $k_1, k_2 < k$ and $k \leqslant k_1 + k_2 \leqslant 
k + 2\delta$. Continue this splitting process until all elements of ${\cal N}$ are smaller than $2 
\delta$. Denote by $\tau(n)$ the maximal sum obtained by the above algorithm starting from $n$.

\begin{lemma}
\label{le:gamma_2_4}
$\tau(n) \leqslant 4\delta n$.
\end{lemma}
\begin{proof} Observe that for $n \leqslant 2\delta$ we have $\tau(n) = n \leqslant 4\delta n$.
Next, for $n > 2\delta$ there exist $i, j \in \N$ such that $\tau(n) = \tau(i) + \tau(j)$ with 
$j \leqslant n - i + 2\delta$. We are gong to prove that for $n > 2\delta$
$$\tau(n) \leqslant 4\delta n - 8 \delta^2 < 4\delta n.$$
Assume without loss of generality that $j \geq i$.

\smallskip

First of all, for $n = 2\delta + 1$, any $n_1, n_2$ will be smaller than $2 \delta$, so the ideal
is to use $n_1 = n_2 = 2\delta$ giving us $\tau(2 \delta + 1) = 4\delta = 4\delta (2\delta + 1)
- 8\delta^2$.

\smallskip

Suppose now that $\tau(k) \leq 4\delta k - 8\delta^2$ for any $2\delta + 1 \leqslant k \leqslant
n - 1$. We have $\tau(n) = \tau(j) + \tau(i)$ with $n - 1 \geqslant j \geqslant i$. It is worth
noticing that $j \leqslant n-1$, so $j + 2 \delta + 1 \leqslant n + 2\delta$. Thus, if we can choose
$j > 2\delta$, we can choose $i > 2\delta$ as well. Since $\tau$ is an increasing function, if $n
> 2\delta + 1$ then we can assume that $i, j > 2\delta$. So, we have that
$$\tau(n) = \tau(i) + \tau(j) \leqslant 4\delta i + 4\delta j - 16\delta^2 \leqslant 4\delta
(n - j + 2\delta) + 4\delta j - 16 \delta^2$$
$$ = 4\delta n - 8 \delta^2 \leqslant 4\delta n.$$
\end{proof}

For any $v,w \in \Gamma_2(X)$ define $[v,w]_2$ to be a geodesic between $v$ and $w$ either in
$\Gamma_2(X)$ and $\varphi[v,w]_2$ its embedding into $\Gamma_1(X)$. To simplify notation, let
$d_2$ represent the graph metric on $\Gamma_2(X)$ and $d_1$ the metric on $\Gamma_1(X)$. Let $k =
30 \delta'^2$. If $[x,y]_2$ contains no other essential vertices so that it coincides with a geodesic in
$\Gamma_1(X)$, we say that $[x,y]_2$ is {\em long} if $d(x,y) > k$ and {\em short} otherwise.

\begin{lemma}
\label{le:gamma_2_5}
\begin{enumerate}
\item[(i)] Let $[x,y]_2$ be long and $z$ be an essential vertex. Then either $\varphi([x,y]_2 \cup 
[x,z]_2)$, or $\varphi([x,y]_2 \cup [y,z]_2)$ is a $(1, 6\delta)$-quasi-geodesic.

\item[(ii)] Let $[x,x']_2$ and $[y,y']_2$ be long. Then there exist $\chi \in \{x, x'\}$ and $\psi
\in \{y, y'\}$ such that $\varphi([x,x']_2 \cup [\chi,\psi]_2 \cup [y,y']_2)$ is a 
$(1,12\delta)$-quasi-geodesic.
\end{enumerate}
\end{lemma}
\begin{proof} (i) Observe that $[x,y]_2$ is long and it has not been split into two segments in the
process of building the graph. It implies that for any essential vertex $v$, any mid-point of $\{x,
y, v\}$ is within a $2\delta$-ball centered either at $x$, or at $y$. Suppose, without loss of
generality, that there exists a mid-point $m$ of $\{x,y,z\}$ in the $2\delta$-neighborhood of $x$.
We have that $\varphi([z,m]_2 \cup [m,y]_2)$ is a $(1, 2\delta)$-quasi-geodesic and that $d(x,m) 
\leqslant 2\delta$, but this implies that $\varphi([z,m]_2 \cup [m,x]_2 \cup [x,m]_2 \cup [m,y]_2)$ 
is a $(1, 6\delta)$-quasi-geodesic. Hence, so is $\varphi([z,x]_2 \cup [x,y]_2)$.

\begin{figure}[htbp]
\centering
\includegraphics[scale=0.7]{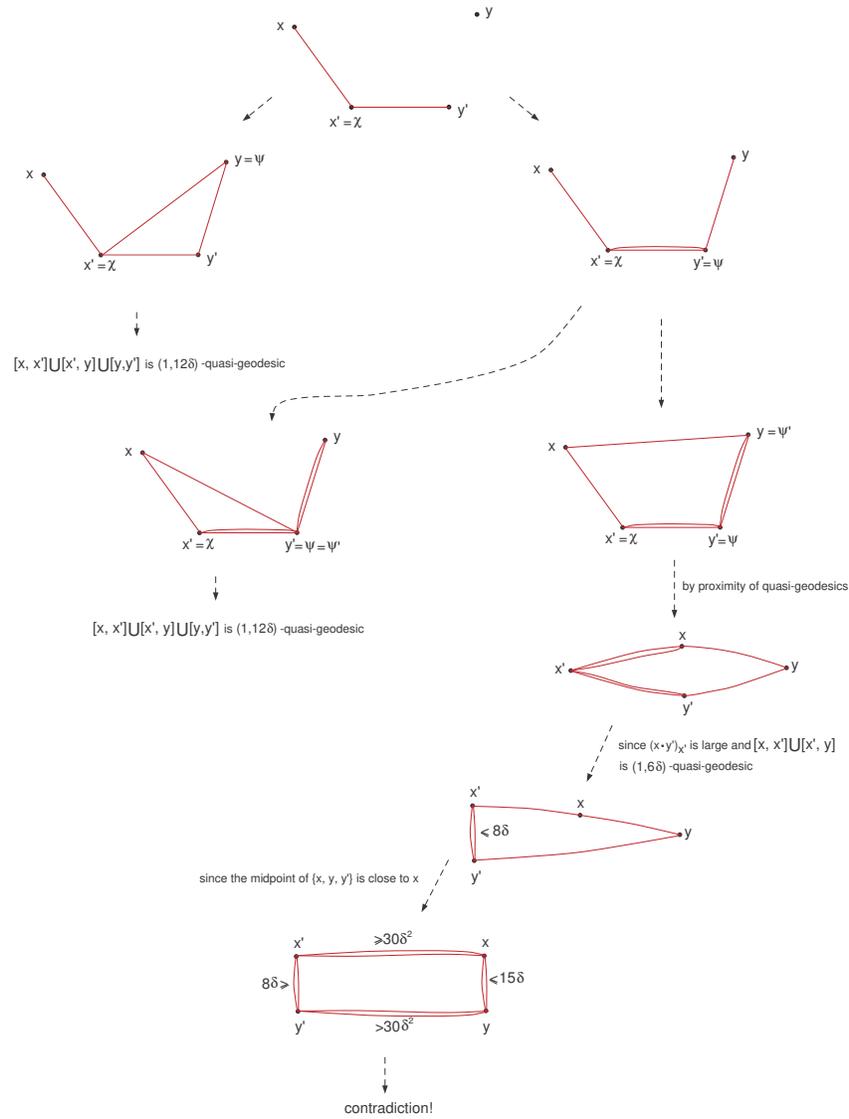}
\caption{Red curves here are $(1, 6\delta)$-quasi-geodesics.}
\label{pic4}
\end{figure}

(ii) By the above result, there exists $\chi \in \{x, x'\}$ such that $\varphi([y', \chi]_2 \cup 
[x, x']_2)$ is a $(1, 6\delta)$-quasi-geodesic. Suppose, without loss of generality, that $\chi = 
x'$. At the same time, similarly there exists $\psi \in \{y, y'\}$ such that $\varphi([x',\psi]_2 
\cup [y',y]_2)$ is a $(1, 6\delta)$-quasi-geodesic. If $\psi = y$ then the path $\varphi([x,x']_2 
\cup [x',y]_2 \cup [y,y']_2)$ is simply the $(1, 6\delta)$ quasi-geodesic $\varphi([y',x']_2 \cup 
[x',x]_2)$ with the segment $\varphi([y', x']_2)$ replaced by a $(1, 6\delta)$-quasi-geodesic. It 
follows that the path in question is a $(1, 12\delta)$ quasi-geodesic.

Suppose that $\chi = x$. There exists some $\psi'$ such that $\varphi([x, \psi']_2 \cup [y,y']_2)$ 
is a $(1, 6\delta)$-quasi-geodesic. If $\psi' = y'$ then the path $\varphi([x, x']_2 \cup [x',y']_2 
\cup [y',y]_2)$ is merely the $(1,6\delta)$-quasi-geodesic $\varphi([x,y']_2 \cup [y',y]_2)$ with 
the segment $\varphi([x, y']_2)$ replaced by a $(1, 6\delta)$-quasi-geodesic. It follows that the 
path in question is a $(1, 12\delta)$-quasi-geodesic.

Suppose that $\chi = x',\ \psi = y',\ \psi' = y$. Both $\varphi([y',y]_2 \cup [y,x]_2)$ and 
$\varphi([y',x']_2 \cup [x',x]_2)$ are $(1, 6\delta)$-quasi-geodesic, so, $c_{y'}(x',x) \geqslant 
d(y',x') - 3\delta$ and $c_{y'}(y,x) \geqslant d(y',y) - 3\delta$, which implies that $c_{y'}(y,x') 
\geqslant \min\{d(y',y), d(y',x')\} - 2\delta' - 3\delta\geqslant  \min\{d(y',y), d(y',x')\} -
5\delta'$. But $\varphi([y, y']_2 \cup [y',x']_2)$ is a $(1, 6\delta)$-quasi-geodesic, which means 
that $d(x',y) \geqslant d(x',y') + d(y',y) - 6\delta$. Now, $d(x',y') = d(x',y') + d(y',y) - 
2c_{y'}(x',y)$, so,
$$3\delta' \geqslant 3\delta \geqslant c_{y'}(x',y) \geqslant \min\{d(y',y), d(y',x')\} - 5\delta',$$
so,
$$8\delta'  \geqslant c_{y'}(x',y) \geqslant \min\{d(y',y), d(y',x')\}.$$
Since $[y,y']$ is long, it follows that $d(y',x') \leqslant 8\delta'$.

Consider $\Delta(x',y,x)$. We have that $\varphi([y',y]_2 \cup [y,x]_2)$ is a $(1, 6 
\delta)$-quasi-geodesic, so $y$ is at a distance of at most $3\delta$ from $\Delta_I(x',y,x)$, and 
so it is at a distance of at most $3\delta + \delta' \leqslant 4\delta'$ from a point $z \in 
\varphi([y',x]_2)$. Furthermore, since $d(x',y') \leqslant 8\delta'$, letting $z'$ be a point on 
$\varphi([y',x']_2 \cup [x',x]_2)$ at a distance of smaller than $\delta'$ from $z$, we get $z' \in 
[y',y]$ since $30 \delta'^2 \leqslant d(y,y') \leqslant d(y,z') + d(z',y') \leqslant 5\delta' + 
d(z',y')$ and $d(x',y') \leqslant 13\delta'$. This implies that $[x',y] \cup [y,x]$ is a 
$(1, 10\delta')$-quasi-geodesic. By the alternative definition of regularity  that any mid-point of 
$\{x',y,x\}$ must be within the $7\delta'$-neighborhood of $\Delta_I(x',y,x)$. But since $[x',y] 
\cup [y,x]$ is a $(1,10\delta')$-quasi-geodesic, it follows that the vertices of $\Delta_I(x',y,x)$ 
which lie on $[x',y]$ and $[x,y]$ are at a distance of at most $5\delta'$ from $y$. So, $\Delta_I(x', 
y,x)$ must be within the $6 \delta'$-ball around $y$, so any mid-point of $\{x',y,x\}$ is within the
$13 \delta'$-ball around $y$.

Now, from the fact that is $[x,x']$ is long it follows that any such mid-point must be within
$2\delta'$ of either $x'$, or $x$. It cannot be near $x'$ since $d(x'y) \geqslant k - 8 \delta' >
10 \delta'$, so, any such mid-point must be at a distance of at most $2 \delta'$ from $x$, and It
follows that $d(x,y) \leqslant 15 \delta'$. However, $d_2(x,y) \geqslant 2k > 4\delta' \cdot 15
\delta' \geqslant 4\delta' d(x,y)$, a contradiction. Thus, the only possible cases occur when there
exists one such path which is a $(1,12\delta)$-quasi-geodesic.
\end{proof}

\begin{lemma}
\label{le:gamma_2_6}
Let $v,w$ be essential vertices. The embedding of $[v,w]_2$ into $\Gamma_1$ is a $(4\delta, 48 
\delta^2 + (8 \delta + 2)k)$-quasi-geodesic.
\end{lemma}
\begin{proof} By Lemma \ref{le:gamma_2_4}, for any two essential vertices $v$ and $w$ we have
$d_1(v,w) \leqslant d_2(v,w) \leqslant 4\delta d_1(v,w)$. Suppose that $a$ and $b$ are vertices
such that $\varphi([v,a]_2 \cup [a,b]_2 \cup [b,w]_2) = \varphi([v,w]_2)$. We can assume that there 
exist no essential vertices between $a$ and $v$, or $b$ and $w$, and that $a \in [v,v']_2,\ b \in 
[w,w']_2$ with $v',w'$ being essential vertices. We have  three cases.

\smallskip

{\bf Case I.} If both $[v, v']$ and $[w, w']$ are short then we have
$$d_1(\varphi(a),\varphi(b)) \leqslant d_2(a,b) = d_2(a,v') + d_2(v',w') + d_2(w',b)$$
$$\leqslant d_1(\varphi(a),v') + 4\delta d_1(v',w') + d_1(w',\varphi(b))$$
$$\leqslant 4\delta(d_1(a,b) + 2k) + 2k = 4\delta d_1(a,b) + (8\delta + 2) k.$$

\smallskip

{\bf Case II.} If $[v,v']_2$ is short and $[w,w']_2$ is long (without loss of generality) then that 
there exists $\omega \in \{w,w'\}$ such that $\varphi([v', \omega]_2 \cup [w,w']_2)$ is a $(1, 
6\delta)$-quasi-geodesic. Thus, we have
$$d_2(a,b) \leqslant d_2(a, v') + d_2(v', \omega) + d_2(\omega, b) \leqslant d_1(a, \varphi(v')) + 
4\delta d_1(v', \omega) + d_1(\omega, \varphi(b))$$
$$\leqslant k + 4 \delta (d_1(v', \omega) + d_1(\omega, \varphi(b))) \leqslant k + 4 \delta(d_1(v',
\varphi(b)) + 6\delta)$$
$$\leqslant k + 4\delta(d(a,b) + k + 6\delta) = 4\delta d(a,b) + 24\delta^2 + (1 + 4\delta) k$$

\smallskip

{\bf Case III.} If both $[v,v']_2$ and $[w,w']_2$ are long then there exist $\phi \in \{v, v'\}$ and 
$\omega \in \{w, w'\}$ such that $\varphi([v, v']_2 \cup [\phi, \omega]_2 \cup [w,w']_2)$ is a 
$(1, 12\delta)$-quasi-geodesic. Thus, we have
$$d_2(a,b) \leqslant d_2(a, \phi) + d_2(\phi, \omega) + d_2(\omega, b) \leqslant d_1(\varphi(a), 
\phi) + 4\delta d_1(\phi, \omega) + d_1(\omega, \varphi(b))$$
$$\leqslant 4\delta(d_1(\varphi(a), \phi) + d_1(\phi, \omega) + d_1(\omega, \varphi(b))) \leqslant
4 \delta(d_1(\varphi(a), \varphi(b)) + 12 \delta)$$
$$= 4\delta d_1(\varphi(a), \varphi(b)) + 48 \delta^2.$$
\end{proof}

\begin{remark}
\label{rem:gamma_2_7}
Since for any essential vertices $x, y$, the embedding of $[x,y]_2$ into $\Gamma_1(X)$ is a $(4\delta, 
240 \delta'^3 + 60 \delta'^2 + 48 \delta^2)$-quasi-geodesic, there exists $\mathcal{H}$ such that
the Hausdorff distance between $\varphi[x,y]_2$ and $[x,y]_1$ is at most $\mathcal{H}$.
\end{remark}

\begin{lemma}
\label{le:gamma_2_8}
Let $v,w$ be vertices of $\Gamma_2(X)$.
\begin{enumerate}
\item[(i)] If $v$ and $w$ are both auxiliary, then $[v,w]_2$ is a $(4 \delta, 240 \delta'^3 + 60 
\delta'^2 + 48 \delta^2 + 2\mathcal{H} + \delta' + 1)$-quasi-geodesic.

\item[(ii)] If $v$ is essential and $w$ is auxiliary then $[v,w]_2$ is a $(4 \delta', 240 \delta'^3 
+ 60 \delta'^2 + 48 \delta^2 + 4\mathcal{H} + 4 \delta' + 2)$-quasi-geodesic.
\end{enumerate}
\end{lemma}
\begin{proof} (ii) By the construction, $w$ is on the geodesic in $\Gamma_2(X)$ connecting two
essential vertices $w_1$ and $w_2$. Consider the geodesic triangle $\Delta(v, w_1, w_2)$. We have
that $w$ is at a distance of at most $H$ from some point on $[w_1, w_2]$ which is at a distance of
at most $\delta'$ from either $[v, w_1]$, or $[v, w_2]$, which itself is at a distance of at most
$\mathcal{H}$ from $[v, w_i]_2$. So, $w$ is at a distance of at most $2 \mathcal{H} + \delta'$ from
either $[v, w_1]_2$, or $[v, w_2]_2$. By the construction, we have a bridge in $\Gamma_2(X)$ of
length at most $2 \mathcal{H} + \delta' + 1$ between $w$ and some auxiliary vertex $w'$ on $[v,
w_i]_2$. It follows that the embedding of $[w, w']_1 \cup [w',v]_1$ is a $(4\delta, 240 \delta'^3 + 
60 \delta'^2 + 48 \delta^2 + 2 \mathcal{H} + \delta' + 1)$-quasi-geodesic.

\begin{figure}[htbp]
\centering
\includegraphics[scale=0.7]{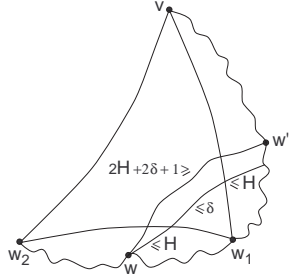}
\caption{Case (ii) in the proof of Lemma \ref{le:gamma_2_8}. Wavy lines are quasi-geodesics}
\label{pic6}
\end{figure}

(i) There exist essential vertices $v_1, v_2, w_1, w_2$ such that $v \in [v_1, v_2]_2,\ w \in [w_1,
w_2]_2$. Consider the geodesic square $\{v_1, v_2, w_2, w_1\}$ (linked together in the given order).
Thus, $v$ and $w$ must be within a distance of $\mathcal{H}$ from some points $v'$ and $w'$ on
$[v_1, v_2]$ and $[w_1, w_2]$ respectively. Furthermore, both $v'$ or $w'$ must be within a distance
of $2\delta'$ from either $[v_1, w_1]$, or $[v_2, w_2]$. If they are both within $2\delta'$ from
the same edge, say $[v_1, w_1]$, then they are within $2 \delta' + 2\mathcal{H}$ from $[v_1, w_1]_2$,
so that both $v$ and $w$ are within at most $2\mathcal{H} + 2\delta'$ from $[v_1,w_1]_2$. Hence,
there are bridges of length at most $2\delta' + 2\mathcal{H} + 1$ between them and that geodesic,
thus, $[v,w]_2$ is a $(4 \delta', 240 \delta'^3 + 108 \delta'^2 + 4 \mathcal{H} + 
4 \delta' + 2)$-quasi-geodesic.

Suppose on the contrary that, without loss of generality, $v'$ is within $2\delta'$ from $[v_1, w_1]$,
$w'$ is within $2\delta'$ from $[v_2, w_2]$, and neither one is within $2\delta'$ of the other edge.
Consider the triangles $\Delta(v_1, v_2, w_2)$ and $\Delta(v_1, w_1, w_2)$. By the assumption,
both $v'$ and $w'$ must be within $\delta'$ from $[v_1, w_2]$, which itself is at a Hausdorff distance
of at most $\mathcal{H}$ from $[v_1, w_2]_2$. Thus, there exist bridges of length at most
$2 \mathcal{H} + \delta' + 1$ from $v$ and $w$ to some $v''$ and $w''$ on $[v_1, w_2]_2$. It follows
that the embedding of $[v, w]_2$ is a $(4 \delta', 240 \delta'^3 + 108 \delta'^2 + 4 \mathcal{H} +
2 \delta' + 1)$-quasi-geodesic.
\end{proof}

\begin{figure}[htbp]
\centering
\includegraphics[scale=0.7]{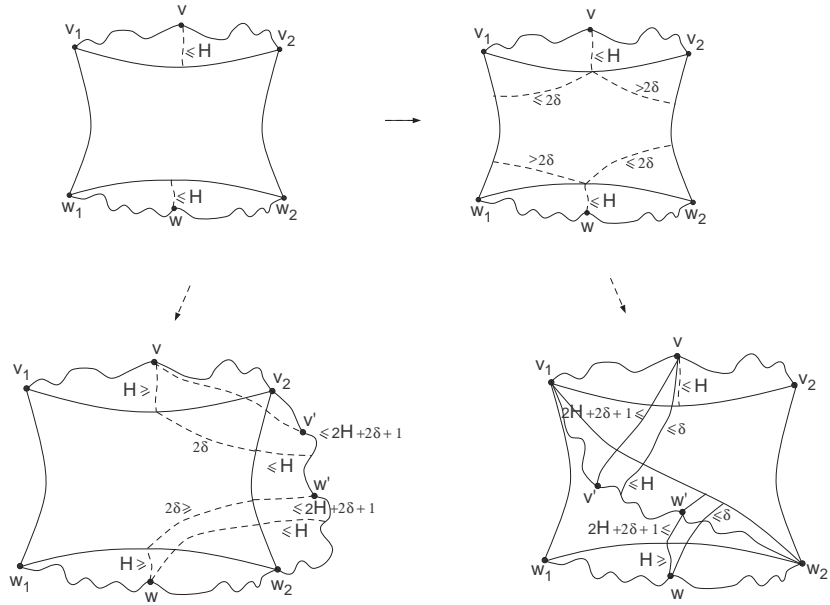}
\caption{Case (i) in the proof of Lemma \ref{le:gamma_2_8}. Wavy lines are quasi-geodesics}
\label{pic5}
\end{figure}

\begin{prop}
\label{pr:gamma_2_9}
$\Gamma_2(X)$ is $\delta''$-hyperbolic with $\delta'' = 240 \delta'^3 + 64 \delta'^2 + 48 \delta^2 
+ 8\mathcal{H} + 8 \delta' + 2$.
\end{prop}
\begin{proof} Let $\varphi : \Gamma_2(X) \to \Gamma_1(X)$ be the function sending vertices of
$\Gamma_2(X)$ to their embedding in $\Gamma_1(X)$. We have that the embedding of any geodesic of
$\Gamma_2(X)$ whose endpoints are either essential, or auxiliary vertices is a $(4 \delta, 240 
\delta'^3 + 60 \delta'^2 + 48 \delta^2 + 4\mathcal{H} + 4 \delta' + 2)$-quasi-geodesic. Suppose then 
that a geodesic has negligible vertices as endpoints. Negligible vertices always have valency $2$ 
and they belong to paths of length at most $2 \mathcal{H} + 2\delta' + 1$ connecting auxiliary 
vertices. Thus, if $a, b$ are the endpoints, there exist auxiliary vertices $a', b'$ such that $[a, 
b]_2 = [a, a']_2 \cup [a',b']_2 \cup [b',b]_2$ with $d(a,a'), d(b,b') \leqslant 2 \mathcal{H} + 
2\delta'$. Thus, the embedding of any geodesic of $\Gamma_2(X)$ into $\Gamma_1$ is a $(4 \delta', 
240 \delta'^3 + 108 \delta'^3 + 8 \mathcal{H} + 8 \delta' + 2)$-quasi-geodesic. It follows that 
$\varphi$ is a $(4 \delta, 240 \delta'^3 + 60 \delta'^2 + 48 \delta^2 + 8 \mathcal{H} + 8 \delta' + 
2)$-quasi-isometric embedding, and the result follows.
\end{proof}

\begin{remark}
\label{rem:gamma_2_10}
Since the maximal Hausdorff distance between a geodesic and a $(4\delta, 240 \delta'^3 + 60 
\delta'^2 + 48 \delta^2)$-quasi-geodesic in a $\delta'$-hyperbolic space is polynomial in $\delta'$ 
and $\delta$, and remembering that $\delta'$ is polynomial in $\delta$, it follows that $\delta''$ 
depends polynomially on $\delta$.
\end{remark}

\begin{corollary}
\label{co:gamma_2_11}
$\Gamma_2(X)$ is quasi-isometrically embedded into $\Gamma_1(X)$ and all universal properties of
$\Gamma_1(X)$ extend to $\Gamma_2(X)$.
\end{corollary}

\begin{lemma}
\label{le:gamma_2_12}
Let $g$ be an isometry of $X$. There exists an unique isometry $\bar{g}$ of $\Gamma_2(X)$ such that
$\bar{g}|_X = g$
\end{lemma}
\begin{proof} If $g$ is an isometry of $X$ then it preserves mid-points of triangles. It follows that
if $[x,y]$ is removed, so is $g [x,y]$. Finally, since by Lemma \ref{le:gamma_1_3}, $g$ can be
extended to $\Gamma_1(X)$ and that bridges are placed based upon proximity in $\Gamma_1(X)$, then
if we have a bridge between $a$ and $b$, we also have a bridge between $g a$ and $g b$ defined by
extending the isometry on the essential vertices to the basic paths joining them.

Uniqueness of $\bar{g}$ follows easily from the fact that it is defined entirely by the action of 
$g$ on $X$ and the fact that the extension of $g$ to $\Gamma_1(X)$ is itself unique.
\end{proof}

\subsection{$\beta(X)$}
\label{subs:beta}

If we have a hyperbolic $\R$-metric space, it is possible to define an embedding into a geodesic 
hyperbolic $\R$-metric space analogous to the construction of $\Gamma_1(X)$ in the discrete case. 
Since the two constructions are very similar, the proofs will often be done by reference to the case
of $\Gamma_1(X)$. The same terminology will be used to make transferring those proofs easier. Notice 
that a continuous analog of $\Gamma_2(X)$ is not possible. There is no guarantee that our algorithm 
of breaking down geodesics ever stops since arbitrarily small distances can occur.

Let now $X$ be an $\R$-metric space. Define a band complex $\beta(X)$ as follows:
\begin{enumerate}
\item[(1)] Define first $\beta(X) = X$ and define these points as {\em essential points}.

\item[(2)] For any pair $(x,y)$ of essential points, add to our complex a copy of the interval $[0, 
d(x,y)]$ with endpoints $x$ and $y$ unless there exists some $z$ such that $d(x,y) = d(x,z) + d(z,
y)$. Let us call these lines {\em basic paths} and the points that are on them {\em auxiliary}. Take 
the completion of the weighted graph so obtained. Observe that after this step, for every essential 
points of $\beta(X)$ there exists an $\R$-geodesic segment (composed from auxiliary vertices) 
connecting them.

\item[(3)] For any triple $\{x, y, z\}$ of essential vertices, consider the projection of the triangle
$\Delta(x,y,z)$ onto the tripod $T(x,y,z)$. Attach bands of length $4 \delta$ to the basic paths 
linking together the points that are mapped together on the tripod, except for those at the distance 
less than $2 \delta$ from $x$, $y$, or $z$. Define the fibers of the bands to be {\em bridges} and 
the points that make them up {\em negligible}.

\item[(4)] We extend the metric $d : X \to \R$ to the metric $\widehat{d} : beta(X) \to \R$ as
follows:
\begin{enumerate}
\item[(a)] the distance between two essential points is inherited from $X$,

\item[(b)] the distance between an auxiliary point to the adjacent essential points is defined by 
construction, hence, the distance from an auxiliary point to any other either essential, or auxiliary 
point is also defined (as the minimum of lengths of paths connecting them),

\item[(c)] the distance from a negligible point to the adjacent auxiliary points is defined by 
construction, so, the distance from a negligible point to any other point of $\beta(X)$ is also 
defined.
\end{enumerate}
\end{enumerate}

\begin{remark}
\label{rem:beta_unique}
As in the case of $\Gamma_1$ and $\Gamma_2$, every given space $X$ has a unique extension $\beta(X)$. 
The reasoning is the same as for $\Gamma_1$.
\end{remark}

\begin{lemma}
\label{le:beta_1}
Let $v,w$ be points of $\beta(X)$.
\begin{enumerate}
\item[(i)] If $v$ and $w$ are essential then $d(v,w) = \widehat{d}(\varphi(v), \varphi(w))$.

\item[(ii)] If $v$ and $w$ are auxiliary then
$$\widehat{d}(\varphi(v), \varphi(w)) \leqslant d(v,w) \leqslant \widehat{d}(\varphi(v), \varphi(w)) 
+ 24\delta$$

\item[(iii)] If $v$ is essential and $w$ is auxiliary then
$$\widehat{d}(\varphi(v), \varphi(w)) \leqslant d(v,w) \leqslant \widehat{d}(\varphi(v), \varphi(w)) 
+ 8\delta$$
\end{enumerate}
\end{lemma}
\begin{proof} The proof is a straightforward adaptation of the proof of Lemma \ref{le:gamma_1_1}. 
All the arguments work in the exact same way.
\end{proof}

\begin{prop}
\label{pr:beta_2}
$\beta(X)$ is $\delta'$-hyperbolic with $\delta' = 29 \delta$.
\end{prop}
\begin{proof}
Let $\varphi : \beta(X) \rightarrow \bar{X}$ be the function sending points of $X \subseteq \beta(X)$ 
to their embedding in $\bar{X}$ extended by mapping basic paths and fibers of bands to geodesics of 
$\bar{X}$.

The embedding of any geodesic of $\beta(X)$ whose endpoints are essential or auxiliary is a 
$(1,24\delta)$-quasi-isometry. Suppose then that a geodesic has endpoints which are on bands. Bands 
are always of length $4 \delta$ and link together essential and auxiliary points. Thus, if $a, b$ 
are the endpoints, there exist auxiliary points $a', b'$ such that $d(a, a'),\ d(b, b') \leqslant 
2 \delta$. It follows that, since $[a,b] \rightarrow [a',b']$ is a $(1, 4 \delta)$-quasi-isometry and 
$\varphi : [a',b'] \rightarrow \bar{X}$ is a $(1, 24 \delta)$-quasi-isometry, $\varphi : [a,b] 
\rightarrow\bar{X}$ is a $(1, 28 \delta)$-quasi-isometry.

Thus, the embedding of any geodesic of $\beta(X)$ into $\bar{X}$ is a $(1, 28 \delta)$-quasi-geodesic. 
It follows that $\varphi$ is a $(1, 28 \delta)$-quasi-isometric embedding, and the result follows easily.
\end{proof}

\begin{lemma}
\label{le:beta_3}
\begin{enumerate}
\item[(i)] Let $g$ be an isometry of $X$. There exists a unique isometry $\overline{g}$ of $\beta(X)$ 
such that $\overline{g}|_{X} = g$.

\item[(ii)] If $\delta = 0$, our construction is equivalent to the one given in \cite[Theorem 
2.4.4]{Chiswell:2001}.
\end{enumerate}
\end{lemma}
\begin{proof} Both statements are quite obvious. Define $\overline{g}$ simply by mapping basic paths 
to basic paths (since we know the action on their endpoints) and bands to bands. Since $g$ is an 
isometry of $X$, it preserves the length of basic paths and the size of triangles, so that it also 
preserves the presence of bands. Uniqueness is proven in a way analogous to Lemma \ref{le:gamma_1_3}. 

If $\delta = 0$, then bands have length $0$, so we simply identify points in the same way as shown
in the proof of \cite[Theorem 2.4.4]{Chiswell:2001}.
\end{proof}

\begin{lemma}
\begin{enumerate}
\item[(i)] Let $Y$ be a geodesic $\Delta$-hyperbolic metric space with $X$ isometrically embedded 
in $Y$. Then $\beta(X)$ is quasi-isometrically embedded in $Y$ and the constants of the quasi-isometry 
depend only on $\delta$ and $\Delta$.

\item[(ii)] Let $Y$ be a $\Delta$-hyperbolic metric space with $X \simeq Y$. Then $\beta(X) \simeq 
\beta(Y)$ and the constants of the quasi-isometry depend only on $\delta$, $\Delta$, and the 
constants of the quasi-isometry between them.

\item[(iii)] Let $Y$ be a geodesic $\Delta$-hyperbolic metric space and $X$ quasi-isometrically 
embedded in $Y$. Then $\beta(X)$ is quasi-isometrically embedded in $Y$.
\end{enumerate}
\end{lemma}
\begin{proof} The proofs are similar to the proofs of Lemma \ref{le:gamma_1_4} and Lemma 
\ref{le:gamma_1_5}. Transfer of the proofs is very straightforward.
\end{proof}

\subsection{Boundaries $\partial \Gamma_1(X), \partial \Gamma_2(X)$ and $\partial \beta(X)$}
\label{subs:bound}

\begin{prop}
\label{pr:boundaries_1}
$\partial X = \partial \Gamma_1(X)$
\end{prop}
\begin{proof} First of all, since $X$ embeds isometrically into $\Gamma_1(X)$, it is obvious that 
$\partial X \subseteq \partial \Gamma_1(X)$.

Let then $\{x_n\}$ be a sequence in $\Gamma_1(X)$ converging at infinity, representing the point
$\alpha \in \partial \Gamma_1$. If $x_n$ is a negligible vertex, it is at distance at most $2 \delta$
from an auxiliary vertex $x'_n$, so we can replace $x_n$ by $x'_n$ without changing the behavior of
$\{x_n\}$ at infinity. If for any $n$ there would exist $k_n \geqslant n$ such that $x_{k_n}$ is an
essential vertex, we would have that $\{x_{k_n}\}$ is a sequence of essential vertices converging 
at infinity, so that $\alpha \in \partial X$.

We can therefore assume that all $x_n$ are auxiliary vertices. Hence, there exist $x_n^\alpha$ and
$x_n^\omega$ such that $x_n \in [x_n^\alpha, x_n^\omega]$. Let $x'_n \in \{x_n^\alpha, x_n^\omega\}$
such that $(x_n \cdot x_n')$ is maximal. If for any $n$ there exists $k_n$ such that $(x_{k_n} \cdot
x_{k_n}') \geqslant n$, then $\{x_{k_n}'\}$ converges at infinity towards the same point as
$\{x_{k_n}\}$, so we have that $\alpha \in \partial X$.

Let us then assume there exists some $N$ such that $(x_n \cdot x_n^\alpha), (x_n \cdot x_n^\omega)
\leqslant N$ for any $n$. Let $x$ be the basepoint. We have
$$d(x_n^\alpha, x_n) = d(x, x_n^\alpha) + d(x, x_n) - 2(x_n \cdot x_n^\alpha) \geqslant d(x, x_n) +
d(x, x_n^\alpha) - 2N.$$
By the same reasoning we have
$$d(x_n^\omega, x_n) \geqslant d(x, x_n^\omega) + d(x,x_n) - 2N.$$
However, $x_n \in [x_n^\alpha, x_n^\omega]$, so we have that
$$d(x_n^\alpha, x_n^\omega) = d(x_n^\alpha, x_n) + d(x_n^\omega, x_n) \geqslant d(x, x_n^\alpha)
+ d(x, x_n^\omega) + 2d(x,x_n) - 4N$$
$$ \geqslant d(x_n^\alpha, x_n^\omega) + 2d(x, x_n) - 4N.$$
This implies that $d(x, x_n) \leqslant 2N$ for any $n$ and we have a contradiction with the assumption
that $\{x_n\}$ converges at infinity.

Then it follows that for any sequence $\{x_n\}$ of vertices in $\Gamma_1(X)$ which converges at
infinity, there exists a subsequence $\{x_{k_n}\}$ and a sequence $\{x_{k_n}'\}$ in $X$ such that
$(x_{k_n} \cdot x_{k_n}') \geqslant n$ for any $n$. We have therefore that $\partial \Gamma_1(X)
\subseteq \partial X$.
\end{proof}

\begin{corollary}
\label{co:boundaries_1}
$\partial X = \partial \Gamma_2(X)$
\end{corollary}
\begin{proof} By Lemma \ref{le:gamma_2_4}, $X$ is quasi-isometrically embedded in $\Gamma_2(X)$, so 
we have that $\partial X \subseteq \partial \Gamma_2(X)$. At the same time, from the proof of Lemma
\ref{pr:gamma_2_9} we have that $\Gamma_2(X)$ is quasi-isometrically embedded in $\Gamma_1(X)$, so
$\partial \Gamma_2(X) \subseteq \partial \Gamma_1(X) = \partial X$.
\end{proof}

\begin{corollary}
\label{co:boundaries_2}
$\partial X = \partial \beta(X)$
\end{corollary}
\begin{proof} It is again straightforward to use the same rationale as in the proof of Proposition 
\ref{pr:boundaries_1}. We leave the details to the reader.
\end{proof}

\bibliography{../main_bibliography}

\end{document}